 \documentclass[draft]{article}

\usepackage{amsmath,amsfonts,amsthm,amssymb,amscd,color}
\renewcommand{\Re}{\mathop{\rm Re}\nolimits}
\renewcommand{\Im}{\mathop{\rm Im}\nolimits}

\theoremstyle{plain} \newtheorem{theorem}{Theorem}[section]
\newtheorem{lemma}[theorem]{Lemma}
\newtheorem{proposition}[theorem]{Proposition}
 \theoremstyle{definition}
\newtheorem{definition}[theorem]{Definition} \theoremstyle{remark}
\newtheorem{remark}[theorem]{Remark} 
 
\newcommand{\R}{{\mathbb R}}

\newcommand{\Z}{{\mathbb Z}}

\newcommand{\N}{{\mathbb N}}

\newcommand{\resto}{{\mathcal R}} \def\im{{\rm i}}

\newcommand{\C}{\mathbb{C}}

\def\({\left(}
\def\){\right)}
\def\<{\left\langle}
\def\>{\right\rangle}

\numberwithin{equation}{section}

\begin{document}

  \title{ On   small energy stabilization     in the NLS with a trapping potential }

 \author {Scipio Cuccagna, Masaya Maeda}

 \maketitle
\begin{abstract}  We describe the asymptotic behavior of small energy solutions of an NLS with a trapping potential  generalizing work of Soffer and Weinstein, and of Tsai and Yau.  The novelty is that we allow generic  spectra associated to the potential.  This is yet a new application of the idea of interpreting the \textit{nonlinear Fermi Golden Rule}  as a consequence of the Hamiltonian structure. \end{abstract}

\section{Introduction}\label{sec:intr}
We consider  the initial value problem
\begin{equation}\label{eq:NLS}
\im u_t = H u + |u|^2  u,\ (t,x)\in \R^{1+3}, \quad u(0)=u_0
\end{equation}
where $H=-\Delta + V$.    For $f,g:\R^3\to \C$ we introduce the bilinear form
 \begin{equation}\label{eq:bilf}
    \langle f , g \rangle = \int _{\R ^3} f(x) g(x) dx.
 \end{equation}
We assume the following.
\begin{itemize}
  \item[(H1)]
$V\in\mathcal{S}(\R^3)$, where $\mathcal S (\R^3)$ is the space of Schwartz functions.

\item[(H2)]
$\sigma_p(H)=\{e_1<e_2< e_3 \cdots < e_n<0\}$.   Here we assume that all the eigenvalues have multiplicity 1.
 0 is neither an eigenvalue nor a resonance (that is, if $(-\Delta +  V)u=0$  with $u\in  C^\infty$ and     $|u(x)|\le C|x|^{-1} $ for a fixed $C$, then $u=0$).

\item[(H3)] There is an $N\in \N$ with $N > |e_1| (\min \{ e_i -e_j : i>j \}) ^{-1}$
s.t.
if  $\mu\in \Z^n$  satisfies
  $|\mu| \leq 4N +8$ and $\mathbf{e} := ({e} _{ 1},\dots , {e} _{n})$,  then we have
$$ \mu \cdot \mathbf{e} :=
\mu _1 {e} _{ 1} +\dots +\mu _n {e} _{n} =0 \iff \mu=0\ .
$$
\item[(H4)]  The following  Fermi Golden Rule (FGR)  holds: the expression \begin{equation*}       \sum _{
L \in \Lambda  }
    \langle \delta ( {H}-L)
   \overline{G} _L(\zeta ),   {G} _L(\zeta )\rangle
 ,
\end{equation*}
which is defined    in the course of the paper (for  $\Lambda \subset \R _+$   see \eqref{eq:FGR22} and  for  $ {G} _L$ see \eqref{eq:pos2})
and which is always nonnegative,
satisfies formula \eqref{eq:FGR}.

\end{itemize}

To each  $e_j$   we associate an eigenfunction $\phi _j$. We choose them s.t.  $\langle  \phi _j,\overline{\phi }_k\rangle  =\delta _{jk} $.
 Since we can,  we also   choose the $\phi _j$  to be all real valued.
To each  $\phi _j$ we associate   nonlinear bound states.

\begin{proposition}[Bound states]\label{prop:bddst}
Fix $j\in \{ 1,\cdots,n\}$. Then
$\exists a _0>0$ s.t.\ $\forall z \in     B_\C ( 0, a _0 )$, there is a unique  $  Q_{jz } \in \mathcal{S}(\R^3, \C ) := \cap _{t\ge 0}\Sigma _t (\R^3, \C )$  (where for the spaces $\Sigma _t  $ see  Sect. \ref{subsec:notation}), s.t.
\begin{equation}\label{eq:sp}
\begin{aligned}
&H Q_{jz } + |Q_{jz }|^2 Q_{jz }= E_{jz }Q_{jz } \quad , \quad
Q_{jz }=  z \phi_j + q_{j z}, \ \langle q_{j z},\overline{\phi}_j\rangle =0,
\end{aligned}
\end{equation}
and s.t. we have for any $r\in \N$:
 \begin{itemize}
\item[(1)]    $(q_{jz },E_{jz }) \in C^\infty ( B_\C ( 0, a _0 ), \Sigma _r\times \R )$;   we have $q_{jz } =  z \widehat{q}_{j  } (|z|^2)$ , with
$ \widehat{q}_{j  } (t^2 ) =t ^2\widetilde{q}_{j }(t^2)$,   $\widetilde{q}_{j } (t ) \in C^\infty (   ( - {a _0 }^{2}, {a _0 }^{2}), \Sigma _r (\R ^3, \R ) )$  and  $E_{jz }  =E_{j } (|z|^2)$ with $E_{j } (t ) \in C^\infty (   ( - {a _0 }^{2}, {a _0 }^{2}),   \R  )$;
\item[(2)]    $\exists$ $C >0$ s.t.
$\|q_{jz }\|_{\Sigma _r} \leq C |z|^3$, $|E_{jz }-e_j|<C | z|^2$.

\end{itemize}

\end{proposition}
For the  proof of Proposition \ref{prop:bddst}   see      Appendix \ref{app:bdstates}.

\begin{definition}\label{def:contsp}
Let    $b _0 >0$ be sufficiently small so that for $z_j\in B_{\C^{n }}(0,b _0) $, $Q_{j z_j}$ exists for all $j\in\{ 1,\cdots,n\}$.
For such $z_j$ and for $D_{jI}$ and $D_{jR}$  defined in Sect. \ref{subsec:notation}, we set
\begin{equation}\label{eq:contsp}
\begin{aligned}
\mathcal{H}_c[z]&=\mathcal{H}_c[ z_1\cdots,z_n]:=\left\{\eta\in L^2 :\  \Re \<\im \, \overline{\eta},D _{jR} Q_{j z_j}\>=\Re \<\im \, \overline{\eta} ,D_{jI}Q_{j z_j}\>=0 \text{ $\forall$ $j$}\right\}.
\end{aligned}
\end{equation}
In particular as an elementary consequence of
\eqref{eq:contsp} and Proposition \ref{prop:bddst}  we have
\begin{equation}\label{eq:contsp1}
\begin{aligned}
\mathcal{H}_c[0]  =\left\{\eta\in L^2 ;\    \<  \overline{\eta},\phi _j\>  =0\text{ for all $j$}\right\}.
\end{aligned}
\end{equation}
We denote by $P_c$ the orthogonal projection of $L^2$ onto $\mathcal{H}_c[0]$.
\end{definition}

A pair $(p,q)$   is
{\it admissible} when \begin{equation}\label{admissiblepair}
2/p+3/q= 3/2\,
 , \quad 6\ge q\ge 2\, , \quad
p\ge 2. \end{equation}
The following theorem is our main result.
\begin{theorem}\label{thm:small en}   Assume $(\mathrm{H1})$--$(\mathrm{H4})$.  Then there exist  $\epsilon _0 >0$ and $C>0$ such that  for $\epsilon =\| u (0)\| _{H^1}<\epsilon _0  $ the  solution  $u(t)$ of  \eqref{eq:NLS} can be written uniquely  for all times as
 \begin{equation}\label{eq:small en1}
\begin{aligned}&    u(t)=\sum_{j=1}^nQ_{j z_j(t)}+\eta (t) \text{ with $\eta (t) \in
\mathcal{H}_c[z(t)]$,}
\end{aligned}
\end{equation}
s.t.  there exist  a unique $j_0$, a
$\rho  _+\in [0,\infty )^n$ with $\rho_{+j}=0$ for $j\neq j_0$,
s.t.
  $| \rho  _+ | \le C  \| u (0)\| _{H^1}  $ and an $\eta _+\in H^1$
with $\|  \eta _+\| _{H^1}\le C  \| u (0)\| _{H^1},$
 s.t.
\begin{equation}\label{eq:small en3}
\begin{aligned}&     \lim_{t\to +\infty}\| \eta (t,x)-
e^{\im t\Delta }\eta  _+ (x)   \|_{H^1_x}=0   \quad  , \quad
 \lim_{t\to +\infty} |z_j(t)|  =\rho_{+j}  .
\end{aligned}
\end{equation}
 Furthermore we have $\eta = \widetilde{\eta} +A(t,x) $  s.t.
    for    all admissible pairs  $(p,q)$
 \begin{equation}\label{eq:small en2}
\begin{aligned}&     \| z \| _{L^\infty _t( \mathbb{R}_+ )}+ \|
\widetilde{ {\eta}}   \| _{L^p_t( \mathbb{R}_+,W^{1,q}_x)} \le C \| u (0)\| _{H^1}   \ , \\& \| \dot z _j +\im  e_{ j }z_j \|  _{L^\infty _t( \mathbb{R}_+ ) } \le C  \| u (0)\| _{H^1}^2\
\end{aligned}
\end{equation}
and s.t. $A(t,\cdot )\in \Sigma _2 $  for all $t\ge 0$ and
\begin{equation}\label{eq:small en4}
\begin{aligned}&      \lim_{t\to +\infty}\| A(t,\cdot )   \|_{\Sigma _2 }=0 .
\end{aligned}
\end{equation}

\end{theorem}

\noindent As an interesting corollary to Theorem \ref{thm:small en} we   show rather simply that the excited states are \textit{orbitally unstable}.
We recall that   $e^{-\im t E _{jz}}Q_{jz}$   is called \textit{orbitally stable}  in $H^1(\R ^3)$ for  \eqref{eq:NLS}
if
\begin{equation}\label{eq:orbstab}
\begin{aligned}& \text{ $\forall$ $\varepsilon >0$  $\exists$ $\delta >0$  s.t. }       \| u_0 - Q_{jz} \| _{H^1(\R ^3)}< \delta \Rightarrow \sup _{t\in \R} \inf _{\vartheta \in \R }\| u (t) - e^{\im \vartheta }  e^{-\im t E _{jz}} Q_{jz} \| _{H^1(\R ^3)}<\varepsilon
\end{aligned}
\end{equation}
and is  orbitally unstable if \eqref{eq:orbstab} does not hold. We prove
what follows.

\begin{theorem}\label{thm:orbstab}  Assume  $(\mathrm{H1})$--$(\mathrm{H4})$.  Then
 there exists  $\epsilon _0 >0$  such that if  $j\ge 2$ and for $|z|<\epsilon _0$
 the standing wave $e^{-\im t E _{jz}}Q_{jz}$ is  orbitally unstable.  Furthermore
 $e^{-\im t E _{1z}}Q_{1z}$  is orbitally stable.

\end{theorem}
 Notice that
\cite{TY2,TY3,TY4,SW4,GW1,GW2,GP,NPT} contain only very partial proofs of
the instability of the 2nd excited state.   Theorem \ref{thm:orbstab}  will be proved in Sect.\ref{sec:stab} and until then, and in particular
in the sequel of this introduction, we will focus only on Theorem \ref{thm:small en}.

We recall that \cite{GNT} proved Theorem \ref{thm:small en}, for $|u|^2u$ replaced by more general
functions, in the case when  $ H $ has one eigenvalue (for the NLS with an electromagnetic potential we refer to \cite{koo}). The case of two eigenvalues is discussed in the series \cite{TY1,TY2,TY3} and in  \cite{SW4}, under more stringent conditions on the initial data, which are such that
  $\| u_0 \| _{H ^{k,s}}$ is  small for $k>2$ and  some $s$ large enough  in   \cite{SW4}
  and    $\| u_0 \| _{H ^{1} \cap L ^{2,s}}$ small for $s> 3$ in   \cite{TY1,TY2,TY3}
  .
    A crucial restriction in these papers is that $2e_2>e_1$.    They then  prove versions of  {Theorem} \ref{thm:small en} involving also rates of decay of $|z(t)|$,  of $\| \eta (t)\| _{L^\infty (\R ^3)}$ and of  $\| \eta (t)\| _{ L ^{2,s} (\R ^3)}$ for appropriate $s>0$.

   The  ideas used in proofs such as in   \cite{TY1,TY2,TY3,SW4} appear very difficult to extend
   to operators with more than 2 eigenvalues, where only partial results like in
    \cite{NPT} are known, and for initial data small only in $H^1$.  On one hand, the Poincar\'e Dulac normal form argument in these papers seems not suited to discuss
   the higher   order  FGR needed  when $2e_2<e_1$.  Furthermore, in these papers there is a subdivision
    of the evolution in distinct phases, which the solution enters in a somewhat irreversible
    fashion and which are considered one by one.    This
   division in distinct phases might become unclear in cases when  $u(t)$   oscillates  from one phase to the other, as it is not unlikely to happen in the  $H^1$ case, or
   when the passage from one phase to the other is very slow, as is certainly true in the $H^1$ case. Moreover, an increase in the number of eigenvalues of $H$  increases also the number of distinct phases
   that need to be accounted for and the complexity of the argument.
     So, any hope of proving Theorem \ref{thm:small en} should rely on an argument which yields the asymptotics in a single stroke and which does not
   distinguish distinct cases. This is what we do, see for example in  the second part of Sect. 6.
     We did not check if our method yields the decay estimates
   of   \cite{TY1,TY2,TY3,SW4} under more stringent conditions on $u_0$.

In the present paper we give a yet new application  of the interpretation of the FGR in terms
of the Hamiltonian structure of the equation. This interpretation was first introduced in
\cite{CuInst}  and was then applied in  \cite{bambusicuccagna} to generalize the result
of \cite{SW3}. It was later applied to the problem of asymptotic stability of ground states
of the NLS, first not allowing translation symmetries  in  \cite{Cu2}, and then with
translation in  \cite{Cu3}, see also \cite{Cu0}.

The link between  FGR  and  Hamiltonian structure  rests in the fact that the latter
yields algebraic identities between coefficients of different coordinates in the system
(compare the r.h.s. in \eqref{4.9} with the  second line in \eqref{eq:FGR01}).
   These   allow to show that some other coefficients  in the equations of the $z_j$'s  have a square power structure and  have a
fixed  sign (in the case of the NLS), see Lemma \ref{lem:pos1}.   This then yields   decay of the  $z_j$'s, except  at most for one of the $j$'s  here.
We refer to pp. 287--288  in \cite{Cu2} for the original  intuition behind this approach to the FGR, which   views the FGR as a simple consequence of   Schwartz's Lemma on mixed derivatives,
and    which has made possible   papers such as
  \cite{CuInst,bambusicuccagna,Cu2,Cu3,Cu0}, as well as  others.
For other applications of this theory we refer to  the references in \cite{Cu0}, \cite{CM}. We refer  also to \cite{Cu4}, whose treatment of the FGR is similar to the
one in this paper.  Earlier  treatments of  FGR,  are in   \cite{TY1,TY2,TY3,SW4} and,
still earlier, in   \cite{BP2,SW3}, but  they seem to work only in relatively simple cases,
because they run into trouble if the normal form argument requires more than very few steps.
For more references and comments see  \cite{Cu2}.

As we will see below, the FGR can be  seen relatively easily after one finds    an appropriate  effective Hamiltonian
  in the right system of coordinates.  This coordinate system is obtained by a normal form
  argument. Right from the beginning though, it is crucial to choose
  the right ansatz and system of coordinates. For example, since $H$ has eigenvalues, it would seem  natural to split the NLS \eqref{eq:NLS} into  a system using the  coordinates of  the spectral decomposition of
$H$, see \eqref{eq:specH}. However this would not be a good choice for our nonlinear system.  Following   \cite{GNT}, it is better to pick  as coordinates the $z_j$'s of
Prop.\ref{prop:bddst}, complementing them    with an  appropriate continuous coordinate.
 There is the    natural  ansatz  \eqref{eq:decomposition1}  (the same used in \cite{SW4}) which, following  \cite{GNT},
 can be used to
 obtain the  continuous coordinate,
 here denoted $\eta$ and introduced in Lemma \ref{lem:systcoo}.

Once we have coordinates $ (z,\eta )$ with $z= (z_1,...,z_n )$, where $z_1$ is the ground state coordinate,
$z_j$ for $j>1$    the excited states coordinates and $\eta$   the radiation coordinate,
Theor. \ref{thm:small en}  can be loosely paraphrased as follows:
 \begin{equation} \label{eq:selec} \begin{aligned} &
     \text{$\eta (t)\to  0$ in $H^1 _{loc} $ and $z _j (t)\to 0$ except at most for one $j$.}\end{aligned}  \end{equation}
 In particular, if   $z (t)\to 0$   the solution $u(t)$ of   \eqref{eq:NLS} scatters
like a  solution of $\im \dot u=-\Delta u$ in $H^1$. Otherwise there is   one $j$ such that
$u(t)$ scatters to a $e^{\im \vartheta (t)}Q _{z_{+j}}$, with $\vartheta (t)$ a phase term which we do not control here. We have
  convergence by scattering  to a ground state if $j=1$, and to an excited state if $j>1$. The latter presumably occurs for the $u(t)$ whose trajectory is contained in an appropriate manifold, see \cite{TY4,Beceanu,GP}.

  It is not easy to see \eqref{eq:selec}      in the initial coordinate system. So we need
  a Birkhoff normal form  argument  to identify an   effective Hamiltonian, like in \cite{bambusicuccagna}.  Unlike \cite{bambusicuccagna} and  like in \cite{Cu2},   the initial coordinates, while quite natural from the point of view of  the NLS \eqref{eq:NLS},
  are not   Darboux coordinates
  for the natural symplectic   form  $\Omega $ in the problem, see \eqref{eq:Omega}. Hence before
  doing normal forms,
  we have first to implement the Darboux theorem to   diagonalize the problem (of course the coordinates arising
  from the spectral decomposition of $H$, see \eqref{eq:specH}, are Darboux coordinates,
  but as we wrote they are not suited for our nonlinear asymptotic analysis). So in this paper   we need to perform
  a number of coordinate changes: first a  Darboux Theorem and then normal
  form  analysis.  At the end of the process we get new coordinates
 $ (z_1,...,z_n,\eta )$   where the Hamiltonian is  sufficiently simple that we can prove
 \eqref{eq:selec}  relatively easily
 using the FGR  (which tells us that all $z_j$'s, except at most one, are damped) and a semilinear NLS for $\eta $ which shows scattering of $\eta$ because of linear
 dispersion.
  In the context of the theory developed in \cite{bambusicuccagna,Cu2} and other literature, the work in the last system of coordinates, that is all the material in Sect.\ref{sec:disp}, is rather routine.

Having proved
\eqref{eq:selec} for the   last system of coordinates $(z,\eta)$, the obvious  question
 is why  \eqref{eq:selec} should hold, as Theorem \ref{thm:small en} is saying,  also for the initial coordinates, which we now denote by $(z',\eta ')$, to distinguish them
 from the final coordinates  $(z,\eta )$.   Keeping in mind that
 all coordinate changes are small
 nonlinear perturbations of the identity,  the only simple reason why this might happen is that  different coordinates must be related in  the form
 \begin{equation} \label{eq:rel1}\begin{aligned} &   z'_1 = z_1 + O( z \eta ) + O(\eta ^2) + \sum _{i\neq j}O( z_i z_j )  , ...,  z'_n = z_n + O( z \eta ) + O(\eta ^2)  + \sum _{i\neq j}O( z_i z_j ), \\& \eta ' =\eta + O( z \eta ) + O(\eta ^2) + \sum _{i\neq j}O( z_i z_j ) .  \end{aligned}
\end{equation}
This relation between any two systems of coordinates  forbids relations  like $z_1'=z_1+z_2^2$ etc.
 Indeed, with the  latter relations it would not be true (except for the case $z(t)\to 0$)
that   \eqref{eq:selec} for
 $(z ,\eta  )$ implies   \eqref{eq:selec} for
 $(z',\eta ')$. So our main strategy is to prove   \eqref{eq:selec} for the final  $(z ,\eta  )$ with some relatively standard method using FGR and linear dispersion,
 and to be careful to implement only coordinate changes like in \eqref{eq:rel1}.
This latter point is the novel problem we need to face in this paper.  It is not obvious from the outset that  \eqref{eq:rel1} should hold.

As we wrote above, \cite{GNT} suggests a very natural choice of functions $z_j$, based
on  {Proposition} \ref{prop:bddst} which can be   completed in a
system of independent coordinates. Loosely speaking, the $z_j$'s have the problem that  they are defined somewhat independently to each other.
This shows up in the expansion of the Hamiltonian in Lemma \ref{lem:EnExp},  with a    certain lack of decoupling inside the energy between
distinct $  z_j$'s, see \eqref{eq:troub31} and Remark \ref{rem:crux1}. This leads in
\eqref{eq:enexp1} (see the 2nd line) to terms whose elimination in a normal form argument
would seem incompatible with coordinate changes satisfying \eqref{eq:rel1}. These bad terms
of the Energy
can be better seen in  \eqref{eq:enexp10}: they are the $l=0$ terms in the 2nd line.
Other additional bad terms   arise in the course of the Darboux theorem transformation.
Bad terms  in the differential form $\Gamma $  in \eqref{eq:alpha1} (used in the classical formula \eqref{eq:fdarboux}) are  those in $I_1$ in
\eqref{eq:bad1}. Specifically  they are   the first term in the r.h.s. of \eqref{eq:bad1}.
The   r.h.s. of \eqref{eq:Omegahat1} is also filled with bad terms in the sense that they yield   a coordinate change $\mathfrak{F}$ in Lemma \ref{lem:darflow0} leading to more
$l=0$ terms in the 2nd line  in  \eqref{eq:enexp10}. Specifically, they  originate from
 the pullback $\mathfrak{F} ^*\sum _{j=1}^{n}E (Q_{j z_j}) $ of the 1st term in
 the r.h.s. of \eqref{eq:enexp1} (more bad terms seem to arise if we use $\Omega _0'$, see
 \eqref{eq:defOm0pr} rather than the slightly more complicated   $\Omega _0$, see
 \eqref{eq:defOm0}, as local model of $\Omega$).  In a  somewhat  empirical fashion,
 for which we don't have   a simple conceptual reason, a plain
    and simple computation  shows  that all the bad terms cancel out
 and that there are no $l=0$ terms in  \eqref{eq:enexp10}. This is proved in the Cancelation Lemma \ref{lem:KExp2}, which is the main new ingredient  in the paper. This lemma proves that
 the change of coordinates designed to diagonalize $\Omega$,   is also decoupling the discrete
 coordinates inside the Hamiltonian.    From that point on,  the structure
 \eqref{eq:rel1} for the coordinate changes is automatic and the various steps of the proof
 of Theorem \ref{thm:small en}  are  similar to arguments such as \cite{Cu0,Cu4}  which have been repeated in a number of papers. So they are fairly standard,  even though
 we are  able to   discuss them only in a    rather technical way. We have to go into the details of the proof, rather than refer to the references,    because of some technical
 novelties required by the fact that in general  $z\not \to 0$, and what converges to 0
 is instead the vector $\mathbf{Z}$ introduced in Def.\ref{def:comb0}, whose components
 are  products of distinct components of $z$.

 In the second part of Sect. 6
 the FGR  and the asymptotics of the $z_j$'s in the final coordinate system are rather simple
 to see in a single stroke. Furthermore,
   Theorem \ref{thm:mainbounds} is more or less the same of \cite{Cu2,Cu4}.

  One limitation in our present paper is that we do not generate
examples of equations which satisfy Hypothesis (H4).  Notice though that our result, for solutions only in $H^1$,
is new even in the 2 eigenvalues case of  \cite{TY1,TY2,TY3,SW4} where our FGR  is the same.
 Still,
we believe    that (H4) holds
for generic $V$. And    even if it fails at one stage, this is not necessarily a problem:
the strict positive sign  in the FGR  is  only an obstruction at performing further
the normal form
argument, so if  there is a 0, in principle it is enough  to  proceed with some further
coordinate change until, after a finite number of steps, there will finally be
a positive sign in the   FGR, and so the stabilization will occur, just at a slower rate.
And if the FGR is always 0, then maybe this is because the NLS has a special structure, see
p.69 \cite{SW3} for some thoughts.

  Prop. 2.2 \cite{bambusicuccagna}
  proves    validity  in general of
the   FGR.  Transposing here that proof   would require replacing the cubic nonlinearity with
 a more general nonlinearity $\beta (|u|^2)u$. This   seems rather simple to do because the cubic power     is only used to simplify the discussion in Lemma \ref{lem:EnExp}. But it is not so clear  how to offset here the absence of  a meaningful mass term $m^2u$, which in  \cite{bambusicuccagna}
 pp. 1444--1445, by choosing $m$ generic, is used to move some    appropriate spheres in phase space. Adding   to the NLS  a term $m^2u$   would not    change the spheres here.

We  reiterate that  Proposition \ref{prop:bddst} is valid for small $z_j\in \C$. As $z_j$ increases there are interesting symmetry breaking bifurcation  phenomena, see \cite{kirr,KirrKevPel}
and therein and see also  \cite{FS,Grecchi,Sacchetti} and therein for
  the semiclassical
NLS. Notice that   Theorem \ref{thm:small en} should allow to prove asymptotic breakdown of the beating motion
in the case $\mu _\infty =0$ in   \cite{Grecchi}.
  \cite{Goodman,Marzuola} consider finite dimensional approximations
of the solutions at energies close to the   symmetry breaking  point of \cite{kirr}
and prove the long time existence of  interesting patterns for the full NLS.
Unfortunately, it is beyond the scope of our analysis, and it remains an interesting open problem, to understand the eventual asymptotic behavior
of the solutions in  \cite{Goodman,Marzuola}.

\section{Notation, coordinates and resonant sets} \label{section:set up}

\subsection{Notation}
\label{subsec:notation}
\begin{itemize}
\item We denote by $\N =\{ 1,2,...\}$  the set of natural numbers  and  set $\N_0= \N\cup \{ 0\}  $.
\item
We denote $z=( z_1,\dots ,z_n)$, $|z|:=\sqrt{\sum_{j=1}^n|z_j|^2}$.

\item Given a Banach space $X$, $v\in X$ and $\delta>0$ we set
$
B_X(v,\delta):=\{ x\in X\ |\ \|v-x\|_X<\delta\}.
$

\item
Let $A$ be an operator on $L^2(\R^3)$.
Then $\sigma_p(A)\subset \C$ is the set of eignvalues of $A$ and $\sigma_e(A)\subset \C$ is the   essential spectrum of $A$.

\item
For $\mathbb{K}=\R , \C $
we     denote
by $\Sigma _r=\Sigma _r(\R ^3,  \mathbb{K}) $ for $r\in \N\cup \{ 0\}$  the Banach spaces  defined by  the completion  of $ C_c (\R ^3,  \mathbb{K})$  by the norms
\begin{equation}
\begin{aligned} &
     \| u \| _{\Sigma _r} ^2:=\sum _{|\alpha |\le r}  (\| x^\alpha  u \| _{L^2(\R ^3   )} ^2  + \| \partial _x^\alpha  u \| _{L^2(\R ^3, \mathbb{K}   )} ^2 ).
\end{aligned}\nonumber
\end{equation}
For $m<0$ we  consider the  topological dual $\Sigma _m=(\Sigma  _{-m})'$.  Notice, see \cite{Cu3}, that the  spaces   $\Sigma _r $
can be equivalently defined using   for $r\in \R$    the norm
 $
      \| u \| _{\Sigma _r}  :=  \|  ( 1-\Delta +|x|^2)   ^{\frac{r}{2}} u \| _{L^2}        .   $

\item
$\mathcal{S}(\R^3)= \cap _{m\ge 0}\Sigma _m$ is the space of  Schwartz functions; $\mathcal{S}'(\R^3)= \cup _{m\le 0}\Sigma _m$ is the space of  tempered distributions.

\item We set $z_j = z_{j R}+ \im z_{j I}$ for $z_{j R}, z_{j I}\in \R$.

\item
For $f:\C^{n }\to \C$ set
$D_{jR} f(z):=\frac{\partial}{\partial z_{j R}}f(z)$, $D_{jI}f(z):=\frac{\partial}{\partial z_{j I}}f(z)$.

\item We set $\partial _{ l}:=\partial _{z_l}$  and  $\partial _{\overline{l} }:=\partial _{\overline{z}_l}$.
Here as customary   $\partial _{z_l}  = \frac 12 (D _{lR}-\im  D _{lI} )$ and    $\partial _{\overline{z}_l}  = \frac 12 (D _{lR}+\im  D _{lI} )$.

\item Occasionally we use a single  index $\ell =j, \overline{j}$. To define $\overline{\ell}$  we use the convention $\overline{\overline{j}}=j$.
We will also write   $z_{\overline{j}}=\overline{z}_j$.

\item  We will consider vectors $z=(z_1, ..., z_n)\in \C^n$ and  for  vectors $\mu , \nu \in (\N \cup \{ 0 \} ) ^{n}$ we set $ z^\mu \overline{z}^\nu  := z_1^{\mu _1} ...z_n^{\mu _n}\overline{z}_1^{\nu _1} ...\overline{z}_n^{\nu _n}$.  We will set $|\mu |=\sum _j \mu _j$.

\item   We have  $dz_j  =dz_{jR}+\im dz_{jI}$,    $d\overline{z}_j  =dz_{jR}-\im dz_{jI}$.

%

\item We consider the vector $ \mathbf{e}=(e_1,...., e_n)$ whose entries are the eigenvalues of $H$.

\item  $P_c$ is the orthogonal projection of $L^2$ onto $\mathcal{H}_c[0]$.

\item  Given two Banach spaces $X$ and $Y$ we denote by $B(X,Y)$ the   space
of bounded linear operators $X\to Y$ with the norm of the uniform operator topology.

\end{itemize}

\subsection{Coordinates}
\label{subsec:coord}

The first thing we need is an ansatz. This is provided by the following lemma.

\begin{lemma}\label{lem:decomposition}
There exists $c _0 >0$ s.t.  there exists a $C>0$ s.t. for all $u \in H^1$   with $\|u\|_{H^1}<c  _0 $, there exists a unique  pair $(z,\Theta )\in   \C^{n }\times  ( H^1 \cap \mathcal{H}_{c}[z])$
s.t.
\begin{equation}\label{eq:decomposition1}
\begin{aligned} &
u=\sum_{j=1}^nQ_{j z_j}+\Theta   \text{ with }  |z |+\|\Theta \|_{H^1}\le C \|u\|_{H^1} .
\end{aligned}
\end{equation}
 Finally,
 the map  $u \to (z,\Theta )$  is $C^\infty (B _{H^1}(0, c _0 ),
 \C ^n \times  H^1 )$   and  satisfies  the   gauge property  \begin{equation}\label{eq:decomposition3}
\begin{aligned} &
 \text{$z(e^{\im \vartheta} u)=e^{\im \vartheta} z( u)
$ and    $\Theta (e^{\im \vartheta} u)=e^{\im \vartheta} \Theta ( u)$ } .
\end{aligned}
\end{equation}

\end{lemma}
\proof
We consider  the functions
\begin{equation*}
\begin{aligned}&
  F _{jA} ( u, z):=\Re \langle    u- \sum_{l=1}^nQ_{l z_l} ,\im \, \overline{D _{jA} Q_{j z_j}}\rangle   \text{ for $A=R,I$}  .
\end{aligned}
\end{equation*}
We have  $  F _{jR} ( 0, 0) =F _{jI} ( 0, 0)=0$.  These functions are smooth in
$L^2\times B_{\C ^n} (0,  b _0)$  for   the $b _0$
in Def. \ref{def:contsp}.
 We have  $F _{jR} ( 0, z) =\Im z_j+O(z^3)$  and  $F _{jI} ( 0, z) =\Re z_j+O(z^3)$
by Proposition \ref{prop:bddst}. By the   implicit function theorem
there is a map $u\to z$ which is  $C^\infty ( B _{L^2}(0, c _0) ,
 \C ^n   )$ for a $c _0>0$ sufficiently small.   Set $\Theta := u-\sum_{j=1}^nQ_{j z_j}$. Then $\Theta
\in C^\infty ( B _{H^1}(0, c  _0)  ,
    H^1  )$.   The inequalities follow from $|z (u)|\le C \| u\|_{H^1}$ which follows from $z\in C^1$ and $z(0)=0$.
Formula    \eqref{eq:decomposition3}  follows  from
\begin{equation*}
\begin{aligned} &
e^{\im \vartheta}u=\sum_{j=1}^n e^{\im \vartheta}Q_{j z_j}+e^{\im \vartheta}\Theta   = \sum_{j=1}^n Q_{j e^{\im \vartheta}z_j}+e^{\im \vartheta}\Theta
\end{aligned}
\end{equation*}
and from the fact that  $\Theta \in \mathcal{H}_{c}[z]$  implies  $e^{\im \vartheta}\Theta \in \mathcal{H}_{c}[  z' ]$  where $z'=e^{\im \vartheta}z$.  This last fact is elementary.  Indeed, setting only for this proof $z_j=x_j+\im y_j$ and    $z_j'=x_j'+\im y_j'$, we have
\begin{equation*}
\begin{aligned}
 \Re \<\im    \overline{e^{\im \vartheta} \Theta }, \partial _{x'_j} Q_{j z_j'}\>  &= \partial _{x'_j}x_j
\Re \<\im    \overline{e^{\im \vartheta} \Theta }, e^{\im \vartheta} \partial _{x _j} Q_{j z_j }\>   +
\partial _{x'_j}y_j  \Re \<\im    \overline{e^{\im \vartheta} \Theta }, e^{\im \vartheta}\partial _{y_j} Q_{j z_j }\>  =0
\end{aligned}
\end{equation*}
if    $\Theta  \in \mathcal{H}_{c}[z]$. Similarly,  $\Re \<\im    \overline{e^{\im \vartheta} \Theta }, \partial _{y'_j} Q_{j z_j'}\>=0$.
Hence    $\Theta  \in \mathcal{H}_{c}[z]$  implies  $e^{\im \vartheta}\Theta  \in \mathcal{H}_{c}[  e^{\im \vartheta}z ]$.

\qed

\begin{definition}\label{def:comb0} Given $z\in \C^n$, we denote by $ \widehat{Z} $ the vector with entries $(  z_{i}  \overline{{z}}_j)$ with $i,j\in [1,n]$   ordered in lexicographic order. We denote by $\textbf{Z}$ the vector     with entries $(  z_{i}  \overline{{z}}_j)$ with $i,j\in [1,n]$   ordered in lexicographic order      but only with pairs of  indexes  with $i\neq j$.  Here  $\textbf{Z}\in   L$ with $L$ the subspace of $ \C ^{n_0} =\{  (a_{i,j} ) _{i, j =1, ..., n} : i\neq j   \}   $ where $n_0=n (n-1),$ with $(a_{i,j} ) \in  L $ iff $a_{i,j} = \overline{a}_{j,i}$ for all $i,j$.
 For  a multi index $\textbf{m}=\{  {m}_{ij}\in \N _0 :i\neq j  \}$ we set $\textbf{Z}^{\textbf{m}}=\prod  (z_{i}  \overline{{z}}_j) ^{ {m}_{ij}} $  and $ |\mathbf{m}| :=\sum _{i,j}m  _{ij} $.

\end{definition}

We need a system of independent coordinates, which
 the $(z,\Theta  )$
in   \eqref{eq:decomposition1}  are not. The following lemma is used to complete the $z$ with a continuous coordinate.

\begin{lemma} \label{lem:contcoo}
There exists
$d _0>0$ such that for any $z\in\C$ with $|z|<d _0$  there exists a $\R$--linear operator $R[z]:\mathcal{H}[0]\to \mathcal{H}_c[z]$ such that  $ \left. P_c\right|_{\mathcal{H}_c[z]}=R[z]^{-1}$,
with $P_c$ the orthogonal projection of $L^2$ onto $\mathcal{H}_c[0]$, see   Def. \ref{def:contsp}.  Furthermore, for $|z|<d_0$ and $\eta \in  \mathcal{H}_c[0] $, we have the following properties.

\begin{itemize}
\item[(1)]
$R[z]\in C^\infty (B _{\C^n}  (0, d _0 ), B  (
H^1,H^1 ) )$, with $B  (
H^1,H^1 )$ the Banach space of $\R$--linear bounded operators from $H^1$ into itself.

\item[(2)]  For any $r>0$, we have $\|(R[z]-1)\eta \|_{\Sigma _r}\le c_r |z	 |^2 \|\eta \|_{\Sigma _{-r}}$  for a fixed $c_r$.

\item[(3)]   We have the covariance property  $R[e^{\im \vartheta }z]= e^{\im \vartheta }R[z] e^{-\im \vartheta }.$

\item[(4)]    We have, summing on repeated indexes,
\begin{equation} \label{eq:contcoo21}
 R[z]\eta =\eta  +   (\alpha_j[z]\eta )\phi_j  \text{  with }    \alpha_j[z]\eta =\langle B_j (z), \eta \rangle + \langle C_j(z), \overline{\eta }\rangle
\end{equation}
where  $B_j (z)=\widehat{B}_j (\widehat{Z} )$ and $C_j(z) = z_{i}  {z}_\ell  \widehat{C}_{i\ell j}(\widehat{Z})$, for $\widehat{B}$  and $\widehat{C}_{i\ell j}$  smooth  and  the $\widehat{Z}$
 of Def. \ref{def:comb0}.

\item[(5)] We have for $r\in \R$ with $\textbf{Z}$   as in Def. \ref{def:comb0}\begin{equation} \label{eq:contcoo2}
\begin{aligned}  &   \| {B} _j  (z) + \partial _{\overline{z}_j}  \overline{q}_{jz_j} \|_{\Sigma _r}  + \| C _j  (z)-\partial _{ \overline{z}_j}   {q}_{jz_j} \|_{\Sigma _r}   \le c_r |\textbf{Z}	|^2.
\end{aligned}
\end{equation}

\end{itemize}
\end{lemma}

\begin{proof}
Summing over repeated indexes, we search for a map $R[z]: L^2\to \mathcal{H}_c[z]$ of the form
\begin{equation*}
 R[z]f=f +   (\alpha_j[z]f)\phi_j  \text{  with }    \alpha_j[z]f=\langle B_j'(z), f\rangle + \langle C_j(z), \overline{f}\rangle
\end{equation*}
such that   $R[z]f \in \mathcal{H}_c[z]$  $\forall$$f\in L^2$. The latter condition    can be expressed as
\begin{equation*}
\begin{aligned}
  \Re \<  \overline{f}          ,\im D _{lA} Q_{lz_l} +      \langle \phi _j ,\im D _{lA} Q_{lz_l}\rangle   \overline{B}'_j
		- \langle \phi _j ,\im D _{lA} \overline{Q}_{lz_l}\rangle   C_j    \>  =0  \text{  for all $f\in L^2$}.
\end{aligned}
\end{equation*}
This and the following equalities
\begin{equation*}
\begin{aligned}  &    \langle \phi _j ,\im D _{lR} Q_{lz_l}\rangle  = \im \delta _{jl}+\langle \phi _j ,\im D _{lR} q_{lz_l}\rangle \ ,\quad
\langle \phi _j ,\im D _{lI} Q_{lz_l}\rangle  = - \delta _{jl}+\langle \phi _j ,\im D _{lI} q_{lz_l}\rangle ,
\\&   \langle \phi _j ,\im D _{lR} \overline{Q}_{lz_l}\rangle             =   \im \delta _{jl}+\langle \phi _j ,\im D _{lR} \overline{q}_{lz_l}\rangle \ ,  \quad   \langle \phi _j ,\im D _{lI} \overline{Q}_{lz_l}\rangle             =     \delta _{jl}+\langle \phi _j ,\im D _{lI} \overline{q}_{lz_l}\rangle  ,
\end{aligned}
\end{equation*}
yield the equalities
\begin{equation*}
\begin{aligned}  &
  D _{lR} Q_{lz_l} +       (\delta _{jl}+\langle \phi _j ,  D _{lR} q_{lz_l}\rangle)   \overline{B}'_j
		- ( \delta _{jl}+\langle \phi _j ,  D _{lR} \overline{q}_{lz_l}\rangle )   C_j =0 , \\&
\im D _{lI} Q_{lz_l} +       (-\delta _{jl}+\im \langle \phi _j ,   D _{lI} q_{lz_l}\rangle)   \overline{B}'_j
		- ( \delta _{jl}+\im \langle \phi _j ,  D _{lI} \overline{q}_{lz_l}\rangle )   C_j =0.
\end{aligned}
\end{equation*}
They can be rewritten  as
\begin{equation} \label{eq:contcoo3}
\begin{aligned}  &
\phi _l + \partial _{ l}  q_{lz_l} +       ( \delta _{jl}+\im \langle \phi _j ,  \partial _{ l} q_{lz_l}\rangle)   \overline{B}'_j
		-   \langle \phi _j ,  \partial _{ l}  \overline{q}_{lz_l}\rangle    C_j =0,\\&
 \partial _{\overline{l} }  q_{lz_l} +       \langle \phi _j ,  \partial _{\overline{l}}  q_{lz_l}\rangle    \overline{B}'_j
		- ( \delta _{jl}+\langle \phi _j , \partial _{\overline{l} }  \overline{q}_{lz_l}\rangle )   C_j =0 .
\end{aligned}
\end{equation}
For $z^2=\{  z_j ^2 \delta _{ij}\} $ and    $\overline{z}^2=\{  \overline{z}_j ^2 \delta _{ij}\} $
two $n\times n$ matrices,
the solution of this  system is of  the form
\begin{equation} \label{eq:contcoo31}
\begin{aligned}  &
 \begin{pmatrix}  \overline{B}'  \\  C \end{pmatrix}  =\sum _{m=0}^\infty   (-1)^m \begin{pmatrix}  \textbf{A}_1 & \overline{z}^2\textbf{A}_2   \\  z^2\textbf{A}_3  & \textbf{A}_4 \end{pmatrix}   ^m
 \begin{pmatrix}  u _1 \\  z^2u _2 \end{pmatrix}
\end{aligned}
\end{equation}
where $\textbf{A}_l=\textbf{A}_l (|z_1|^2, ...,|z_n|^2)$  are $n\times n$ matrices  and    $u_l=u_l (|z_1|^2, ...,|z_n|^2) $ are  $n\times 1$ matrices
 for $l=1$ (resp. $l=2$) with entries  $\phi _j + \partial _{ j} q_{jz_j}$
(resp. $\partial _{\overline{ j}}  q_{jz_j}$) as $j=1,...,n$.  This yields the structure
$\overline{B}' (z)= \widehat{{B}' }(\widehat{Z} )$ and $C _j(z) = z_{i}  {z}_\ell  \widehat{C}_{i\ell j}(\widehat{Z})$.

\noindent Using $\langle \phi _j ,    q_{jz_j}\rangle =0, $ we can rewrite
\eqref{eq:contcoo3}  in the form
\begin{equation}
\begin{aligned}  \label{eq:contcoo4}&
\overline{B}'_l = -\phi _l - \partial _{ l}  q_{lz_l} - \sum _{j\neq l} (\im \langle \phi _j ,  \partial _{ l} q_{lz_l}\rangle \overline{B}'_j-   \langle \phi _j ,  \partial _{ l}  \overline{q}_{lz_l}\rangle    C_j ),\\&
C_l =
\partial _{\overline{l} }  q_{lz_l} +  \sum _{j\neq l} (     \langle \phi _j ,  \partial _{\overline{l} }  q_{lz_l}\rangle    \overline{B}'_j
		- \langle \phi _j , \partial _{\overline{l} }  \overline{q}_{lz_l}\rangle )   C_j   .
\end{aligned}
\end{equation}
  By Proposition \ref{prop:bddst}   this implies
\begin{equation}   \label{eq:contcoo400}
\begin{aligned}  &   \| \overline{B}'_l +\phi _l  \|_{\Sigma _r}   +\| C_l \|_{\Sigma _r}    \le C |z_l|^2 .
\end{aligned}
\end{equation}
Reiterating this estimate, from \eqref{eq:contcoo4}  and for $B_l$ defined by the following formula, we get
\begin{equation*}
\begin{aligned}  &   \|   \overbrace{\overline{B}'_l +\phi _l -\sum _{j\neq l}  \im \langle \phi _j ,  \partial _{ l} q_{lz_l}\rangle \phi _{ j}  }^{\overline{B} _l}+ \partial _{ l}  q_{lz_l}   \|_{\Sigma _r}  \le  C |\mathbf Z|^2\\&
 \| C_l -
\partial _{\overline{l} }  q_{lz_l}  \|_{\Sigma _r}    \le     C |\mathbf Z|^2 .
\end{aligned}
\end{equation*}
 This yields \eqref{eq:contcoo2}.  Claim (3) follows by
\begin{equation}\label{eq:gau1}
\begin{aligned}
  \alpha _j [e^{\im \vartheta }z]\eta = e^{\im \vartheta } \alpha _j [z] e^{-\im \vartheta }\eta ,
\end{aligned}
\end{equation}
which in turn follows by claim (4). Indeed
\begin{equation*}
\begin{aligned} &
  \alpha _j [e^{\im \vartheta }z]\eta = \langle \widehat{B}_j (\widehat{Z} ), \eta \rangle + \langle  e^{2\im \vartheta } z_{i}  {z}_\ell  \widehat{C}_{i\ell j}(\widehat{Z}), \overline{\eta }\rangle   \\ =  &   e^{\im \vartheta }\langle \widehat{B}_j (\widehat{Z} ), e^{-\im \vartheta }\eta \rangle + e^{\im \vartheta }\langle    z_{i}  {z}_\ell  \widehat{C}_{i\ell j}(\widehat{Z}), \overline{e^{-\im \vartheta }\eta }\rangle
	=e^{\im \vartheta } \alpha _j [z] e^{-\im \vartheta }\eta .
\end{aligned}
\end{equation*}

\end{proof}

We are now able to define a system of coordinates near the origin  in  $L^2$.

\begin{lemma} \label{lem:systcoo}  For the $d  _0>0$ of Lemma \ref{lem:contcoo}
the map $(z,\eta )\to u$  defined,
by
\begin{equation} \label{eq:systcoo1}
 u=\sum_{j=1}^nQ_{j z_j}+R[z] \eta  \text{ for $(z,\eta )\in B_{\C^n}(0 , d _0) \times ( H^1\cap \mathcal{H}_c[0]) $}
\end{equation}
is  with values in   $H^1$  and  is   $C^\infty$.
Furthermore, there is a $d_1>0$ such that    for $(z,\eta )\in B_{\C^n}(0 , d _1) \times  (B_{H^1}(0 , d _1)\cap \mathcal{H}_c[0])$
the above map is a  diffeomorphism and
\begin{equation} \label{eq:coo11}
  |z|+\| \eta \| _{H^1} \sim   \| u \| _{H^1}.
\end{equation}

\noindent
	Finally, we have the gauge properties  $u ( e^{\im \vartheta } z, e^{\im \vartheta } \eta )= e^{\im \vartheta } u (z,\eta ) $
	and    \begin{equation}\label{eq:detion31}
\begin{aligned} &
 \text{$z(e^{\im \vartheta} u)=e^{\im \vartheta} z( u)
$ and    $\eta (e^{\im \vartheta} u)=e^{\im \vartheta} \eta ( u)$ } .
\end{aligned}
\end{equation}

\end{lemma}
\proof  The smoothness follows from the smoothness in $z$ in Proposition \ref{prop:bddst} and Lemma \ref{lem:contcoo}. Property   $u ( e^{\im \vartheta } z, e^{\im \vartheta } \eta )= e^{\im \vartheta } u (z,\eta ) $
and its equivalent  formula   \eqref{eq:detion31}
follow  from \eqref{eq:decomposition3} and claim (3) in Lemma \ref{lem:contcoo}.
Notice that $u  = u (z,\eta ) $ is the inverse of the smooth map
$u\to (z,\Theta )\to (z, P_c \Theta )$.  Formula \eqref{eq:coo11} follows by the
estimates in Prop. \ref{prop:bddst} and by claim (2) in Lemma \ref{lem:contcoo}.

\qed

\subsection{Resonant sets}

\begin{definition}\label{def:setM}   Consider the set of multiindexes  $\textbf{m}$
as in Def.  \ref{def:comb0} and for any $k\in \{ 1, ..., n  \}$ the set
\begin{equation} \label{eq:setM0} \begin{aligned} &  \mathcal{M}_k (r)=
\{   \textbf{m} :   \sum _{i=1} ^{n} \sum _{j=1} ^{n} m  _{ij} (e_i -e_j) - e_k< 0   \text{ and  }  |\mathbf{m}|   \le r \}  \\& \mathcal{M}_0 (r)=
\{   \textbf{m} :   \sum _{i=1} ^{n} \sum _{j=1} ^{n} m  _{ij} (e_i -e_j) =0   \text{ and  }  |\mathbf{m}|   \le r \}  .
\end{aligned}
\end{equation}
Set now
\begin{equation} \label{eq:setM1} \begin{aligned} &   {M}_k(r)=
\{    (\mu , \nu )\in \N _0 ^{n}  \times  \N _0 ^{n} :   \exists \textbf{m}   \in \mathcal{M}_k (r) \text{ s.t. }   z^{\mu} \overline{z}^{\nu} = \overline{z}_k   \mathbf{Z} ^{\textbf{m} }  \}  , \\& M(r)= \cup  _{k=1} ^{n} {M}_k(r) \quad  \text{  and $M= M (2N+4)$}
\end{aligned}
\end{equation}
 \end{definition}

\begin{lemma} \label{lem:M0} Assuming (H3) we have the following facts.

 \begin{itemize}
\item[(1)]
 For $\mathbf{Z}^{\mathbf{m}}=z^{\mu}\overline{z}^{\nu}$, then $\mathbf{m}\in \mathcal{M}_0 (2N+4)$ implies  $\mu =\nu$. In particular  $\mathbf{m}\in \mathcal{M}_0 (2N+4)$ implies $\mathbf{Z}^{\mathbf{m}}=|z_1| ^{2l_1}...|z_n| ^{2l_n}$ for some $(l_1,..., l_n)\in \N _0^n$.

 \item[(2)] For $|\textbf{m}  |  \le  2N+3$  and any $j$ we   have $\sum _{a,b}(e_a-e_b)  {m}_{ab} -e_j \neq 0$.

 \end{itemize}
\end{lemma}
\proof First of all, if  $\mu =\nu$ then $z^{\mu}\overline{z}^{\nu}=|z_1| ^{2\mu _1}...|z_n| ^{2\mu _n} $. So the first sentence in claim (1)  implies the second sentence in claim (1).
  We have
\begin{equation*}     \begin{aligned} &  \mathbf{Z}^{\mathbf{m}} = \prod _{i,l=1} ^{n}(z_i\overline{z}_l)^{m_{il}}= \prod _{i =1}^{n} z_i^{\sum _{l=1}^{n}m_{il} } \overline{z}_i^{\sum _{l=1}^{n}m_{li} } =  z^{\mu}\overline{z}^{\nu}.
\end{aligned}\end{equation*}
The pair $(\mu , \nu)$ satisfies $|\mu |=|\nu|\le 2N+4$ by
\begin{equation*}     \begin{aligned} &
 |\mu |=\sum _l \mu _l=\sum _{i,l} m_{il}  = |\nu| .
\end{aligned}\end{equation*}
We have $(\mu -\nu )\cdot \mathbf{e}=0$ by $\mathbf{m}\in \mathcal{M}_0 (2N+4)$ and
\begin{equation*}     \begin{aligned} &  \sum _i \mu _i e_i -\sum _l \nu _l e_l =  \sum _{i,l} m_{il}(e_i - e_l) = 0   .
\end{aligned}\end{equation*}
We conclude by (H3) that $\mu -\nu =0$.   This proves the 1st sentence of  claim (1).

\noindent The proof of claim (2)  is similar. Set
\begin{equation*}     \begin{aligned} &  \mathbf{Z}^{\mathbf{m}}\overline{z}_j = \prod _{i,l=1} ^{n}(z_i\overline{z}_l)^{m_{il}} \overline{z}_j = \prod _{i =1}^{n} z_i^{\sum _{l=1}^{n}m_{il} } \overline{z}_i^{\sum _{l=1}^{n}m_{li} } \overline{z}_j =  z^{\mu}\overline{z}^{\nu}
\end{aligned}\end{equation*}
  We have
\begin{equation*}     \begin{aligned} & (\mu - \nu )\cdot \mathbf{e}=  \sum _i \mu _i e_i -\sum _l \nu _l e_l  =  \sum _{i,l} m_{il}(e_i - e_l)-e_j    .
\end{aligned}\end{equation*}
We have
\begin{equation}  \label{eq:in}   \begin{aligned} &
 |\mu |=\sum _l \mu _l=\sum _{i,l} m_{il}  = |\nu|-1 .
\end{aligned}\end{equation}
If  $(\mu - \nu )\cdot \mathbf{e}=0$  then by $|\mu -\nu |\le 4N+5$
and by (H3) we would have $\mu =\nu$,   impossible by \eqref{eq:in}.

\qed

\begin{lemma} \label{lem:comb1} We have the  following facts.

\begin{itemize}
\item[(1)]
Consider   $\mathbf{m}=( m  _{ij}) \in \N _0 ^{n_0}  $ s.t.
$\sum _{i<j}m  _{ij} >N$ for   $N > |e_1| (\min \{ e_j -e_i : j>i \}) ^{-1}$, see (H3). Then for any eigenvalue $e_k$  we have

\begin{equation}\label{eq:comb3}
\begin{aligned}
& \sum _{i<j}  m  _{ij} (e_i -e_j) - e_k< 0.
\end{aligned}
\end{equation}

\item[(2)] Consider $\mathbf{m}\in \N _0 ^{n_0} $  with  $|\mathbf{m}|   \ge 2N+3$  and the monomial $ z _j \mathbf{Z}^{\mathbf{m}}$.
    Then $\exists$  $\mathbf{a}, \mathbf{b}\in \N _0 ^{n_0} $ s.t.
\begin{equation}\label{eq:comb4}
\begin{aligned}
& \sum _{i<j}  a  _{ij}    =N+1=\sum _{i<j}  b  _{ij} ,\\&
  a  _{ij}    =  b  _{ij}  =0  \text{ for all }   i>j \text{ and }
   a  _{ij}   +  b  _{ij} \le m  _{ij} + m  _{ji}   \text{ for all } (i,j)
\end{aligned}
\end{equation}
   and moreover there are   two indexes $ (k,l)$ s.t.
 \begin{equation}\label{eq:comb5}
\begin{aligned}
& \sum _{i<j}  a  _{ij} (e_i -e_j) - e_k< 0   \text{ and }    \sum _{i<j}  b  _{ij} (e_i -e_j) - e_l< 0
\end{aligned}
\end{equation}
 and such that  for $|z|\le 1$
  \begin{equation}\label{eq:comb6}
\begin{aligned}
& |z _j \mathbf{Z}^{\mathbf{m}}|\le |z_j| \  |z _k \mathbf{Z}^{\mathbf{a}}|  \  |z _l \mathbf{Z}^{\mathbf{b}}| .
\end{aligned}
\end{equation}

\item[(3)] For $\mathbf{m}$ with $|\mathbf{m}|   \ge 2N+3$  there exist $(k,l)$, $\mathbf{a} \in \mathcal{M}_k$ and $\mathbf{b} \in \mathcal{M}_l$ s.t.  \eqref{eq:comb6}  holds.

    \end{itemize}
\end{lemma}
\proof \eqref{eq:comb3} follows immediately from
\begin{equation*}
\begin{aligned}
& \sum _{i<j}  m  _{ij} (e_i -e_j) - e_k\le -\min \{ e_j -e_i : j>i \} N -e_1 <0,
\end{aligned}
\end{equation*}
where the latter inequality follows by the definition of $N$.

Given  $\mathbf{a}, \mathbf{b}\in \N _0 ^{n_0} $ satisfying  \eqref{eq:comb4},
by claim (1) they satisfy   \eqref{eq:comb5}  for any pair of indexes  $ (k,l)$.  Consider now the monomial  $ z _j \mathbf{Z}^{\mathbf{m}}$.   Since  $|\mathbf{m}|\ge 2N+3$, there are vectors
$\mathbf{c}, \mathbf{d}\in \N _0 ^{n_0} $ s.t.  $|\mathbf{c}|=|\mathbf{d}| = N+1$ with $c  _{ij}   +  d  _{ij} \le m  _{ij}  $   for all  $(i,j)$. Furthermore
we have
 \begin{equation}\label{eq:comb7}
\begin{aligned}
& z _j \mathbf{Z}^{\mathbf{m}} = z _j   z^\mu \overline{z} ^\nu \mathbf{Z}^{\mathbf{c}}  \mathbf{Z}^{\mathbf{d}} \text{ with $|\mu | >0 $ and  $|\nu | >0 $.}
\end{aligned}
\end{equation}
So, for $z_k$ a factor of $z^\mu $ and  $\overline{z}_l$ a factor of $\overline{z}^\nu $,
and for
\begin{equation}\label{eq:comb8}
\begin{aligned}
&    a_{ij}  =\left\{\begin{matrix}
          c_{ij}+c_{ji}  \text{ for $i<j$} \,\\
  0     \text{ for $i>j$}
\end{matrix}\right.  ,  \qquad    b_{ij}  =\left\{\begin{matrix}
          d_{ij}+d_{ji}  \text{ for $i<j$} \,\\
  0     \text{ for $i>j$}
\end{matrix}\right.
\end{aligned}
\end{equation}
for $|z|\le 1$ we have from \eqref{eq:comb7}
\begin{equation*}
\begin{aligned}
& |z _j \mathbf{Z}^{\mathbf{m}}| \le | z _j  | \ |  z_k   \mathbf{Z}^{\mathbf{c}} | \  | \ z_l \mathbf{Z}^{\mathbf{d}}  | =  | z _j  | \ |  z_k   \mathbf{Z}^{\mathbf{a}} | \  | \ z_l \mathbf{Z}^{\mathbf{b}}  | .
\end{aligned}
\end{equation*}
Furthermore,  \eqref{eq:comb4} is satisfied.

\noindent Since our $(\mathbf{a},\mathbf{b})$ satisfy
$\mathbf{a} \in \mathcal{M}_k$ and $\mathbf{b} \in \mathcal{M}_l$,
claim (3) is a consequence of claim (2).

\qed

We end this section exploiting the notation introduced  in claim (5) of Lemma \ref{lem:contcoo}  to introduce two classes  of functions.   First of all notice that    the linear maps $\eta \to  \< {\eta},\phi _j \>$
extend into bounded linear maps $\Sigma _r\to \R$ for any $r\in \R $.   We set
\begin{equation}\label{eq:phsp1}
\begin{aligned}
\Sigma _r^c & :=\left\{\eta\in \Sigma _r :\  \< {\eta},\phi _j \>=0,\ j=1,\cdots,n\right\} .
\end{aligned}
\end{equation}
The following two classes of functions  will be used in the rest of the paper.
Recall  that in   Def. \ref{def:comb0}  we introduced the space
$  L$ with   $\dim L= n (n-1).$  In Definitions \ref{def:scalSymb}--\ref{def:opSymb} by $\textbf{Z} $ we denote an auxiliary variable  independent of $z$ which takes values in $L$

\begin{definition}\label{def:scalSymb}
Let  $\mathfrak{B}$ be an open subset of a Banach space.
We will say that   $F(t, \mathfrak{b}, z, \mathbf{Z} ,\eta )\in C^{M}(I\times \mathfrak{B} \times \mathcal{A},\R)$, with
$I$ a neighborhood of 0 in $\R$ and
 $\mathcal{A}$   a neighborhood of 0 in  $  \C ^n \times L \times \Sigma _{-K}^c $
is   $F=\mathcal{R}^{i, j}_{ K,M} (t, \mathfrak{b},z,\textbf{Z},\eta)$,
 if    there exists    a $C>0$   and a smaller neighborhood  $\mathcal{A}'$ of 0   s.t.
 \begin{equation}\label{eq:scalSymb}
  |F(t, \mathfrak{b},z,\mathbf{Z},\eta)|\le C (\|  \eta \| _{\Sigma   _{-K}}+|\textbf{Z} |)^j (\|  \eta \| _{\Sigma   _{-K}}+ |\mathbf{Z}   |+|z |)^{i} \text{  in $I\times \mathfrak{B} \times \mathcal{A}  '$} .
\end{equation}
We will specify      $F=\mathcal{R}^{i, j}_{ K,M} (t, \mathfrak{b},z,\textbf{Z})$    if
\begin{equation}\label{eq:scalSymb1}
  |F(t, \mathfrak{b},z,\mathbf{Z},\eta)|\le C |\textbf{Z} | ^j  |z | ^{i}
\end{equation}
and       $F=\mathcal{R}^{i, j}_{ K,M} (t, \mathfrak{b},z,\eta  )$    if
\begin{equation}\label{eq:scalSymb2}
  |F(t, \mathfrak{b},z,\mathbf{Z},\eta)|\le C  \|  \eta \| _{\Sigma   _{-K}} ^j (\|  \eta \| _{\Sigma   _{-K}}+|z |)^{i} .
\end{equation}
We will  omit $t$ or $\mathfrak{b}$  if there is no dependence on such variables.

We write  $ F=\mathcal{R}^{i, j} _{K, \infty}$  if $F=\mathcal{R}^{i,j}_{K, m}$ for all $m\ge M$.
We write    $F=\mathcal{R}^{i, j}_{\infty, M} $       if   for all   $k\ge K$    the above   $F$ is the restriction  of an
$F(t, \mathfrak{b},z,\eta )\in C^{M}(I\times \mathfrak{B} \times \mathcal{A}_{k },\R)$ with  $\mathcal{A}_k$   a neighborhood of 0 in
$  \C  ^{n }\times L\times \Sigma _{-k}^c $ and
 which is
$F=\mathcal{R}^{i,j}_{k, M}$.
Finally we write
$F=\mathcal{R}^{i, j} _{\infty, \infty} $   if $F=\mathcal{R}^{i, j} _{k, \infty}$  for all $k$.

\end{definition}

\begin{definition}\label{def:opSymb}  We will say that an   $T(t, \mathfrak{b},z,\eta )\in C^{M}(I\times \mathfrak{B} \times \mathcal{A},\Sigma   _{K}  (\R^3, \C ))$,  with the above notation,
 is   $T= \mathbf{{S}}^{i,j}_{K,M} (t, \mathfrak{b},z,\textbf{Z},\eta)$,
 if     there exists  a $C>0$   and a smaller neighborhood  $\mathcal{A}'
$ of 0   s.t.
 \begin{equation}\label{eq:opSymb}
  \|T(t, \mathfrak{b},z,\mathbf{Z},\eta)\| _{\Sigma   _{K}}\le    C (\|  \eta \| _{\Sigma   _{-K}}+|\textbf{Z} |)^j (\|  \eta \| _{\Sigma   _{-K}}+ |\mathbf{Z}   |+|z |)^{i}   \text{  in $I\times \mathfrak{B} \times\mathcal{A}'$}.
\end{equation}
We  use notations
$ \mathbf{{S}}^{i,j}_{K,M} (t, \mathfrak{b},z,\textbf{Z})$, $ \mathbf{{S}}^{i,j}_{K,M} (t, \mathfrak{b},z,\eta  )$ etc.  as above.

\end{definition}

Notice that we have the elementary formulas
\begin{equation}\label{eq:algSymb}
 \resto^{a,b}_{K,M} \mathbf{{S}}^{i,j}_{K,M} = \mathbf{{S}}^{i+a,j+b}_{K,M}  \text{ and  }  \resto^{a,b}_{K,M} \resto^  {i,j}_{K,M} = \resto^  {i+a,j+b}_{K,M}   .
\end{equation}

 \begin{remark}\label{rem:sym}
 For   functions   $F(t, \mathfrak{b},z, \eta)$ and $T(t, \mathfrak{b},z, \eta)$ we   write
  $F(t, \mathfrak{b},z, \eta)=\resto ^{i,j}_{K,M} (t, \mathfrak{b},z,\textbf{Z},\eta)$ and $T(t, \mathfrak{b},z, \eta)=\mathbf{{S}}^{i,j}_{K,M} (t, \mathfrak{b},z,\textbf{Z},\eta)$ when the equality holds  restricting the variable $\mathbf{Z}$  to  $\mathbf{Z} =(z_i \overline{z}_j) _{i,j =1,...,n} $
  where $i\neq j$, for symbols satisfying Definitions \ref{def:scalSymb}--\ref{def:opSymb}.

 \noindent Furthermore, later, when we write $\resto ^{i,j}_{K,M} $ and $\mathbf{{S}}^{i,j}_{K,M}$,    we   mean $\resto ^{i,j}_{K,M} ( z,\textbf{Z},\eta) $ and $\mathbf{{S}}^{i,j}_{K,M}  ( z,\textbf{Z},\eta)$.

 \noindent Notice that $F= \resto ^{i,j}_{K,M} ( z,\textbf{Z} ) $ or  $T= \mathbf{S} ^{i,j}_{K,M} ( z,\textbf{Z} ) $ do not mean independence of the variable $\eta$.
\end{remark}

\section{Invariants}
\label{sec:invariants}

Equation \eqref{eq:NLS}  admits the  energy and mass  invariants,
   defined as follows:
\begin{equation}\label{eq:inv}
\begin{aligned}
&E(u ):= E_ K(u)   + E_P(u) \text{, where }  E_ K(u) := \langle  H u, \overline{u} \rangle \text{ and } \\&
 E_P(u)= \frac 12 \int_{\R^3}  |u(x)  |^4     \,dx \  ; \quad
Q(u ):= \langle    u, \overline{u} \rangle  .
\end{aligned}
\end{equation}
   We have   $ {E}\in C^\infty ( H^1( \R ^3, \C  ), \R  )$ and $Q\in C^\infty (  L^2(
\R ^3, \C ) , \R  )$.    We denote by $dE$ the Frech\'et derivative of $E$.
 We define $ \nabla {E}\in C^\infty ( H^1( \R ^3, \C  ), H ^{-1}( \R ^3, \C  )  )$ by  $dEX= \Re \langle   \nabla {E}, \overline{X} \rangle$
  for any $X\in   H^1$. We define also $\nabla _{u}{E}$ and  $\nabla _{\overline{u}}{E}$
 by
 \begin{equation*}
\begin{aligned}
& dEX =   \langle   \nabla _{u} {E},  {X} \rangle  + \langle   \nabla _{\overline{u}} {E}, \overline{X} \rangle  \text{ that is }
 \nabla _{u} {E} =2^{-1} \overline{\nabla {E}} \text{  and } \nabla _{\overline{u}} {E} =2^{-1}  {\nabla {E}}.
\end{aligned}
\end{equation*}
Notice that $ \nabla   {E} =2Hu+2 |u|^2 u$.
Then equation \eqref{eq:NLS} can be
interpreted as

\begin{equation}\label{eq:NLSham}
\begin{aligned}
& \im \dot u = \nabla _{  \overline{u}} E(u ) .
\end{aligned}
\end{equation}

\begin{lemma}
  \label{lem:EnExp} Consider the coordinates $(z,\eta )\to u$ in  Lemma \ref{lem:systcoo}. Then  there exists some functions as in Definitions
  \ref{def:scalSymb} and  \ref{def:opSymb}
  s.t.  for $(z,\eta )\in B_{\C^n}(0 , d  _0) \times ( B_{H^1}(0 , d  _0)\cap \mathcal{H}_c[0]) $
    we  have   for any preassigned $r_0\in \N$ the expansion (where c.c. means complex conjugate)
\begin{equation}\label{eq:enexp1} \begin{aligned}
  & E (u)= \sum_{j=1}^n E(Q_{jz_j}) + \langle   H   \eta,  \overline{ \eta} \rangle  +  \resto ^{1,  2  }_{r_{0} , \infty} (z, \eta  )
 \\&
+\sum _{j\neq k} [ E _{j z_j}  ( \Re  \langle   q_{j z_j},  \overline{z}_k\phi _k   \rangle   + \Re  \langle   q_{k z_k},  \overline{z}_j\phi _j   \rangle  )+ \Re \langle |Q_{k z_k}|^2   Q_{k z_k}    ,  \overline{z}_j\phi_j \rangle       ]
 \\&   +
\resto ^{0,  2N+5 }_{r_{0} , \infty} (z,\textbf{Z}  ) +
	 \sum _{j =1}^n\sum _{l =1} ^{2N+3}	 \sum _{
	  |\textbf{m}|=l+1 }  \textbf{Z}^{\textbf{m}}   a_{j \textbf{m} }( |z_j |^2 )   \\&   +   \Re \langle
  \textbf{S} ^{0,   2N+4 }_{r_{0} , \infty} (z,\textbf{Z}  ) , \overline{\eta} \rangle   +
   \sum _{j,k=1}^n \sum _{l =1} ^{2N+3}	 \sum _{
	  |\textbf{m}|=l  }
     ( \overline{z}_j \textbf{Z}^{\textbf{m}}  \langle
  G_{jk\textbf{m}}(|z_k|^2  ), \eta \rangle  +c.c. )
		+\\&    \sum _{ i+j  = 2}   \sum _{
		  |\textbf{m}|\le 1 }  \textbf{Z}^{\textbf{m}}    \langle G_{2\textbf{m} ij } ( z ),   \eta ^{   i} \overline{\eta} ^j\rangle
		+  \sum _{  d+c  = 3 }   \sum _{ i+j  = d}   \langle G_{d ij } ( z ),   \eta ^{   i} \overline{\eta} ^j\rangle  \resto ^{0,  c }_{r_{0} , \infty} (z, \eta )       + E _P( \eta)  \text{ where:}
  \end{aligned}
\end{equation}

  \begin{itemize}

\item $  ( a_{j\textbf{m} } ,  G_{jk\textbf{m}}   )  \in C^\infty (B_{\R}(0,d  _0), \C\times  \Sigma _{r_{0}}(\mathbb{R}^3, \mathbb{C}   ) $;
\item
$(G_{2\textbf{m} ij },G_{d ij })\in C^\infty (B_{\C^n}(0,d  _0),
\Sigma _{r_{0}} (\mathbb{R}^3, \mathbb{C}   )\times
\Sigma _{r_{0}} (\mathbb{R}^3, \mathbb{C}   ))$;

\item For $|\textbf{m}|=0$, where in particular we have $G_{2\textbf{0} ij } (0)=0$, we have
\begin{equation}\label{eq:b02}
\begin{aligned}
  &    \sum _{ i+j  = 2}          \langle G_{2\textbf{0} ij } ( z ),   \eta ^{   i} \overline{\eta} ^j\rangle  =  \sum _{j=1}^{n}
  \langle |Q_{j z_{j}}|^2 \eta  , \overline{\eta } \rangle +2 \sum _{j=1}^{n}   \Re \langle  Q_{j z_{j}} \Re (Q_{j z_{j}} \overline{\eta}) ,    \overline{\eta } \rangle ;\end{aligned}
\end{equation}


\item $\resto ^{1,  2  }_{r_{0} , \infty} (e^{\im \vartheta} z, e^{\im \vartheta} \eta )=\resto ^{1,  2  }_{r_{0} , \infty} (  z,   \eta )  $ for all $\vartheta\in \R$ for the 3rd term in the r.h.s. of \eqref{eq:enexp1}.

\end{itemize}
\end{lemma}

\begin{remark}
\label{rem:crux1} In formula \eqref{eq:enexp1}   the terms  of the second line could potentially derail our proof.  They appear in \eqref{eq:troub2}--\eqref{eq:troub31}.
Similarly problematic is the first term in the r.h.s. in
\eqref{eq:alphaest1} later. All these terms are tied up.
Indeed,   in
Lemma \ref{lem:KExp2} we will show that in a system of  coordinates  better suited to search an effective Hamiltonian  the problematic terms in the expansion of $E$ cancel out.
\end{remark}

\noindent
In the proof of Lemma \ref{lem:EnExp} we use  the following lemma.
\begin{lemma}
  \label{lem:id1} For   we have for $j\neq k$ and $\delta  E_{jz_j } := E_{jz_j }  -e_j$
\begin{equation}\label{eq:id2}
\begin{aligned}
&  E_{jz_j }  \langle q_{kz_k } ,   \phi _j \rangle +
\langle |Q_{kz_k }|^2Q_{kz_k } ,   \phi _j\rangle     =E_{kz_k }  \langle q_{kz_k } , \phi _j\rangle  +  \delta E_{jz_j }  \langle q_{kz_k } , \phi _j\rangle   .
\end{aligned}
\end{equation}

\end{lemma}
\proof   We apply  $\langle \ , \phi _j \rangle $   to
\begin{equation*}
\begin{aligned}
& Hq_{kz_k } +  |Q_{kz_k }  |^2 Q_{kz_k } =
z_k \delta  E_{kz_k } \phi _k +E_{kz_k }  q_{kz_k }
\end{aligned}
\end{equation*}
to get the following equality which from  $e_j = E_{jz_j }- \delta E_{jz_j }$  yields \eqref{eq:id2}:
\begin{equation*}
\begin{aligned}
& e_j \langle q_{kz_k }  , \phi _j \rangle    +
\langle |Q_{kz_k }|^2Q_{kz_k } ,   \phi _j\rangle     =E_{kz_k }  \langle q_{kz_k } , \phi _j\rangle \ .
\end{aligned}
\end{equation*}

\qed

\noindent \textit{Proof of Lemma \ref{lem:EnExp}.}   First of all, we have  the Taylor expansion
\begin{align}
  & \label{eq:enexp2} E (u)  = E  ( \sum _{j=1}^{n} Q_{j z_j}) + \Re \langle \nabla  E  (\sum _{j=1}^{n} Q_{j z_j}),  \overline{R[z]\eta} \rangle  \\&\nonumber+  2^{-1}\Re \langle  \nabla  ^2 E  (\sum _{j=1}^{n} Q_{j z_j})R[z]\eta ,  \overline{R[z]\eta} \rangle +E_3 (\eta ) \text{ with }  E_3 (\eta ):= \\&   \int _0^1  (1-t) \Re  \langle  \big [\nabla  ^2 E_P  (\sum _{j=1}^{n} Q_{j z_j}+tR[z]\eta)-\nabla  ^2 E _P  (\sum _{j=1}^{n} Q_{j z_j})\big ] R[z]\eta ,  \overline{R[z]\eta} \rangle dt   \nonumber
  \end{align}

  {\bf Step 1}. We consider the  expansion of the 1st term in the r.h.s of \eqref{eq:enexp2}.
   We have
\begin{equation*}
\begin{aligned}
  &    |\sum Q_{j z_j} | ^4 = \sum | Q_{j z_j} | ^4  + 4 \sum _{j\neq k}
   | Q_{j z_j} | ^2 \Re ( Q_{j z_j}  \overline{Q}_{k z_k})
   \\& +
   2 \sum _{j<k} | Q_{j z_j} | ^2| Q_{k z_k} | ^2+  \sum _{\substack{ j\neq k , \  j'\neq k'} }\Re ( Q_{j z_j}  \overline{Q}_{k z_k})\Re ( Q_{j' z_{j'}}  \overline{Q}_{k' z_{k'}})
+4\sum_{k<l,\ j\neq k,l}|Q_{j z_j}|^2\mathrm{Re}(Q_{k z_k}\overline{Q}_{l z_l}).
     \end{aligned}
\end{equation*}
\noindent All terms are invariant
by change of variable  $z\rightsquigarrow e^{\im \vartheta}z$.
 The 2nd   line is $O( |\mathbf{Z}|^2)$.
 We conclude that
\begin{align}\nonumber
     E  ( \sum _{j=1,..., n}  Q_{j z_j})  =&     \sum  _{j,k} \langle H Q_{j z_j} , \overline{Q_{k z_k}}\rangle  + \frac 12 \int   |\sum _{j=1,..., n} Q_{j z_j} | ^4
 = \sum _{j=1,..., n}E  (   Q_{j z_j})       +R_1 \nonumber \\&+ \sum _{j\neq k}[\Re  \langle HQ_{j z_j}, \overline{Q_{k z_k}} \rangle
	+ 2
	 \Re \langle |Q_{j z_j}|^2   Q_{j z_j}    ,  \overline{Q}_{k z_k} \rangle           ]  , \label{eq:troub2}
  \end{align}
where
\begin{align*}
R_1:= &   \sum _{j<k} \int | Q_{j z_j} | ^2| Q_{k z_k} | ^2+ \frac 1 2  \sum _{\substack{ j\neq k , \  j'\neq k'} }\int \Re ( Q_{j z_j}  \overline{Q}_{k z_k})\Re ( Q_{j' z_{j'}}  \overline{Q}_{k' z_{k'}})
\\&+2\sum_{k<l,\ j\neq k,l}\int|Q_{j z_j}|^2\mathrm{Re}(Q_{k z_k}\overline{Q}_{l z_l})=O(|\mathbf Z|^2).
\end{align*}
  By Prop.  \ref{prop:bddst}   and by \eqref{eq:id2} the summation in the last line of \eqref{eq:troub2} equals
\begin{align}\nonumber
  &   \sum _{j\neq k}[E _{j z_j}\Re  \langle   Q_{j z_j}, \overline{Q_{k z_k}} \rangle
	+         \Re \langle |Q_{j z_j}|^2   Q_{j z_j}    ,  \overline{Q}_{k z_k} \rangle ]  \\& \label{eq:expand0}  = \sum _{j\neq k} [ E _{j z_j}  ( \Re  \langle   q_{j z_j},  \overline{z}_k\phi _k   \rangle   + \Re  \langle   q_{k z_k},  \overline{z}_j\phi _j   \rangle  )+ \Re \langle |Q_{k z_k}|^2   Q_{k z_k}    ,  \overline{z}_j\phi_j \rangle       ] +R_2,
  \end{align}
where
\begin{align*}
R_2:=\sum_{j\neq k}E_{j z_j}\Re \<q_{j z_j},\overline{q_{k z_k}}\>+\Re \<|Q_{kz_k}|^2Q_{kz_k},\overline{q_{jz_j}}\>=O(|\mathbf Z|^2).
\end{align*}
  The summation in \eqref{eq:expand0} is $O(|z| ^2 \ |\textbf{Z}|) $ and  not of  the form  $O( |\textbf{Z}|  ^2) $.   Indeed, in the particular case
		when $z_k=\rho _k$ and    $z_j=\rho _j$  are real numbers, we have what follows, which is not $O( \rho _k ^2 \rho _j^2 )$,
	\begin{equation} \label{eq:troub31}
\begin{aligned}
  &        E _{j z_j}    \Re  \langle   q_{j z_j},  \overline{z}_k\phi _k   \rangle   + E _{k z_k}  \Re  \langle   q_{k z_k},  \overline{z}_j\phi _j   \rangle  +\Re \<|Q_{kz_k}|^2Q_{kz_k},\overline{z}_j\phi_j \>\\& =  \rho _k \rho _j      [ E _{j \rho _j}  \rho _j ^2    \langle   \widetilde{q}_{j }(\rho_j^2),  \phi _k   \rangle   + E _{k \rho _k}  \rho _k ^2    \langle   \widetilde{q}_{k }(\rho _k),    \phi _j   \rangle +\rho_k^2\<(\phi_k+\hat q_{k}(\rho_k^2))^3, \phi_j\> ]    \   .
  \end{aligned}
\end{equation}
 Finally, we observe that   the $R_1+R_2=O(|\mathbf Z|^2)$  summed up together
 yield the  3rd line of  \eqref{eq:enexp1}.

\noindent Indeed, since $R_1+R_2$ is gauge invariant, by Lemma \ref{lem:expandzzj} in Appendix \ref{app:B}  we have
  \begin{equation} \label{eq:fix2}
\begin{aligned}
  &    R_1+R_2=
	 \sum _{j =1}^n\sum _{l =1} ^{2N+3}	 \sum _{
	  |\textbf{m}|=l+1 }  \textbf{Z}^{\textbf{m}}   b_{j \textbf{m} }(  |z_j|^2  ) + O( |\mathbf{Z}| ^{2N+5}).
     \end{aligned}
\end{equation}
  with  $O( |\mathbf{Z}| ^{2N+5})$   smooth in $z$, independent of $\eta$
and gauge invariant.

We  have discussed the contribution to \eqref{eq:enexp1} of the 1st term in the expansion  \eqref{eq:enexp2}. Now we consider the other terms  in
\eqref{eq:enexp2}.

 {\bf  Step 2}. We consider the  expansion of the 2nd term in the r.h.s of \eqref{eq:enexp2}.

By   $\Re  \langle \nabla  E  (Q_{j z_j}),  \overline{R[z]\eta }\rangle =2\Re  E_{jz_j} \langle  Q_{j z_j} ,  \overline{R[z]\eta}  \rangle  =0$, which follows by $R[z]\eta\in \mathcal{H}_c[z] $ and by $ \im Q_{ j z_j} =- z _{jI} D _{jR} {Q}_{ j z_j} + z _{jR} D _{jI}  {Q}_{ j z_j}$,
see (11) in \cite{GNT} (and which is an immediate   consequence of ${Q}_{ j z_j}=e^{\im\theta}{Q}_{ j |z_j|}$ for  $z_j = e^{\im\theta} |z_j| $),
we have
\begin{align} & \nonumber   \Re \langle \nabla  E  (\sum _{j=1}^{n} Q_{j z_j}),  \overline{R[z]\eta} \rangle  = \overbrace{\Re \langle \nabla  E  (  Q_{1 z_1}),  \overline{R[z]\eta} \rangle}^{0}  \\&\nonumber    + \int _0^1    \partial _{t}    \Re \langle \nabla  E  (Q_{1 z_1} + t\sum _{j>1}  Q_{j z_j}),  \overline{R[z]\eta} \rangle   dt = \Re \langle \nabla  E  (\sum _{j>1} Q_{j z_j}),  \overline{R[z]\eta} \rangle   \\& \nonumber   + \int _{[0,1]^2}\partial _{s}\partial _{t}  \Re \langle  \nabla  E  _P(sQ_{1 z_1}+t\sum _{l>1}  Q_{l z_l}),  \overline{R[z]\eta}  \rangle dt ds  \\& = \sum    _{j=1}^{n-1}
  \int _{[0,1]^2}\partial _{s}\partial _{t}  \Re \langle  \nabla  E  _P(sQ_{j z_j}+t\sum _{l>j }  Q_{l z_l}),  \overline{R[z]\eta}  \rangle dt ds ,\label{eq:enexp20}
\end{align}
where the last line is obtained repeating the argument in the first three lines.
For $\widehat{Q}_j =\sum _{l>j }  Q_{l z_l} $ and   by $\nabla E_P(u)= 2|u|^2u$, the last line of \eqref{eq:enexp20}
is, in the notation of Lemma \ref{lem:contcoo},
\begin{equation*}\label{eq:enexp21}
\begin{aligned}
       2 \sum    _{j=1}^{n-1}
     \Re \big \langle  2 Q_{j z_j} |\widehat{Q}_j|  ^2 +& 2 |Q_{j z_j} | ^2\widehat{Q}_j +Q_{j z_j}^2\overline{\widehat{Q}_j} +\overline{Q}_{j z_j}\widehat{Q}_j ^2,\    \overline{ \eta}   + \phi _j  (\langle \overline{\widehat{B}}_j (\widehat{Z})  , \overline{\eta} \rangle +      \langle \overline{z}_i \overline{z}_\ell  \overline{\widehat{C}}_{i\ell j}(\widehat{Z}) ,  {\eta}\rangle)\big \rangle  .
  \end{aligned}
\end{equation*}
Further expanding  $\widehat{Q}_j =\sum _{l>j }  Q_{l z_l} $ and using  $Q_{l z_l}   =z_l (\phi _l + \widehat{q} _{l  }  (|z_l|^2)$,  the above term is of the form
\begin{align} \nonumber
  &     \nonumber  \sum _{j=1}^n	\sum _{   |\textbf{m}  | =1} ( \overline{z}_j \textbf{Z}^{\textbf{m}}  \langle
  G_{j\textbf{m}}(\widehat{Z} ), \eta \rangle  +   \text{c.c.} )  .
  \end{align}
  As in Step 1, by Lemma \ref{lem:expandzzS},
	this can be expanded into
\begin{align}
  &     \sum _{j=1}^n	\sum _{ 1\le |\textbf{m}  |  \le 2N+3}\left(  \overline{z}_j \textbf{Z}^{\textbf{m}}  \langle
  G_{jk\textbf{m}}(|z_k|^2    ), \eta \rangle  +      \text{c.c.}   \right)
		       +  \sum _{   |\textbf{m}  |  = 2N+4}
\left(  \textbf{Z}^{\textbf{m}}  \langle
  G_{ \textbf{m}}(z  ), \eta \rangle  +         \text{c.c.}   \right)   . \label{eq:Tayl1}
  \end{align}
	Thus    the last line in \eqref{eq:enexp20} can be  absorbed in the  4th line of
	\eqref{eq:enexp1}.

{\bf  Step 3}. We consider the  expansion of the 3rd term in the r.h.s of \eqref{eq:enexp2}.
Using $ \nabla  ^2 E _K (u)=2H$ and
proceeding  as  for \eqref{eq:enexp2},  we obtain
\begin{equation*}
\begin{aligned}
  &  2^{-1}\Re \langle  \nabla  ^2 E  (\sum _{j=1}^{n} Q_{j z_j})R[z]\eta ,  \overline{R[z]\eta} \rangle \\& = 2^{-1}
	\Re \langle  \nabla  ^2 E_K  (\sum _{j=1}^{n} Q_{j z_j})R[z]\eta ,  \overline{R[z]\eta} \rangle + 2^{-1}
		\sum _{j=1}^{n}\Re \langle  \nabla  ^2 E _P (  Q_{j z_j})R[z]\eta ,  \overline{R[z]\eta} \rangle \\& + 2^{-1} \sum    _{j=1}^{n-1}
  \int _{[0,1]^2}\partial _{s}\partial _{t}  \Re \langle  \nabla ^2 E _{P} (sQ_{j z_j}+t\sum _{l=j+1}^{n} Q_{l z_l})R[z]\eta,  \overline{R[z]\eta}  \rangle dt ds.
  \end{aligned}
\end{equation*}
The 3rd line is absorbed in the $ \textbf{Z}^{\textbf{m}}    \langle G_{2\textbf{m} ij } ( z ) ,   \eta ^{   i} \overline{\eta} ^j\rangle+ \resto ^{1,  2  }_{{r_{0}} , \infty} (z, \eta  )$ with $|\textbf{m}|=1$
terms    in \eqref{eq:enexp1}.
From the  2nd line, using \eqref{eq:contcoo21}--\eqref{eq:contcoo2}  and in particular
$\alpha_j[z]\eta =\resto ^{1,  1  }_{{r_{0}} , \infty} (z, \eta  )
         $ for
 the last equality,   we have
\begin{equation*}
\begin{aligned}
        2^{-1}
	\Re \langle  \nabla  ^2 E_K  (\sum _{j=1}^{n} Q_{j z_j})R[z]\eta ,  \overline{R[z]\eta} \rangle & =  \langle  H R[z]\eta ,  \overline{R[z]\eta} \rangle
     =  \langle   H \eta,  \overline{ \eta} \rangle  +2 \sum _{j=1}^{n}\Re \left [(\alpha_j[z]\eta ) \langle   H \phi_j  ,  \overline{ \eta} \rangle \right ]\\&+ \sum _{j,k=1}^{n} e_j |\alpha_j[z]\eta  |^2
         =  \langle   H \eta,  \overline{ \eta} \rangle +
         \resto ^{1,  2  }_{{r_{0}} , \infty} (z, \eta  ),
  \end{aligned}
\end{equation*}
which yield the 2nd and 3rd terms in the r.h.s. of \eqref{eq:enexp1}.
For
\begin{equation*}
\begin{aligned}
  &   2^{-1}\sum _{j=1}^{n} \nabla  ^2 E _P (  Q_{j z_j}) \eta = \sum _{j=1}^{n}  |Q_{j z_{j}}|^2 \eta +2 \sum _{j=1}^{n}   Q_{j z_{j}} \Re (Q_{j z_{j}} \overline{\eta})
  \end{aligned}
\end{equation*}
 we have for $ G_{2\textbf{0} ij } ( z )$  as in     \eqref{eq:b02}
 \begin{equation}\label{eq:exp-1}
\begin{aligned}
  &   2^{-1}
		\sum _{j=1}^{n}\Re \langle  \nabla  ^2 E _P (  Q_{j z_j})R[z]\eta ,  \overline{R[z]\eta} \rangle   =      \resto ^{1,  2  }_{{r_{0}} , \infty} (z, \eta  )+\sum _{ i+j  = 2}   \langle G_{2\textbf{0} ij } ( z ),   \eta ^{   i} \overline{\eta} ^j\rangle  .
  \end{aligned}
\end{equation}
This $ \resto ^{1,  2  }_{{r_{0}} , \infty} (z, \eta  )$ defines the
   3rd term  in the r.h.s. of \eqref{eq:enexp1}.  Notice that $\resto ^{1,  2  }_{r_{0} , \infty} (e^{\im \vartheta} z, e^{\im \vartheta} \eta )=\resto ^{1,  2  }_{r_{0} , \infty} (  z,   \eta )  $ because this invariance is satisfied  both by the l.h.s. of  \eqref{eq:exp-1} (by the invariance of $E$, \eqref{eq:decomposition3} and by Lemma \ref{lem:contcoo}) and by the last summation in the r.h.s. of  \eqref{eq:exp-1}, by formula \eqref{eq:b02}.

  {\bf  Step 4}.
	We now turn to the  $ E_3 (\eta )$ term in  \eqref{eq:enexp2}.
	By elementary computations
\begin{align}
  &  \nonumber  E_3 (\eta ) = \int _{[0,1]^2} t (1-t)   d^3
  E_P  (\sum _{j\ge 1}  Q_{j z_j}+stR[z]\eta)\cdot (R[z]\eta  )^3 dt ds= E_P(R[z]\eta) \\&+\int _{[0,1]^3} t(1-t)   d^4
  E_P  (\tau \sum _{j\ge 1} Q_{j z_j}+stR[z]\eta)\cdot (R[z]\eta  )^3 \sum _{j\ge 1}Q_{j z_j}dt dsd\tau , \label{eq:EPexp}
  \end{align}
  with   $ d^3E_P(u)\cdot v ^3$  the trilinear differential form applied to
  $(v,v,v)$ and $ d^4E_P(u)\cdot v ^3w$ the 4--linear differential form applied to
  $(v,v,v,w)$.

\noindent In particular we have used the fact that since $ d^{j}E_P(0)=0$ for $0\le j\le 2$ we have
\begin{equation}\label{eq:EPexp1}
\begin{aligned}
  &    E_P(R[z]\eta) = \int _{[0,1]^2} t \ (1-t)   d^3
  E_P  ( stR[z]\eta)\cdot (R[z]\eta  )^3 dt ds .
  \end{aligned}
\end{equation}
 For $\beta (u)=|u|^4$ and using the fact that
  $d^4\beta (u)\in  B^4( \C , \ R )$ is   constant  in $u$,  the
  2nd line of  \eqref{eq:EPexp}  is
   \begin{align}
  &    \frac{1}{12}  \int _{\R ^3}   d^4
  \beta  \cdot ((R[z]\eta ) (x) )^3 \sum _{j\ge 1}Q_{j z_j}(x) dx  , \nonumber
  \end{align}
  and  can be  absorbed in the   $\langle G_{d ij } ( z ),   \eta ^{   i} \overline{\eta} ^j\rangle  \resto ^{0,  c }_{{r_{0}}, \infty} (z, \eta )$ terms in \eqref{eq:enexp1}.  We   expand $E_P(R[z]\eta)$  as a sum of
   similar terms and of
   $E_P( \eta)$.
\qed

In order  to extract from the functional in \eqref{eq:enexp1} an effective Hamiltonian well suited for the FGR and dispersive estimates, we need to implement   a Birkhoff normal form  argument, see Sect.\ref{sec:nforms}. This requires an intermediate change of coordinates,
which will partially normalize  the  symplectic form $\Omega$ defined in  \eqref{eq:Omega}
below, and   diagonalize the homological equations. Notice that, as a bonus,
 this  change of coordinates
   erases the bad terms in the
expansion of $E$  in \eqref{eq:enexp1} discussed in {Remark}
\ref{rem:crux1}.

\section{Darboux Theorem}
\label{sec:darboux}

System  \eqref{eq:NLSham}    is   Hamiltonian   with respect to the symplectic form in $H^1(\R^3,\C
)$
\begin{equation}
\label{eq:Omega} \Omega (X,Y
):=\im \langle X ,\overline{Y} \rangle -\im  \langle \overline{X} ,Y \rangle =2\Im \langle \overline{X} , Y\rangle .
\end{equation}
In terms of the spectral decomposition of $H$ (recall $\overline{\phi} _j= {\phi} _j$)
\begin{equation}
\label{eq:specH}  X=  \sum _{j=1}^{n} \langle X ,  {\phi} _j \rangle {\phi} _j+P_cX
\end{equation}
\begin{equation}
\label{eq:Omega1} \Omega (X,Y
) = \im \sum _{j=1}^{n} \left (  \langle X ,  {\phi} _j \rangle \langle \overline{Y} ,  {\phi} _j \rangle  -
  \langle \overline{X} ,  {\phi} _j \rangle \langle  {Y} ,  {\phi} _j \rangle
  \right ) +   \im  \langle P_cX ,P_c\overline{Y} \rangle - \im \langle P_c\overline{X} ,P_cY \rangle .
\end{equation}
However, in terms of the  coordinates in Lemma \ref{lem:systcoo}, $\Omega $ admits
a quite more complicated representation, as we shall see. This will require  us to adjust these coordinates.

Our first observation is that for the  coordinates in Lemma \ref{lem:systcoo} we have the following facts.

\begin{lemma}\label{lem:etaprime}
 The Frech\'et derivative of $\eta (u)$ and $d z_{j  }$ is given by the following formula:
\begin{equation}\label{eq:etaprime}
d\eta (u)=  - \sum _{j=1,...,n}  \sum _{A=I,R}P_c D_ {jA} q_{j z_j} dz_{j A} +P_c ,
\end{equation}
\begin{equation}   \label{eq:dzR}   \begin{aligned}&   d z_{j  }=    \langle \quad  , \phi _j\rangle
        -\sum _{k:k\neq j}\sum _{A=I,R}  \langle D_{ kA}q_{kz_k}  , \phi _j\rangle dz_{kA } -   \sum _{k=1 }^{n}\sum _{A=I,R}    D_{ kA} \alpha _j[z]\eta  dz_{kA } -    \alpha _j[z] \circ d\eta .
         \end{aligned}
\end{equation}
 Analogous formulas for $ d z_{j R }$  and  $ d z_{j  I}$   are obtained applying $\Re$  and $ \Im $   to \eqref{eq:dzR}.
\end{lemma}

\begin{proof}
We start  with \eqref{eq:etaprime}.
By the  independence of $z$ and $\eta$, we have
\begin{equation}\label{eq:coord1}
d\eta  \frac{\partial}{\partial z_{j R}}=d\eta  \frac{\partial}{\partial z_{j I}}=0,
\end{equation}
where  \begin{equation}\label{eq:coord2}
\begin{aligned}
\frac{\partial}{\partial z_{j A}}=& D _{jA} Q_{j z_j}+\sum_{k=1}^nD _{jA} \(\alpha_k[z]\eta\)\phi_k.
\end{aligned}
\end{equation}
Next, for $\xi\in \mathcal{H}_c[0]$ we have what follows, which implies $d\eta  R[z]P_c=\left. 1 \right|_{\mathcal{H}_c[0]}$:
\begin{equation*}
 \begin{aligned}  & d\eta  R[z]P_c\xi =
\left.\frac {d}{dt}\right|_{t=0}\eta(Q_{j z_j}+R[z](\eta+t\xi))
=\xi .  \end{aligned}    \end{equation*}
 So $d\eta  = \sum (a_j dz_{j R} + b_j dz_{j I}) + P_c $,
where we   used $P_c R[z]=1$.
  $a_j$ and $b_j$ can be computed applying $ \sum (a_j dz_{j R} + b_j dz_{j I}) + P_c $  to   the vectors \eqref{eq:coord2}
and using  \eqref{eq:coord1}.
Finally  \eqref{eq:dzR}   follows by  \begin{align} \nonumber
                         & z_{j }(u)   =    \langle u-\sum _{k=1}^{n}q_{k z_k}  - R[z] \eta    , \phi _j\rangle  =  \langle u-\sum _{k:k\neq j} q_{k z_k} , \phi _j\rangle -  \alpha _j[z]\eta    .
                      \end{align}
\end{proof}
We consider the function $\overline{\eta} (u)$. Notice that
$d \overline{\eta} (u) X= \frac{d}{dt}\overline{\eta} (u+t X) _{|t=0}
=\overline{d  {\eta} (u) X }$. Now we introduce a new symplectic form. Notice that
our final choice of  symplectic form is not the  $\Omega_0'$ defined right here in
\eqref{eq:defOm0pr},  but rather the  $\Omega_0 $ defined   in
\eqref{eq:defOm0} further down.

\begin{lemma}
  \label{lem:Omega0}
Set
\begin{equation}\label{eq:defOm0pr}
\begin{aligned}
\Omega_0 '&:=  2\sum_{j=1}^ndz_{j R}\wedge dz_{j I}  + \im \<d \eta  , d\overline{\eta}  \>   -\im \<d \overline{\eta}  , d {\eta}  \>   \text{ and }
\\ B_0' &:=  \sum_{j=1}^n (z_{j R}dz_{j I}-z_{j I}dz_{j R})   - \frac{{\im} }{2} (\<  \overline{\eta} , d {\eta} \>-\<   {\eta} , d \overline{{\eta}} \> ).
\end{aligned}
\end{equation}
Then   $dB'_0=\Omega_0'$ and       $\Omega=\Omega_0'$ at $u=0$ for the $\Omega$ of \eqref{eq:Omega}. Furthermore
\begin{equation}\label{eq:defOm0pr1}
\begin{aligned}
 \Phi ^*B'_0=B'_0  \text{ for $\Phi (u)= e^{\im \vartheta}u$ for any fixed $\vartheta\in \R$}.
\end{aligned}
\end{equation}

\end{lemma}
\proof The equality  $dB_0'=\Omega_0'$  is elementary.
Indeed $d(z_{j R}dz_{j I}-z_{j I}dz_{j R})=2dz_{j R}\wedge dz_{j I}$
and for a pair of constant vectorfields $X$ and $Y$, by     $d ^2\eta (X,Y)=d ^2\eta (Y,X)$, we have
\begin{equation*}
\begin{aligned}
  d\<  \overline{\eta} , d {\eta} \> (X,Y)=   X\<  \overline{\eta} , d {\eta} Y\>-Y\<  \overline{\eta} , d {\eta} X\>=  \<  d\overline{\eta} X , d {\eta} Y\>- \<  d\overline{\eta} Y , d {\eta} X\> .
\end{aligned}
\end{equation*}
  This yields $d\<  \overline{\eta} , d {\eta} \> = \<  d\overline{\eta  }  , d {\eta}  \>-\<  d {\eta}    , d \overline{{\eta}}  \> $ and also $d\<   {\eta} , d \overline{{\eta}} \> =-d\<  \overline{\eta} , d {\eta} \> = \<  d {\eta  }  , d \overline{{\eta}}  \>-\<  d \overline{{\eta} }   , d  {{\eta}}  \> $

To compute $\Omega_0'$ at $u=0$ we observe that by Lemma \ref{lem:etaprime}
we have $d\eta =P_c$ at  $u=0$, so that
\begin{equation}\label{eq:defOm01}
  \im \<d \eta  X, d\overline{\eta} Y \>   -\im \<d \overline{\eta} X , d {\eta} Y \> = \im  \langle P_cX ,P_c\overline{Y} \rangle - \im \langle P_c\overline{X} ,P_cY \rangle \text{ at  $u=0$}.
\end{equation}
By Lemma \ref{lem:etaprime} and
    Proposition \ref{prop:bddst},
 at $u=0$ we have $dz_{jR }= \Re \langle \quad  , \phi _j\rangle $
and $dz_{jI }= \Im \langle \quad  , \phi _j\rangle $.
Summing on repeated indexes, we have

  \begin{align}&\label{eq:defOm02}  \im   \left (  \langle X ,  {\phi} _j \rangle \langle \overline{Y} ,  {\phi} _j \rangle  -
  \langle \overline{X} ,  {\phi} _j \rangle \langle  {Y} ,  {\phi} _j \rangle
  \right ) = -2\Im  \left(  \langle X ,  {\phi} _j \rangle \langle \overline{Y} ,  {\phi} _j \rangle  \right ) =\\& 2 (\Re \langle X ,  {\phi} _j \rangle \Im \langle {Y} ,  {\phi} _j \rangle  -\Re \langle {Y} ,  {\phi} _j \rangle \Im \langle X ,  {\phi} _j \rangle   )=\nonumber\\& 2  \Re \langle \ ,  {\phi} _j \rangle \wedge \Im \langle \ ,  {\phi} _j \rangle  (X,Y) =2  dz_{j R}\wedge dz_{j I}   {| _{u=0}} (X,Y) .\nonumber \end{align}
By \eqref{eq:defOm01}--\eqref{eq:defOm02} we get    $\Omega=\Omega_0'$ at $u=0$.
Finally, \eqref{eq:defOm0pr1} follows immediately by
\begin{equation}\label{eq:defOm0pr2}
\begin{aligned}
B_0'&:=  \sum_{j=1}^n \Im (\overline{z}_{j  }dz_{j  } )   + \Im    \<  \overline{\eta} , d {\eta} \> .
\end{aligned}
\end{equation}

\qed

\noindent  Summing on repeated indexes and using the notation in Prop.\ref{prop:bddst}, we introduce  the differential  forms:
 \begin{equation}\label{eq:defOm0}
\begin{aligned}
&\Omega_0  :=  \Omega_0 ' +   \im\gamma _j(|z_j|^2)  dz_{j } \wedge d\overline{z}_{j }
  \text{  where}
\\&   \gamma _j(|z_j|^2):= \<  \widehat{{q}}_{j  } (|z_j|^2), \widehat{{q}}_{j  }   (|z_j|^2) \> +  2|z_j|^2   \<  \widehat{{q}}_{j  } (|z_j|^2), \widehat{{q}}_{j  } ' (|z_j|^2) \>
,   \\ &
B_0  :=  B_0' -\Im \<  D_{jA}\overline{q}_{j z_j}, {q}_{j z_j} \>  dz_{jA} .
\end{aligned}
\end{equation}
 with $\widehat{{q}}_{j  } '  (t)= \frac{d}{dt}\widehat{{q}}_{j  } $.
We have the following lemma.

\begin{lemma}
  \label{lem:gau2} We have
    $\gamma _j(|z_j|^2) =\resto ^{2,0}_{\infty,\infty}(|z_j|^2)$.
  We have
  $dB_0 =\Omega_0$ and
\begin{equation}\label{eq:gau21}
\begin{aligned}
 \Phi ^*B _0=B _0  \text{ for $\Phi (u)= e^{\im \vartheta}u$ for any fixed $\vartheta\in \R$}
\end{aligned}
\end{equation}
\end{lemma}
\proof  $\gamma _j(|z_j|^2) =\resto ^{2,0}_{\infty,\infty}(|z_j|^2)$ is elementary
from  Prop. \ref{prop:bddst} and Def. \ref{def:scalSymb}.
$dB_0 =\Omega_0$ follows by  $dB_0 '=\Omega_0'$ and
by
\begin{equation*}
\begin{aligned}& - d\Im \<  D_{jA}\overline{q}_{j z_j}, {q}_{j z_j} \>  dz_{jA} =
   \Im \<  D_{jA}\overline{q}_{j z_j}, D_{jB}{q}_{j z_j} \>  dz_{jA} \wedge dz_{jB} =\\&
2 \Im \<  D_{jR}\overline{q}_{j z_j}, D_{jI}{q}_{j z_j} \>  dz_{jR} \wedge dz_{jI}=2\gamma(|z_j|^2)dz_{jR}\wedge dz_{jI}\\&
  = \im\gamma _j(|z_j|^2)  dz_{j } \wedge d\overline{z}_{j }
\end{aligned}
\end{equation*}
where
$ q_{jz_j } =z_j \widehat{q}_{j }(|z_j|^2)$.

Turning to the proof of \eqref{eq:gau21}, we have
\begin{align*}
\Phi^*\(\im\gamma_j(|z_j|^2)dz_j \wedge d\overline z_j\) = \im\gamma_j(|z_j|^2)d\(\Phi^* z_j\) \wedge d\(\Phi^* \overline z_j\)=\im\gamma_j(|z_j|^2)dz_j \wedge d\overline z_j.
\end{align*}
\qed

\begin{lemma}
  \label{lem:1forms}
  We  have  $dB=\Omega $  with $B$ the differential form   in the manifold $H^1$ defined by
\begin{equation} \label{eq:1forms}\begin{aligned} &
B (u)X:= \Im \langle \overline{u} , X\rangle
\end{aligned}
\end{equation}
 Consider  for $u\in B _{  H^1} (0,d _0)$ for the $d _0>0$  of Lemma \ref{lem:contcoo}
 the function $\psi \in C^\infty (B _{  H^1} (0,d _0) , \R) $   and the differential form $\Gamma(u)$ defined as follows:
 \begin{align} &
\psi (u) :=\sum _{j=1}^{n} \Im \langle \overline{q}_{jz_j } , u\rangle
 +\sum _{j=1}^{n} \Im \left (    \alpha _j[z]\eta  \ \overline{z} _j \right )
\label{eq:psi} \\&   \Gamma(u):=B (u)-B _0(u)
+d\psi (u)  . \label{eq:alpha1}
\end{align}
Then the map $( z, \eta )\to \Gamma(u (z,\eta ))$,  for $u (z,\eta )$ the r.h.s. of
\eqref{eq:systcoo1}, which is
 initially defined in   $  B_{\C^n}(0 , d _0  ) \times ( H^1\cap \mathcal{H}_c[0]) $,
 extends to  $   B_{\C^n}(0 , d _0 ) \times \Sigma ^c _{-r} $ for any $r\in \N$.
In particular,   we have  $\Gamma = \Gamma _{jA}  dz_{jA} + \langle \Gamma _{\eta}, d\eta \rangle  + \langle \Gamma _{\overline{\eta}}, d\overline{\eta} \rangle
   $ with, in the sense of Remark \ref{rem:sym},
\begin{align} \label{eq:alphaest1}
  &          \Gamma _{jA}=   \mathcal{R}^{1,1}_{\infty, \infty}(z,\textbf{Z},\eta )  \text{ and } \Gamma _{\xi }= \mathbf{S} ^{1,1}_{\infty, \infty}(z,\textbf{Z},\eta )  \text{ for $\xi = \eta ,\overline{\eta} $.}
\end{align}
    Furthermore, $\Gamma $ satisfies the  invariance property    in $ B_{H^1}(0, d _0)$:
\begin{equation}\label{eq:gau41}
\begin{aligned}
 \Phi ^*\Gamma=\Gamma  \text{ for $\Phi (u)= e^{\im \vartheta}u$ for any fixed $\vartheta\in \R$}.
\end{aligned}
\end{equation}
\end{lemma}
\proof  By the definition of the exterior differential,
and focusing on constant vectorfields $X$ and $Y$,
\begin{equation*}   \begin{aligned} &
  dB(X,Y)=XB(u)Y-YB(u)X=\Im \langle \overline{X} , Y\rangle  -\Im \langle \overline{Y} , X\rangle   =\Omega (X,Y).
\end{aligned}
\end{equation*}
This is enough to prove $dB=\Omega $.
Next, using $R[z]\eta = \eta +\sum _j\alpha _j [z]\eta  \,  \phi _j  $,
we expand
\begin{equation} \label{eq:1forms1}   \begin{aligned}
B (u) &= \sum _{j}\Im \langle \overline{Q} _{jz_j } ,  \ \rangle  +\Im \langle \overline{R[z]\eta}  ,  \ \rangle = \sum _{j}\Im \langle \overline{z}_j\phi _j  ,  \ \rangle  +\Im \langle \overline{ \eta}  ,  \ \rangle \\& + \sum _{j}\Im \langle \overline{q} _{jz_j }  ,  \ \rangle  +
\sum _{j} \Im   (\overline{\alpha _j [z]\eta}  \ \  \langle \phi _j , \ \rangle )
 \  .  \end{aligned}
\end{equation}
  By the definition of $B_0$ in \eqref{eq:defOm0}   we have
\begin{align} &
B- B_0 = I_1+I_2+I_3 +\sum _{j,A}\Im \langle  D_{jA}\overline{q}_{j z_j}, {q}_{j z_j} \rangle  dz_{jA}   +\sum _{j}\Im \langle \overline{q} _{jz_j }  ,  \ \rangle \, , \label{eq:bad2} \\&
I_1:=
\sum _{j}  \Im \left  [     \overline{z}_{j }   \( \langle  \phi _j  ,  \ \rangle -dz_j \) \right   ] \ , \quad   I_2:=   -\Im   \<  \overline{\eta} , d {\eta}   - P_c   \quad \>
  \ ,
   \quad   I_3:=
\sum _{j} \Im   \left [\overline{\alpha _j [z]\eta}  \   \langle \phi _j , \ \rangle   \right ] \  .\nonumber
\end{align}
We substitute  $ d {\eta}$ with \eqref{eq:etaprime} and
$\langle \phi _j , \ \rangle $ with \eqref{eq:dzR}. For $\alpha _j[z] \circ d\eta$ the linear operator defined by $\alpha _j[z] \circ d\eta (X):= \alpha _j[z]  d\eta (X)$  we then get
 \begin{equation} \label{eq:bad1} \begin{aligned}
  I_1  &= \Im    \langle D_{ jA}q_{jz_j}  ,  \overline{z}_k \phi _k\rangle  dz_{jA }+
  \Im       (  \overline{z}_j D_{ kA}\alpha _j [z]\eta   )   dz_{kA }   +   \Im \left  (   \overline{z}_j \alpha _j [z]\circ d\eta   \right   )  \\& =  \sum  _{jA} \resto ^{1,1}_{\infty,\infty } dz_{jA } +   \Im \left  (   \overline{z}_j \alpha _j [z]\circ d\eta   \right   ),
\end{aligned}
\end{equation}
where, as anticipated in
 Remark \ref{rem:sym},  here we set  $\resto ^{i,j}_{K,M}=\resto ^{i,j}_{K,M} ( z,\textbf{Z},\eta) $ and $\mathbf{{S}}^{i,j}_{K,M}=\mathbf{{S}}^{i,j}_{K,M}  ( z,\textbf{Z},\eta)$, where here $\mathbf{Z} =(z_i \overline{z}_j) _{i,j =1,...,n} $
  with $i\neq j$.

 The  second term in the last line  of the last formula is incorporated in  \eqref{eq:bb0}. We have
 \begin{equation*}
    I_2  = \Im    \langle \overline{ \eta} , D_{ jA}q_{jz_j}  \rangle  dz_{jA }  =\sum  _{jA} \resto ^{2,1}_{\infty,\infty } dz_{jA } .
 \end{equation*}
    Substituting with  \eqref{eq:dzR}  we have
  \begin{equation*}  \begin{aligned} &
  I_3  =  \sum  _{jA} \resto ^{2,1}_{\infty,\infty } dz_{jA }  +
 \langle \mathbf{S} ^{1,1}_{\infty, \infty}, d\eta \rangle  + \langle \mathbf{S} ^{1,1}_{\infty, \infty}, d\overline{\eta} \rangle   .
\end{aligned}
\end{equation*}
Hence we get
\begin{align} & \label{eq:bb0}
B- B_0= \sum _{j} \Im \left  (   \overline{z}_j \alpha _j [z]\circ d\eta   \right   )    \\&
\label{eq:bb1}   +  \sum  _{jA} \resto ^{1,1}_{\infty,\infty } dz_{jA }  +
 \langle \mathbf{S} ^{1,1}_{\infty, \infty}, d\eta \rangle  + \langle \mathbf{S} ^{1,1}_{\infty, \infty}, d\overline{\eta} \rangle
			\\& \label{eq:bb6}  +\sum  _{jA}
    \Im \langle  D_{jA}\overline{q}_{j z_j}, {q}_{j z_j} \rangle  dz_{jA} + \sum _{j} \Im \langle \overline{q} _{jz_j }  ,  \ \rangle \ .
\end{align}
Set now $
  \widetilde{\psi }(u):=-\sum _{j=1}^n \Im \langle \overline{q} _{jz_j }  , u \rangle \  . $
 Then   it is elementary that we have
 \begin{equation} \label{eq:tildepsi} \begin{aligned} &
  d\widetilde{\psi }  =-\sum _{j=1}^n\Im \langle \overline{q} _{jz_j }  ,   \rangle
	 -\sum _{j, A}\Im \langle D_ {jA}\overline{q} _{jz_j }  ,  {q} _{jz_j } \rangle    dz_{j A}
	  + \sum _{j, A} \resto ^{2,1}_{\infty,\infty}  dz_{j A}   .
\end{aligned}
\end{equation}
  By the Leibnitz rule  we have
  \begin{align} & \label{line1} \Im \left  (   \overline{z}_j \alpha _j [z]\circ d\eta   \right   ) = d \Im \left  (   \overline{z}_j \alpha _j [z]\ \eta   \right   ) - \Im \left  (   d(\overline{z}_j \alpha _j [z])\ \eta   \right   ).\end{align}
The contribution to  \eqref{eq:bb0}  of the
last    term   in the r.h.s. of \eqref{line1}  can be absorbed in  \eqref{eq:bb1}.
Then
\begin{align*} &
B-B_0+d\psi=   \sum  _{jA} \resto ^{2,1}_{\infty,\infty } dz_{jA }  +
 \langle \mathbf{S} ^{1,1}_{\infty, \infty}, d\eta \rangle  + \langle \mathbf{S} ^{1,1}_{\infty, \infty}, d\overline{\eta} \rangle .
\end{align*}
Here we have used:    the first two terms in the r.h.s. of \eqref{eq:tildepsi} cancel  with   \eqref{eq:bb6};  there is a cancelation between the  contribution to
 the
r.h.s. of  \eqref{eq:bb0}  of the first term on the r.h.s. of  \eqref{line1}
and the differential of the last   term  in   \eqref{eq:psi}.  This yields \eqref{eq:alphaest1}.

Finally we consider \eqref{eq:gau41}. We have $\Phi ^*B_0=B_0$ by  \eqref{eq:gau21}, while  $\Phi ^*B =B $ follows immediately from the definition
of $B$ in  \eqref{eq:1forms}. Finally $\Phi ^*\psi  =\psi $  follows immediately from  $ \Phi ^* \langle \overline{q}_{jz_j } , u\rangle = \langle \overline{q}_{jz_j } , u\rangle$, which follows from ${q}_{jz_j }(e^{\im \vartheta}z)=e^{\im \vartheta} {q}_{jz_j }(z)$, and  from \eqref{eq:gau1} and \eqref{eq:detion31} which impy
$$ \Phi ^*    \left ( \overline{z}_{j }   {\alpha _j[z]\eta}  \right )=   e^{-\im \vartheta}\overline{z}_{j }   {\alpha _j[e^{\im \vartheta}z]e^{\im \vartheta}\eta}     =    \overline{z}_{j }   {\alpha _j[ z]\eta}    .  $$

\qed

 \begin{lemma}
  \label{lem:Omegahat}  Consider the differential form $ \Omega -\Omega _{0}$, which is defined in $ B_{H^1}(0, d _0)$ for the  $d _0>0$ of Lemma \ref{lem:contcoo}.
 Then, summing on repeated indexes, we have
\begin{equation}\label{eq:Omegahat1}
\begin{aligned}
     &\Omega -\Omega _{0}   =
 \widetilde{\Omega}_{ijAB} dz_{iA} \wedge dz_{jB}  + \sum _{\xi = \eta , \overline{\eta}} dz_{iA}  \wedge \langle 	\widetilde{\Omega}  _{iA\xi}, d\xi \rangle +   \sum _{\xi , \xi '= \eta , \overline{\eta}} \langle 	
  \widetilde{{\Omega}}  _{ \xi ' \xi  }   d\xi  , d\xi '\rangle
\end{aligned}
\end{equation}
where,  expressed   as functions  of $(z,\eta )$, the coefficients
extend into functions defined $B_{\C ^n}(0, d _0 ) \times \Sigma ^{c}_{-r}$
for any $r\in \N$ and in particular we have
  $ \widehat{\Omega}  _{iA \xi} =\textbf{S}^{1,0}_{\infty, \infty }( z,\textbf{Z},\eta)   $,
 $ \widehat{\Omega}  _{ijAB} =\mathcal{R}^{1,0}_{\infty, \infty } ( z,\textbf{Z},\eta)   $
 in the sense of Remark \ref{rem:sym}  and $ \widetilde{{\Omega}}  _{ \xi ' \xi  } = \partial _{\xi }\textbf{S}^{1,1}_{\infty, \infty }( z,\textbf{Z},\eta) -(\partial _{\xi '}\textbf{S}^{1,1}_{\infty, \infty }( z,\textbf{Z},\eta))^*     $ (with two distinct $\mathbf{S}$'s).
We furthermore have
\begin{equation}\label{eq:gau31}
\begin{aligned}
 \Phi ^*(\Omega -\Omega _{0})=\Omega -\Omega _{0}  \text{ for $\Phi (z,\eta)= (e^{\im \vartheta}z,e^{\im \vartheta}\eta)$ for any fixed $\vartheta\in \R$}.
\end{aligned}
\end{equation}

\end{lemma}
\proof   We have
\begin{equation*}
    \Omega -\Omega _{0} =d\Gamma =d  \sum _{j,A}\resto ^{1,1}_{\infty, \infty }d z_{jA} +  d  \sum _{\xi } \langle \textbf{S}^{1,1}_{\infty, \infty }, d\xi \rangle  .
\end{equation*}
Summing over  $k,B,\xi$ we have
 \begin{equation*}
     d   (\resto ^{1,1}_{\infty, \infty }d z_{jA}) = \partial _{z_{kB}}\resto ^{1,1}_{\infty, \infty }   dz_{kB} \wedge dz_{jA}  + \langle \partial _\xi  \resto ^{1,1}_{\infty, \infty } , d\xi \rangle \wedge  dz_{jA}
\end{equation*}
 with the $\partial _\xi  \resto ^{1,1}_{\infty, \infty } \in \mathcal{H}_c[0]$
  defined, summing on repeated indexes and for $F$ with values in $\R$,  by
 \begin{equation*}
     d   F    X= \partial _{z_{kB}} F \  d z_{ {kB}} X  + \langle \partial _\xi F  ,  d \xi  X
     \rangle   \text{  for any $X\in L^2(\R ^3 , \C )$.}
\end{equation*}
 It is easy to see that $\partial _\xi  \resto ^{1,1}_{\infty, \infty }= \textbf{S}^{1,0}_{\infty, \infty } $  and $\partial _{z_{kB}}\resto ^{1,1}_{\infty, \infty }  =  \resto ^{1,0}_{\infty, \infty }  $ .

 \noindent Furthermore, summing on repeated indexes  we have
 \begin{equation} \label{eq:Omegahat2} \begin{aligned}
      d\langle \textbf{S}^{1,1}_{\infty, \infty }, d\xi \rangle  & =  dz_{kB} \wedge \langle \partial _{z_{kB}} \textbf{S}^{1,1}_{\infty, \infty }, d\xi \rangle  +  \langle \partial _{\xi '}  \textbf{S}^{1,1}_{\infty, \infty }d\xi ' , d\xi \rangle  - \langle d\xi , \partial _{\xi '}  \textbf{S}^{1,1}_{\infty, \infty }d\xi '\rangle  \\& =  dz_{kB} \wedge \langle \partial _{z_{kB}} \textbf{S}^{1,1}_{\infty, \infty }, d\xi \rangle   +  \langle \partial _{\xi '}  \textbf{S}^{1,1}_{\infty, \infty }d\xi ' , d\xi \rangle  - \langle (\partial _{\xi '}  \textbf{S}^{1,1}_{\infty, \infty })^* d\xi , d\xi '\rangle    , \end{aligned}
\end{equation}
where, for $T\in C^1(U _{L^2}, L^2)$ for $U _{L^2}$ open subset in $ L^2$,
$\partial _\xi T \in B(\mathcal{H}_c[0],L^2)$
is defined by
 \begin{equation*}
     d   T    X= \partial _{z_{kB}} T \  d z_{ {kB}} X  +   \partial _\xi T     d \xi  X
        \text{  for any $X\in L^2(\R ^3 , \C )$.}
\end{equation*}
Summing on $\xi$ in \eqref{eq:Omegahat2} we get terms which are absorbed in the last two terms of \eqref{eq:Omegahat1}.

\noindent Formula  \eqref{eq:gau31} follows from \eqref{eq:gau41},  $\Omega _0=dB _0$ and $\Omega  =dB $.
\qed

\begin{lemma}
  \label{lem:vectorfield0}   Consider the form  $\Omega _t:=\Omega _0+t(\Omega -\Omega _0)$
  and set $i_X\Omega _t (Y):=\Omega _t (X,Y)$. For any preassigned $r\in \N$ recall by \eqref{eq:defOm0pr}, \eqref{eq:defOm0} and  Lemmas   \ref{lem:1forms} and \ref{lem:Omegahat} that  $(\Omega -\Omega _0)$ and $\Gamma $  extend to forms defined in $    B_{\C^n}(0 , d _0 ) \times \Sigma ^c _{-r} $. Then there is   $\delta _0 \in (0, d_0) $ s.t.
  for any   $(t, z, \eta )\in (-4, 4)\times B_{\C ^n} (0,\delta _0 )
    \times   B_{\Sigma _{-r}^c} (0,\delta _0 )$    there exists
  exactly one solution
    $\mathcal{X}^t  (z, \eta )\in L^2$ of the  equation     $
 i_{\mathcal{X}^t} \Omega _t=-  \Gamma
   $.
     Furthermore,  we have  the following facts.

\begin{itemize}
\item[(1)]
     $\mathcal{X}^t  (z, \eta )\in \Sigma _{ r}$ and if we set $\mathcal{X}^t _{jA} (z, \eta )=dz_{jA}\mathcal{X}^t  (z, \eta )$ and  $\mathcal{X}^t _{\eta} (z, \eta )=d\eta \mathcal{X}^t  (z, \eta )$,   we have
  $\mathcal{X}^t _{jA} (z, \eta )=\mathcal{R}^{1,1}_{r,\infty }(t,z,\textbf{Z}, \eta )$ and
  $ \mathcal{X}^t _{\eta} (z, \eta )=\textbf{S}^{1,1}_{r,\infty}(t,z,\textbf{Z},\eta )$ in the sense of Remark \ref{rem:sym}.

  \item[(2)]   For $\mathcal{X}^t _{j }:=dz_j\mathcal{X}^t$ and $\mathcal{X}^t _{\eta }:=d\eta \mathcal{X}^t$, we have
  $\mathcal{X}^t _{j } (e^{\im \vartheta}z, e^{\im \vartheta}\eta) = e^{\im \vartheta}\mathcal{X}^t _{j } (z, \eta)$ and  $\mathcal{X}^t _{\eta } (e^{\im \vartheta}z, e^{\im \vartheta}\eta) = e^{\im \vartheta}\mathcal{X}^t _{\eta } (z, \eta)$.

 \end{itemize}

\end{lemma}
\proof  We define $Y$ such that $i_Y \Omega '_0= -\Gamma$, which yields  $Y _{jR}=-\frac 12 \Gamma  _{jI}
$,      $Y _{jI}= \frac 12 \Gamma  _{jR}
$  (both $\resto ^{1,1}_{\infty,\infty}$),    $Y _{\eta }=  -\im \Gamma _{\overline{\eta} }  $
and    $Y _{\overline{\eta} }= \im \Gamma _{ {\eta} }    $  (both $\mathbf{S} ^{1,1}_{\infty,\infty}$) . We use  $i_{K_tX} \Omega '_0=    i_X (  \Omega  _0- \Omega '_0+ t\widehat{ \Omega}    ) $, where $\widehat {\Omega}:=\Omega-\Omega_0$, to   define
in $L^2
$    the operator $K_t$ .
 We claim the following lemma.

 \begin{lemma}
  \label{lem:kt} For appropriate symbols    $\resto ^{1,0}_{\infty,\infty}(t,z,\textbf{Z}, \eta ) $ and $\mathbf{S} ^{1,0}_{\infty,\infty}(t,z,\textbf{Z}, \eta )$ which differ from one term to the other and for   $\mathbf{Z} =(z_i \overline{z}_j) _{i,j =1,...,n} $
  with $i\neq j$,  we have
\begin{equation} \label{eq:kt0}\begin{aligned} &  (K_tX) _{jA} =\sum  _{ l B} \resto ^{1,0}_{\infty,\infty}   X  _{ l B} +\sum _{\xi = \eta ,\overline{\eta}} \langle \mathbf{S} ^{1,0}_{\infty,\infty} ,   X  _{\xi }  \rangle ,  \\&   (K_tX) _{\xi } =\sum  _{ l B} \mathbf{S} ^{1,0}_{\infty ,\infty}   X  _{ l B} +\sum _{\xi ' = \eta ,\overline{\eta}}  ( \partial _{\xi  ' }\textbf{S}^{1,1}_{\infty, \infty }( t,z,\textbf{Z},\eta) -(\partial _{\xi }\textbf{S}^{1,1}_{\infty, \infty }( t,z,\textbf{Z},\eta))^*  X  _{\xi '}  .
\end{aligned}\end{equation}
  \end{lemma}
 We assume  for a moment Lemma \ref{lem:kt} and complete the proof of  Lemma \ref{lem:vectorfield0}.
$
 i_{\mathcal{X}^t} \Omega _t=-  \Gamma
   $ becomes $ \mathcal{X}^t +K_t \mathcal{X}^t =Y  .$
Indeed, suppose $ \mathcal{X}^t +K_t \mathcal{X}^t =Y  $ holds.
Then, by definition of $K_t$, we have
\begin{equation*}
i_{\mathcal X_t}(\Omega_t-\Omega') = i_{K_t \mathcal X_t}\Omega_0' \text{  and so }
i_{\mathcal X _t} \Omega_t = i_{\mathcal X _t}\Omega_0' + i_{K_t \mathcal X_t}\Omega_0' =-\Gamma.
\end{equation*}
   By Lemma \ref{lem:kt}, in coordinates and for $\xi =\eta ,\overline{\eta} $  the last equation is  schematically  of the form
\begin{align}     & \mathcal{X}^t  _{jA}   +  \sum  _{\ell ,B}\resto ^{1,0}_{r, \infty}   \mathcal{X}^t  _{\ell B}       + \sum _{\xi = \eta , \overline{\eta}} \langle   \textbf{S} ^{1,1}_{r, \infty} , \mathcal{X}^t _\xi \rangle = \resto ^{1,1}_{r, \infty} \label{eq:kt010} \\&    \mathcal{X}^t _{\xi }   + \sum  _{\ell B} \mathbf{S} ^{1,0}_{r,\infty}   \mathcal{X}^t  _{\ell B} +\sum _{\xi ' = \eta ,\overline{\eta}}( \partial _{\xi  ' }\textbf{S}^{1,1}_{\infty, \infty }( t,z,\textbf{Z},\eta) -(\partial _{\xi }\textbf{S}^{1, 1}_{\infty, \infty }( t,z,\textbf{Z},\eta))^*   \mathcal{X}^t  _{\xi '}  =\mathbf{S} ^{1,1}_{r,\infty}   .    \nonumber
\end{align}
Notice that $ (\partial _{\xi }\textbf{S}^{1, 1}_{\infty, \infty })  \mathbf{S} ^{1,1}_{r,\infty} $ is $C^\infty$  in $( t,z,\textbf{Z},\eta)$ with values in $ \Sigma _r$. We have
\begin{equation*}   \begin{aligned} & \| (\partial _{\xi }\textbf{S}^{1, 1}_{\infty, \infty })  \mathbf{S} ^{1,1}_{r,\infty} \| _{\Sigma _r} \le   \| \partial _{\xi }\textbf{S}^{1, 1}_{\infty, \infty }   \| _{B(  \Sigma  _{-r} ,\Sigma _r)}  \|   \mathbf{S} ^{1,1}_{r,\infty} \| _{\Sigma _r}
   .
\end{aligned}
\end{equation*}
By  \eqref{eq:opSymb} we have $\partial _{\xi }\textbf{S}^{1, 1}_{\infty, \infty } ( t,0,0,0) $. This implies
\begin{equation}\label{eq:kt011}
  \| \partial _{\xi }\textbf{S}^{1, 1}_{\infty, \infty }   \| _{B(  \Sigma  _{-r} ,\Sigma _r)}\le C  \|  \eta \| _{\Sigma   _{-K}}+ |\mathbf{Z}   |+|z |
\end{equation}
  and so
\begin{equation*}   \begin{aligned} & \| (\partial _{\xi }\textbf{S}^{1, 1}_{\infty, \infty })  \mathbf{S} ^{1,1}_{r,\infty} \| _{\Sigma _r} \le  C (\|  \eta \| _{\Sigma   _{-K}}+|\textbf{Z} |)  (\|  \eta \| _{\Sigma   _{-K}}+ |\mathbf{Z}   |+|z |)^{i2}
   .
\end{aligned}
\end{equation*}
So $(\partial _{\xi }\textbf{S}^{1, 1}_{\infty, \infty })  \mathbf{S} ^{1,1}_{r,\infty} = \mathbf{S} ^{2,1}_{r,\infty}. $

\noindent Inequality \eqref{eq:kt011}, a Neumann  expansion and formulas \eqref{eq:algSymb}
yield claim (1) in Lemma \ref{lem:vectorfield0}.

\noindent Claim (2) in Lemma \ref{lem:vectorfield0} follows from
$$ i _{\Phi  _*^{-1}\mathcal{X}^t}\Phi ^*\Omega _t=-\Phi ^*\Gamma =- \Gamma = i _{ \mathcal{X}^t} \Omega _t = i _{\Phi  _*^{-1}\mathcal{X}^t} \Omega _t , $$
where $\Phi ^*\Gamma =  \Gamma$ is \eqref{eq:gau41} and we use \eqref{eq:gau21} and   \eqref{eq:gau31} to conclude  $\Phi ^*\Omega _t= \Omega _t $. Then $\Phi  _*^{-1}\mathcal{X}^t=\mathcal{X}^t$, which is equivalent to  $\Phi _*\mathcal{X}^t=\mathcal{X}^t.$  For the other formulas in claim (2) we have for instance
\begin{equation*}
\begin{aligned}     &  \mathcal{X}^t_j (e^{\im \vartheta}z, e^{\im \vartheta}\eta) = \mathcal{X}^t_j (\Phi (u) ) = dz_j(\mathcal{X}^t  (\Phi (u) )) = dz_j(\Phi _*\mathcal{X}^t  ( u ))\\&  =d(z_j \circ \Phi ) ( \mathcal{X}^t ( u )) = e^{\im \vartheta} \mathcal{X}^t_j ( u ) .
\end{aligned}
\end{equation*}
This ends the proof of Lemma \ref{lem:vectorfield0}, assuming Lemma \ref{lem:kt}.

\qed

\noindent \textit{Proof of Lemma \ref{lem:kt}.}
 By \eqref{eq:defOm0}   and summing over the indexes $(j,A ,B )$ we can write
\begin{equation} \label{eq:pert1}
\begin{aligned}    \Omega_0  -  \Omega_0 '&=    \resto ^{4,0}_{\infty,\infty}
   dz_{jA} \wedge dz_{jB}   \Rightarrow   i_X(\Omega_0  -  \Omega_0 ')=   \resto ^{4,0}_{\infty,\infty} X_{jR} dz_{jI} +
\resto ^{4,0}_{\infty,\infty} X_{jI} dz_{jR} .
\end{aligned}
\end{equation}
  So if we define $K'X$ setting $ i_{K'X}\Omega_0 '=i_X(\Omega_0  -  \Omega_0 ')$, by comparing \eqref{eq:pert1} with
\begin{equation*}
\begin{aligned}     &    i_{K'X}\Omega_0 '=   2 (K'X)_{jR} dz_{jI} -
 2 (K'X)_{jI} dz_{jR}  + \im
 \langle  (K'X)_{\eta} , X _{\overline{\eta}} \rangle  - \im
 \langle  (K'X)_{\overline{\eta}} , X _{ {\eta}} \rangle ,
\end{aligned}
\end{equation*}
we obtain
\begin{equation} \label{eq:kt1}\begin{aligned} &  (K'X) _{jA} =  \resto ^{4,0}_{\infty,\infty}   X  _{j A}   \text{ and }   (K'X) _{\xi } =0 \text{ for $\xi = \eta ,\overline{\eta}$.}
\end{aligned}\end{equation}
Summing on $(j,l,A,B,\xi ,\xi ')$ we have
\begin{equation*}
\begin{aligned}     &  t \widehat{\Omega} =    \resto ^{1,0}_{\infty,\infty}
   dz_{jA} \wedge dz_{l B}   + dz_{jA} \wedge    \langle  \mathbf{S} ^{1,0}_{\infty,\infty} , d\xi \rangle  + t \langle  [\partial _{\xi }\textbf{S}^{1,1}_{\infty, \infty }( z,\textbf{Z},\eta) -(\partial _{\xi '}\textbf{S}^{1,1}_{\infty, \infty }( z,\textbf{Z},\eta))^*] d\xi   , d\xi '\rangle  .
\end{aligned}
\end{equation*}
Hence
\begin{equation*}
\begin{aligned} \label{eq:K1}    &   t\  i_X\widehat{\Omega}=   \resto ^{1,0}_{\infty,\infty} X_{jA} dz_{ l B} +  \langle  \mathbf{S} ^{1,0}_{\infty,\infty} , X_\xi  \rangle   dz_{jA}        +X_{jA}     \langle  \mathbf{S} ^{1,0}_{\infty,\infty} , d\xi \rangle  +    \langle  [\partial _{\xi }\textbf{S}^{1,1}_{\infty, \infty } -(\partial _{\xi '}\textbf{S}^{1,1}_{\infty, \infty } )^*] X_\xi   , d\xi '\rangle .
\end{aligned}
\end{equation*}
So, if we define $K ^{\prime\prime}X$ setting $ i_{K ^{\prime\prime}X}\Omega_0 '=t\  i_X\widehat{\Omega}$,
we obtain
\begin{equation} \label{eq:kt2}\begin{aligned} &  (K ^{\prime\prime}X) _{jA} =\sum  _{\ell B} \resto ^{1,0}_{\infty,\infty}   X  _{\ell B} +\sum _{\xi = \eta ,\overline{\eta}} \langle \mathbf{S} ^{1,0}_{\infty,\infty} ,   X  _{\xi }  \rangle ,  \\&   (K ^{\prime\prime}X) _{\xi } =\sum  _{l B} \mathbf{S} ^{1,0}_{\infty,\infty}   X  _{l B} + \sum _{\xi = \eta ,\overline{\eta}} [\partial _{\xi '}\textbf{S}^{1,1}_{\infty, \infty } -(\partial _{\xi }\textbf{S}^{1,1}_{\infty, \infty } )^*] X_{\xi '}  .
\end{aligned}\end{equation}
Since $K_t =K ^{\prime } +K ^{\prime\prime}$, summing up \eqref{eq:kt1} and \eqref{eq:kt2}
we get \eqref{eq:kt0}  and so    Lemma \ref{lem:kt}.

\qed

\noindent Having established that  $\mathcal{X}^t(z,\eta )$ has components
which are restrictions of symbols as in Definitions \ref{def:scalSymb} and  \ref{def:opSymb}
 we  have the following   result.

\begin{lemma}
  \label{lem:darflow0}    Fix an $r\in \N$ and for the   $\delta _0>0$  and the $\mathcal{X}^t  (z, \eta ) $
   of  Lemma \ref{lem:vectorfield0},   consider
         the   following system, which  is well defined in
     $(t, z, \eta )\in (-4, 4)\times B_{\C ^n} (0,\delta _0 )
    \times   B_{\Sigma _{ k}^c} (0,\delta _0 )$  for all  $k\in \Z\cap [-r,r]$:
\begin{equation}\label{eq:syst1}
\begin{aligned}
  &    \dot z_{j } = \mathcal{X}^t _{j }  (z, \eta ) \text{  and }     \dot \eta = \mathcal{X}^t _{\eta}  (z, \eta ) .
     \end{aligned}
\end{equation}
Then the following facts hold.

\begin{itemize}
\item[(1)]  For   $\delta _1\in (0,\delta _0 ) $  sufficiently small system \eqref{eq:syst1}
generates flows
\begin{align} \nonumber
  &   \mathfrak{F}^t \in C ^\infty ((-2, 2)\times B_{\C ^n} (0,\delta _1 )
    \times   B_{\Sigma _{ k}^c} (0,\delta _1 ) ,B_{\C ^n} (0,\delta _0 )
    \times   B_{\Sigma _{ k}^c} (0,\delta _0 )) \text{ for all $k\in \Z\cap [-r,r]$} \\&  \mathfrak{F}^t \in C ^\infty ((-2, 2)\times B_{\C ^n} (0,\delta _1 )
    \times   B_{H^1 \cap \mathcal{H}_{c}[0]} (0,\delta _1 ) ,B_{\C ^n} (0,\delta _0 )
    \times   B_{H^1 \cap \mathcal{H}_{c}[0]} (0,\delta _0 ) ) .\label{eq:syst122}
     \end{align}
 In particular   for $ z_{j  }  ^{t} :=z_j\circ \mathfrak{F}^t (z,\eta )$ and $ \eta   ^{t} :=\eta \circ \mathfrak{F}^t (z,\eta )$ we have
\begin{equation}\label{eq:syst11}
\begin{aligned}
  &      z_{j  }  ^{t}= z_{j }+ S _{j  }  (t,z, \eta ) \text{  and }       \eta  ^t= \eta+ S _{\eta}  (t,z, \eta )\end{aligned}
\end{equation}
with
  $S _{j  } (t,z, \eta )=\mathcal{R}^{1,1}_{r, \infty}(t,z,\textbf{Z}, \eta )$   and
  $ S _{\eta}  (t,z, \eta )=\textbf{S}^{1,1}_{r, \infty}(t,z,\textbf{Z}, \eta )$ in the sense
  of Remark \ref{rem:sym}.

 \item[(2)]       $\mathfrak{F}=\mathfrak{F}^1$ is a local diffeomorphism
 of $H^1$ into itself near  the origin s.t. $\mathfrak{F}^*\Omega =\Omega _0$.

  \item[(3)]      We have  $S _{j }  (t,e^{\im \vartheta}z,e^{\im \vartheta} \eta )=e^{\im \vartheta}S _{j }  (t,z, \eta )$,
 $S _{\eta }  (t,e^{\im \vartheta}z,e^{\im \vartheta} \eta )=e^{\im \vartheta}S _{\eta }  (t,z, \eta )$.
 \end{itemize}

\end{lemma}
 \proof  The first sentence has been established in Lemma \ref{lem:vectorfield0}.
 Elementary theory of ODE's yields \eqref{eq:syst122}. The rest of  claim  (1) is a special case of a more general result,   see {Lemma}
  \ref{lem:flow1} below.  We  get
 claim  (2)   by the classical formula, for $L_X$ the Lie derivative,  \begin{equation}\label{eq:fdarboux}  \begin{aligned} & \partial _t(
 \mathfrak{F}^{ t  * } \Omega  _t   )      =   \mathfrak{F}^{ t  * }  (L_{\mathcal{ X} ^{ t   }}\Omega  _t +\partial _t\Omega  _t   )     =
 \mathfrak{F}^{ t  * }
  ( d i_{\mathcal{X}  ^{ t   }}\Omega  _t +d\Gamma    )
=0  .\end{aligned} \end{equation}
 Notice that   \eqref{eq:fdarboux} is well defined here,
while it has no clear meaning  for the NLS with translation treated in  \cite{Cu0,Cu3}, where
   the flows $\mathfrak{F}^{ t    }$ are not differentiable (see \cite{Cu0} for a rigorous argument on how to offset this problem).
The symmetry in  claim  (3) is elementary and we skip it.

  \qed

\begin{lemma}
  \label{lem:flow1}   Consider   a system
   \begin{equation}\label{eq:syst23}
\begin{aligned}
  &      \text{ $ \dot z_{j } = X_j (t,z, \eta ) $  and  $ \dot \eta = X _{\eta}  (t, z, \eta )$,}\end{aligned}
\end{equation}
   where $X_j  =\resto ^{a,b}_{r,m}(t,z, \mathbf{Z},\eta )$ $\forall $ $j$  and   $X_\eta  =\mathbf{S} ^{c,d}_{r,m}(t,z, \mathbf{Z},\eta )$, for fixed pairs  $(r,m)$, $(a,b) $   and   $(c,d) $. Assume $m, b, d\ge 1$, with possibly $m=\infty$,  and $r\ge 0$.  Then for the
  flow
  $(z^t,\eta^t )=\mathfrak{F}^t(z ,\eta  )$  we have
\begin{equation}\label{eq:syst21}
\begin{aligned}
  &      z_{j  }  ^{t}= z_{j }+ S _{j  }  (t,z, \eta ) \text{  and }       \eta  ^t= \eta+ S _{\eta}  (t,z, \eta ) \end{aligned}
\end{equation}
for appropriate functions   $S _{j }= \resto ^{a,b}_{r,m}(t,z, \mathbf{Z},\eta )  $  and
    $S _{\eta}=\mathbf{S} ^{c,d}_{r,m}(t,z, \mathbf{Z},\eta )$ in the sense of Remark \ref{rem:sym}.
  \end{lemma}
\proof Consider the vectors $ {\mathbf{Z}}=(z_i  \overline{z}_j) _{i,j=1,...,n}$ with $i\neq j$. Notice that
$\dot {\mathbf{Z}} = \resto ^{a+1,b}_{r,m}(t,z, \mathbf{Z},\eta )$, and this equation can be extended to a whole neighborhood of 0 in the space $L$. Pairing  the latter equation with
equations \eqref{eq:syst21},  a system remains defined  which has a flow $\mathfrak{F}^t(z ,\mathbf{Z}, \eta  )$ which is $C^m$ in $(t,z ,\mathbf{Z}, \eta)$ and which reduces to the flow in \eqref{eq:syst23} when we restrict to vectors   $ {\mathbf{Z}}\in \{ (z_i  \overline{z}_j) _{i,j=1,...,n}: i\neq j \}$, by construction.  The inequalities \eqref{eq:scalSymb} and  \eqref{eq:opSymb}, required to prove  $S _{j }= \resto ^{a,b}_{r,m}   $  and
    $S _{\eta}=\mathbf{S} ^{c,d}_{r,m} $,
can be obtained as follows. We have    for all $|k|\le r $
\begin{align}
  &   \nonumber   | z^t-z|\le \int _0^ t |\resto ^{a ,b}_{r,m}(s,z^s, \mathbf{Z}^s,\eta  ^s)|   ds
  \le C  \int _0^ t   (\|  \eta ^s\| _{\Sigma   _{-r}}+|\textbf{Z} ^s|)^b (\|  \eta ^s\| _{\Sigma   _{-r}}+|\textbf{Z} ^s|+|z^s |)^{a}  ds, \\&  \nonumber  \| \eta ^t-\eta \| _{\Sigma _{k}}\le \int _0^ t \|\mathbf{S} ^{c,d}_{r,m}(s,z^s, \mathbf{Z}^s,\eta  ^s)\|   _{\Sigma   _{k}}   ds
  \le C  \int _0^ t   (\|  \eta ^s\| _{\Sigma   _{-r}}+|\textbf{Z} ^s|)^d (\|  \eta ^s\| _{\Sigma   _{-r}}+|\textbf{Z} ^s|+|z^s |)^{c}  ds \\&  | \mathbf{Z}^t-\mathbf{Z}|\le \int _0^ t |\resto ^{a ,b}_{r,m}(s,z^s, \mathbf{Z}^s,\eta  ^s)|   ds
  \le C  \int _0^ t   (\|  \eta ^s\| _{\Sigma   _{-r}}+|\textbf{Z} ^s|)^b (\|  \eta ^s\| _{\Sigma   _{-r}}+|\textbf{Z} ^s|+|z^s |)^{a+1}  ds. \label{eq:syst33}\end{align}
By Gronwall inequality we get that $|\mathbf{Z}^t|  $  and    $\| \eta ^t  \| _{\Sigma _{-r }}  $   are bounded by $C (  |\mathbf{Z}   |+\|  \eta \| _{\Sigma _{-r}})$.   Plugging
this in the r.h.s. of \eqref{eq:syst33}, we obtain  the last part of the statement.

\qed

	We  discuss  the
	pullback of the energy $E$ by the map $\mathfrak{F}:=\mathfrak{F}^1  $ in claim (2) of  Lemma \ref{lem:darflow0}.  We set $H_2(z,\eta ) = \sum _{j=1}^n
	e_j |z_j|^2+  \langle   H  \eta,  \overline{ \eta} \rangle$.  Our first preliminary result
is the following one.

\begin{lemma}
  \label{lem:KExp1}    Consider the $\delta _1>0$ of Lemma \ref{lem:darflow0}, the   $\delta _0>0$ of Lemma \ref{lem:vectorfield0}
   and set      $r=r_0$ with  $r_0 $  the index   in Lemma \ref{lem:EnExp}.   Then for the map $\mathfrak{F}   $ in claim (2) of  Lemma \ref{lem:darflow0} we have
   \begin{equation}\label{eq:KExp11}
\begin{aligned} &
  \mathfrak{F}(B_{\C^n}(0 , \delta  _1) \times ( B_{H^1}(0 , \delta  _1)\cap \mathcal{H}_c[0]) )\subset  B_{\C^n}(0 , \delta   _0) \times ( B_{H^1}(0 , \delta  _0)\cap \mathcal{H}_c[0])
\end{aligned}
\end{equation}
   and $ \mathfrak{F}|_{B_{\C^n}(0 , \delta  _1) \times ( B_{H^1}(0 , \delta  _1)\cap \mathcal{H}_c[0])} $ is a diffeomorphism between domain and   an open neighborhood of the origin in
   $\C^n \times (H^1 \cap \mathcal{H}_c[0])$.
Furthermore, the functional  $K:=E\circ   \mathfrak{F} $ admits an expansion
\begin{align} \nonumber
  & K (z,\eta )=H_2(z,\eta )  + \sum _{j=1,...,n} \lambda _j(|z_j| ^{2}  )
 \\&   \nonumber  +
	 	 \sum _{l=0}^{2N+3 } \sum _{
		  |\textbf{m}|=   l+1 }  \textbf{Z}^{\textbf{m}}   a_{  \textbf{m} }^{(1)}(|z _{1}|^2,..., |z _{n}|^2 )+ \sum _{j =1}^n \sum _{l=0}^{ 2N+3 }	\sum _{   |\textbf{m}  |  =l} ( \overline{z}_j \textbf{Z}^{\textbf{m}}  \langle
  G_{j \textbf{m}}^{(1)}(|z_j|^2  ), \eta \rangle  +   c.c. )
		\\&  \nonumber
	 	  +     \resto ^{1,  2  }_{{r_1}, \infty} (z , \eta  )  +
	 	  \resto ^{0,  2N+5 }_{{r_1}, \infty} (z,\textbf{Z}  ,\eta )+   \Re \langle
  \textbf{S} ^{0,   2N+4 }_{{r_1}, \infty} (z,\textbf{Z},\eta ) , \overline{\eta} \rangle
		\\& +  \sum _{ i+j  = 2}   \sum _{
		  |\textbf{m}|\le 1 }  \textbf{Z}^{\textbf{m}}    \langle G_{2\textbf{m} ij } ^{(1)}( z,\eta  ),   \eta ^{   i} \overline{\eta} ^j\rangle
		+     \sum _{ d+c  = 3}   \sum _{ i+j  = d}   \langle G_{d ij }^{(1)} ( z ),   \eta ^{   i} \overline{\eta} ^j\rangle  \resto ^{0,  c }_{r, \infty} (z, \eta )      + E _P( \eta) \  , \label{eq:enexp10}
  \end{align}
where: $r_1=r_0-2$;  $G_{j \textbf{m}}^{(1)}$,     $G^{(1)}_{2\textbf{m} ij }    $   and    $G^{(1)}_{d ij }  $ are   $\textbf{S} ^{0,  0} _{{r_1}, \infty }  $;   $a_{  \textbf{m}}^{(1)}(|z _{1}|^2,..., |z _{n}|^2 )  =\resto ^{0,  0} _{\infty, \infty  }  (z) $;     c.c. means complex conjugate;  $\lambda _j(|z_j| ^{2}  )
=\resto ^{2,  0  }_{\infty, \infty} (|z_j| ^{2}  )$. For $|\textbf{m}|=0$,  $G^{(1)} _{2\textbf{m} ij }( z,\eta  ) = G_{2\textbf{m} ij }( z   ) $ is the same of \eqref{eq:b02}.
Finally, we have the invariance  $\resto ^{1,  2  }_{r_{1} , \infty} (e^{\im \vartheta} z, e^{\im \vartheta} \eta )\equiv \resto ^{1,  2  }_{r_{1} , \infty} (  z,   \eta )  $.

\end{lemma}

\proof Consider the expansion  \eqref{eq:enexp1}  for   $E(u(z',\eta '))$, and substitute  the  formulas   $  z_{j }  '= z_{j }+ S _{j }  ( z, \eta ) $   and  $   \eta  '= \eta+ S _{\eta}  ( z, \eta )$, with $ S _{\ell }  ( z, \eta ) =S _{\ell }  (1,z, \eta )$  for $\ell =j, \overline{j}, \eta , \overline{\eta }$, with $ S _{\overline{\ell} }    =\overline{S} _{\ell }   $.  By  $S _{j  } ( z, \eta )=\mathcal{R}^{1,1}_{{r_0}, \infty}( z,\textbf{Z}, \eta )$   and
  $ S _{\eta}  ( z, \eta )=\textbf{S}^{1,1}_{{r_0}, \infty}( z,\textbf{Z}, \eta )$   it is elementary to see that the last three lines of
   \eqref{eq:enexp1} yield terms that can be absorbed in   last three lines of
   \eqref{eq:enexp10} (with $l\ge 1$ in the 2nd line). Notice that the   $z$ dependence of the   $ a_{  \textbf{m} }^{(1)}$ in terms of $(|z _{1}|^2,..., |z _{n}|^2 )$ follows by Lemmas \ref{lem:darflow0} and \ref{lem:expandzzj}.
    The  $z$ dependence of the   $ G_{ j \textbf{m} }^{(1)}$ is obtained by Lemma \ref{lem:expandzzS}.
     Notice also that if an $\mathcal{R}^{i,0}_{{r }, \infty}( z  )$
depends only on $z$, then it is an $\mathcal{R}^{i,0}_{\infty, \infty}( z  )$.

We have  $ \resto ^{1,  2  }_{{r_0}, \infty } (z', \eta ' )  =  \resto ^{1,  2  }_{{r_0}, \infty } (z , \mathbf{Z}, \eta   )$.  Notice that  by the    invariance  of
   $ \resto ^{1,  2  }_{{r_0}, \infty } (z , \eta   )$ and by claim (3) in Lemma \ref{lem:darflow0} we have
  $\resto ^{1,  2  }_{r_{0} , \infty} (e^{\im \vartheta} z, \mathbf{Z}, e^{\im \vartheta} \eta )\equiv  \resto ^{1,  2  }_{r_{0} , \infty} (  z, \mathbf{Z},   \eta )  $.
By Taylor expansion (using the conventions under \eqref{eq:EPexp})
\begin{equation}\label{eq:tay1}   \begin{aligned} &  \resto ^{1,  2 }_{{r_0}, \infty } (z,\textbf{Z} ,\eta  ) =
\resto ^{1,  2 }_{ {r_0}, \infty } (z,\textbf{Z} ,0    ) +d _\eta   \resto ^{1,  2 }_{ {r_0}, \infty } (z,\textbf{Z} ,0    ) \eta +\int _0^1 (1-t)\partial _\eta  ^2  \resto ^{1,  2 }_{ {r_0}, \infty  } (z,\textbf{Z} ,t \eta     ) dt \cdot \eta ^ 2 .
\end{aligned}
\end{equation}
   Each of the terms in the r.h.s. is invariant by change of variables $(  z,    \eta ) \rightsquigarrow (e^{\im \vartheta} z,  e^{\im \vartheta} \eta )$.
   We have
  \begin{equation*}   \begin{aligned} &  \resto ^{1,  2 }_{ {r_0}, \infty } (z,\textbf{Z} ,\eta  )| _{\eta =0} =
  \resto ^{1,  2 }_{ \infty, \infty } (z,\textbf{Z}    )  =\sum _{k\le 2N+4} \frac{1}{k!} d _{\textbf{Z} }^{k} \resto ^{1,  2 }_{\infty, \infty } (z,0    )  \textbf{Z}^k +
  \resto ^{1,  2N+5 }_{\infty, \infty } (z,\textbf{Z}    )= \\&   \resto ^{1,  2N+5 }_{ \infty, \infty } (z,\textbf{Z}    ) +  \sum _{l =2} ^{2N+4}	 \sum _{
	  |\textbf{m}|=l+1 }  \textbf{Z}^{\textbf{m}}   c _{  \textbf{m} }( z )   =  \resto ^{1,  2N+5 }_{\infty, \infty } (z,\textbf{Z}    ) +  \sum _{l =2} ^{2N+4}	 \sum _{
	  |\textbf{m}|=l+1 }  \textbf{Z}^{\textbf{m}}  \sum _{j =1}^n  c_{ j \textbf{m} }( |z_j |^2 )    ,
\end{aligned}
\end{equation*}
where,  as in step 1 in Lemma \ref{lem:EnExp},
the last equality is obtained   by the invariance w.r.t $(  z,    \eta ) \rightsquigarrow (e^{\im \vartheta} z,  e^{\im \vartheta} \eta )$
and by smoothness.
We have  proceeding like above   \begin{equation*}   \begin{aligned} & d _\eta   \resto ^{1,  2 }_{ {r_0}, \infty } (z,\textbf{Z} ,0    ) \eta =
 \Re \langle \mathbf{S} ^{1,  1 }_{{r_0}, \infty } (z,\textbf{Z}     ),  \overline{\eta}   \rangle   =\sum _{k\le 2N+3} \frac{1}{k!}    \Re \langle  d _{\textbf{Z} }^{k}  \mathbf{S} ^{1,  1 }_{ {r_0}, \infty } (z,0    )  ,  \overline{\eta}   \rangle \textbf{Z}^k \\& +
  \Re \langle \mathbf{S} ^{1,  2N+4  }_{ {r_0}, \infty } (z,\textbf{Z}  ,\eta   ),  \overline{\eta}   \rangle  = \Re \langle \mathbf{S} ^{1,  2N+4  }_{ {r_0}, \infty } (z,\textbf{Z}  ,\eta   ),  \overline{\eta}   \rangle   +
   \sum _{j =1}^n \sum _{l =1} ^{2N+3}	 \sum _{
	  |\textbf{m}|=l  }
     ( \overline{z}_j \textbf{Z}^{\textbf{m}}  \langle
  A_{j \textbf{m}}(|z_j|^2  ), \eta \rangle  +c.c. )     ,
\end{aligned}
\end{equation*}
Finally, for a   $\resto ^{1,  2  }_{r_{0} , \infty} (e^{\im \vartheta} z, e^{\im \vartheta} \eta )\equiv \resto ^{1,  2  }_{r_{0} , \infty} (  z,   \eta )  $ we have, see {Definition} \ref{def:scalSymb},
 \begin{equation*}   \begin{aligned} &     \int _0^1 (1-t)\partial _\eta  ^2  \resto ^{1,  2 }_{{r_0}, \infty } (z,\textbf{Z} ,t \eta     ) dt \eta ^ 2 =  \resto ^{1,  2 }_{ {r_0}, \infty } (z,  \eta     ) .
\end{aligned}
\end{equation*}
By \eqref{eq:tay1} and the subsequent formulas we see that  $  \resto ^{1,  2  }_{{r_0}, \infty } (z', \eta ' )$ is absorbed  in   last three lines of
   \eqref{eq:enexp10} (with $l\ge 1$ in the 2nd line).
 The term
     $\langle   H   \eta',  \overline{ \eta}' \rangle = \langle   H   \eta ,  \overline{ \eta}  \rangle +\resto ^{1,  2 }_{ {r_0}-2, \infty } (z,\textbf{Z} ,\eta  )$  behaves similarly, recalling that $r_1=r_0 -2 $. Here too we have $\resto ^{1,  2 }_{ {r_0}-2, \infty } (e^{\im \vartheta} z,\textbf{Z}, e^{\im \vartheta} \eta )\equiv \resto ^{1,  2 }_{ {r_0}-2, \infty } (  z,\textbf{Z},   \eta )  $. This function can be treated like
     the  $  \resto ^{1,  2 }_{ {r_0} , \infty } (  z,\textbf{Z},   \eta )  $ discussed earlier.

     The terms $E (Q_{j z_j})$  and, for $j\neq k$,  $ \Re  \langle   q_{j z_j},  \overline{z}_k\phi _k \rangle  =\resto ^{1,  1 }_{\infty , \infty } (z,\textbf{Z}    )  $
     can be expanded similarly. But this time we need $l=0$  in the 2nd line.

\qed

The expansion in
Lemma \ref{lem:KExp1}  is too crude. We have the following additional and crucial fact.

\begin{lemma} [Cancellation Lemma]
  \label{lem:KExp2}  In   the 2nd line  of  \eqref{eq:enexp10}  all the terms with  $l = 0$ are zeros.
\end{lemma}

\proof
We first observe that the terms in the 2nd line of \eqref{eq:enexp10} with $l=0$ can be written as
\begin{equation}\label{eq:altcan0}
\sum_{k=1}^n\sum_{j\neq k}\sum_{A=R,I}z_{j A}b_{kjA}(z_k)+\sum_{k=1}^n\Re  \langle \mathbf{A}_k  ( z_k  ),\overline{\eta }\rangle.
\end{equation}
Indeed they are
\begin{equation}\label{eq:altcan3}
\sum _{ |\textbf{m}|=   1 }  \textbf{Z}^{\textbf{m}}   a_{  \textbf{m} }^{(1)}(|z _{1}|^2,..., |z _{n}|^2 )
+ \sum _{j =1}^n   ( \overline{z}_j    \langle
  G_{j \textbf{0}}^{(1)}(|z_j|^2  ), \eta \rangle  +   c.c. ),
\end{equation}
and it is obvious that the 2nd term of \eqref{eq:altcan3} is the second term of \eqref{eq:altcan0}.
    Arguing as in Lemma \ref{lem:EnExp}, the first term of \eqref{eq:altcan3} can be written as
\begin{equation*}
\sum_{k=1}^n\sum _{ |\textbf{m}|=   1 }  \textbf{Z}^{\textbf{m}}   a_{  k\textbf{m} }^{(1)}(|z _ k |^2 )
\end{equation*}
Further, for $\textbf{Z}^{\textbf{m}}=z_i\overline{z}_j$, we can assume that $i$ or $j$ must equal to $k$, because if not, it can be absorbed in the terms with $l\geq1$.
Set $\mathcal N _k:=\{\textbf{m}\ |\ |\textbf{m}|=1,\ m_{i,j}=0\ \mathrm{if}\ i\neq k\ \mathrm{and}\ j\neq k\}$.
We have
\begin{align*}
\sum_{k=1}^n\sum _{ |\textbf{m}|=   1 }  \textbf{Z}^{\textbf{m}}   a_{  k\textbf{m} }^{(1)}(|z _ k |^2 )&=\sum_{k=1}^n\sum _{ \textbf{m}\in \mathcal N _k }  \textbf{Z}^{\textbf{m}}   a_{  k\textbf{m} }^{(1)}(|z _ k |^2 )=\sum_{k=1}^n\sum_{j\neq k}(z_j\overline{z}_ka_{  k m_{jk} }^{(1)}(|z _ k |^2) +z_k \overline{z}_ja_{  k m_{jk} }^{(1)}(|z _ k |^2) ).
\end{align*}
So, we can write the term in the form of the first term of \eqref{eq:altcan0}.

\noindent Next, notice that for $p_k=(0,\cdots,0,z_k,\cdots,0;0)$,
\begin{align}
b_{kjA}(z_k)&=\left.\partial_{z_{j A}}K(z,\eta)\right|_{p_k} \text{  and }
 \mathbf{A}_k  ( z_k  )  = \nabla _{\eta} K (p_k).  \label{eq:altcan2}
\end{align}
Therefore, it suffices to show the r.h.sides in \eqref{eq:altcan2} are both zero.
Recall $u(z,\eta)=\sum_{j=1}^n Q_{j z_j} + R[z]\eta$. We have
\begin{align*}
\left.\partial_{z_{j A}}K(z,\eta)\right|_{p_k} &= \left.\partial_{z_{j A}}E(u(z'(z,\eta),\eta'(z,\eta)))\right|_{p_k}\\
&= \Re \langle \nabla E(u(z'(p_k),\eta'(p_k))), \overline{\left.\partial_{z_{j A}}u(z'(z,\eta),\eta'(z,\eta))\right|_{p_k}}\rangle.
\end{align*}
By Lemma \ref{lem:darflow0} we have
\begin{equation}\label{eq:Fpk}
(z'(p_k),\eta'(p_k))=p_k.
\end{equation}
So
\begin{equation*}
\nabla E(u(z'(p_k),\eta'(p_k)))=\nabla E(Q_{k z_k})=2E_{k z_k}Q_{k z_k}.
\end{equation*}
By Prop. \ref{prop:bddst}  and by \eqref{eq:Fpk}, for  $z_k= e ^{\im \vartheta _k} \rho_k$ we have
\begin{equation*}  \begin{aligned}
  &   -\im  \mathfrak{F}_*\frac{\partial}{\partial \vartheta _k} | _{p_k} =
  -\im  \frac{\partial}{\partial \vartheta _k}  (\sum_{j=1}^nQ_{j z_j'}+R[z'] \eta ') | _{p_k}
    =
    -\im \frac{\partial}{\partial \vartheta _k}  Q_{k  z_k }  =-  \im \frac{\partial}{\partial \vartheta _k} e ^{\im \vartheta _k} Q_{k   \rho_k }=  Q_{k z_k } ,
  \end{aligned}
\end{equation*}
where the 1st equality follows by definition of push forward, the 2nd by \eqref{eq:Fpk} and the 3rd by Prop.\ref{prop:bddst}.
Similarly, by the definition of push forward, we have
\begin{equation*}
\left.\partial_{z_{j A}}u(z'(z,\eta),\eta'(z,\eta))\right|_{p_k}= \mathfrak{F}_* \left.\partial_{z_{j A}}\right|_{p_k}.
\end{equation*}
Therefore  $b_{kjA}(z_k)=0$ follows by
\begin{equation*}
\left.\partial_{z_{j A}}K(z,\eta)\right|_{p_k} = 2E_{k z_k}\Im\langle \mathfrak{F}_* \partial_{\vartheta _k} | _{p_k}, \overline{ \mathfrak{F}_* \partial_{z_{j A}}} |_{p_k} \rangle=-  E_{k z_k}
\left.\Omega_0(\partial_{\vartheta _k}, \partial_{z_{j,A}}) \right|_{p_k} =0.
\end{equation*}

To get  $\mathbf{A}_k  ( z_k  )=0$, fix $\Xi \in \mathcal H_c [0]$ and set $p_{k,\Xi}(t):=(0,\cdots,0,z_k,0,\cdots,0;t\Xi)$. Then $\forall \, \Xi$
\begin{align*}
\Re \langle \nabla K (p_k), \Xi \rangle &= \frac{d}{dt} \left.K(p_{k,\Xi}(t))\right|_{t=0}
=\frac{d}{dt}\left.E(u(z'(p_{k,\Xi}(t)),\eta'(p_{k,\Xi}(t))))\right|_{t=0}\\
&=\Re\langle \nabla E(Q_{k z_k}), \overline{\frac{d}{dt}\left.u(z'(p_{k,\Xi}(t)),\eta'(p_{k,\Xi}(t)))  \right|_{t=0}}   \rangle \\
&=2E_{k z_k} \Im \langle \mathfrak{F}_*\frac{\partial}{\partial \vartheta _k} | _{p_k}, \overline{\mathfrak{F}_* \Xi} \rangle= -  E_{k z_k} \left.\Omega_0 (\frac{\partial}{ \partial \vartheta _k}, \Xi)\right|_{p_k}=0\Rightarrow \mathbf A _k (z_k) =0.
\end{align*}
\qed

\section{Birkhoff normal form  }
\label{sec:nforms}
In this section, where we search the effective Hamiltonian, the main result is Theorem \ref{th:main}.

 \noindent We consider the symplectic form $\Omega _0$ introduced in \eqref{eq:defOm0}.
  We introduce an index $\ell =j, \overline{j}$, for $\overline{\overline{j}}=j$
  with $j=1,...,n$. We write $\partial _j=\partial _{z_j}$ and  $\partial _{ \overline{j}}=\partial _{\overline{z}_j}$, $z_{\overline{j}}=\overline{z}_j$. With this notation, summing
    on $j$,
    by \eqref{eq:defOm0pr} and \eqref{eq:pert1}  for $\gamma _j(|z_j|^2) =\resto ^{2,0}_{\infty, \infty }(|z_j|^2)$ we have
\begin{equation} \label{eq:sympl0} \begin{aligned}          \Omega _0& = \im     (1 +   \gamma _j(|z_j|^2) ) dz_j
\wedge d\overline{z}_{j }  + \im \<d \eta  , d\overline{\eta}  \>   -\im \<d \overline{\eta}  , d {\eta}  \>      .
\end{aligned}\end{equation}
Given   $F\in C^1(U,\R )$
with $U$ an open subset of $\C^n\times \Sigma _{r}^c$,  its Hamiltonian vector field $X_F$ is defined by
$i_{X_F}\Omega _0=dF$.     We have summing on $j$
\begin{equation*} \begin{aligned}        &      i_{X_F}\Omega _0=    \im     (1 +   \gamma _j(|z_j|^2) ) ((X_F) _{j }   d\overline{z}_{j } -(X_F) _{\overline{j} }   d {z}_{j })  + \im \<(X_F) _{\eta  }  , d\overline{\eta}  \>   -\im \<(X_F) _{\overline{\eta} }    , d {\eta}  \>   \\&  =  \partial _j F    d {z}_j   +     \partial _{\overline{j} } F  d\overline{z}_{j }    +\<\nabla _{\eta  } F , d {\eta}  \> +  \<\nabla _{\overline{\eta}  } F   , d\overline{\eta}  \>     .
\end{aligned}\end{equation*}
So comparing the components of the two sides we get   for $1 +   \varpi _j(|z_j|^2)=  (1 +   \gamma _j(|z_j|^2) )^{-1}$  where    $\varpi  _j(|z_j|^2) =\resto ^{2,0}_{\infty, \infty }(|z_j|^2)$
:
\begin{equation} \label{eq:hamf1}\begin{aligned}        &    (X_F)_j    =-\im   (1 +   \varpi  _j(|z_j|^2))
 \partial _{\overline{j}} F   \ ,  \quad    (X_F)_{\overline{j}}   = \im   (1 +  \varpi  _j(|z_j|^2)) \partial _j F          \\&    (X_F)_\eta   =-\im    \nabla  _{\overline{\eta }} F  \,  , \quad     (X_F)_{\overline{\eta }}  =  \im    \nabla  _\eta F .
\end{aligned}\end{equation}

Given    $G\in C^1(U,\R )$  and   $F\in C^1(U, \textbf{E} )$  with $\textbf{E}$ a Banach space, we set
$  \{   F, G  \}  :=dF X_G$.

\begin{definition}[Normal Form]
\label{def:normal form}  Recall Def. \ref{def:setM}  and in particular  \eqref{eq:setM0}.
Fix
$r\in \N _0$.
A real valued function $Z( z,\eta  )$ is in normal form if  $
  Z=Z_0+Z_1
$
where $Z_0$ and $Z_1$ are finite sums of the following type  for $\mathbf l \ge 1$ and for $\mathbf{Z}= (z_i \overline{z}_j) _{i,j =1,...,n} $
  where $i\neq j$:
\begin{equation}
\label{e.12a}Z_1 (z,\mathbf{Z},\eta )=      \sum _{j =1}^n	\sum _{   \substack{|\textbf{m}  |   = \mathbf l \\ \mathbf{m}  \in \mathcal{M}_{j}(\mathbf l)} } \(\overline{z}_j \textbf{Z}^{\textbf{m}}  \langle
  G_{j \textbf{m}}(|z_j|^2  ), \eta \rangle  +   \text{c.c.}\), \text{ where   $ G_{j \textbf{m}}( |z_j|^2 ) =\mathbf{S} ^{0,0}_{r,\infty}(|z_j|^2  )$}
\end{equation}
   and where c.c. means complex conjugate; for  $ a_{  \textbf{m} }( |z _1|^2,...,|z _n|^2  )  =\resto ^{0,0}_{r,\infty}(|z _1|^2,...,|z _n|^2  )$
\begin{equation}
\label{e.12c}Z_0 (z,\mathbf{Z}  ) =  	\sum _{   \substack{|\textbf{m}  |   = \mathbf l +1\\ \mathbf{m}  \in \mathcal{M}_{0}(\mathbf l+1)} } \textbf{Z}^{\textbf{m}}   a_{  \textbf{m} }( |z _1|^2,...,|z _n|^2 )   \ . \end{equation}
 \end{definition}

\begin{remark}\label{rem:H3}
  By Lemma \ref{lem:M0},       $\textbf{Z}^{\textbf{m}} =|z_1| ^{2m_1}...|z_n| ^{2m_n} $  $\forall$  $\textbf{m} \in  \mathcal{M}_{0}(2N+4)$ for an $m \in  \N  _0^n$  with $2|m|=|\mathbf{m}| $.
By Lemma \ref{lem:M0}  for $|\textbf{m}  |  \le  2N+3$   either  $\sum _{a,b}(e_a-e_b)  {m}_{ab} -e_j >0$
or    $\sum _{a,b}(e_a-e_b)  {m}_{ab} -e_j <0$.
\end{remark}

For $ \mathbf{l} \le 2N+4$ we will consider   flows   associated to   Hamiltonian   vector fields $X_\chi$
with real valued
   functions  $\chi $ of the following form,  with    $ b_{  \textbf{m} }  =\resto ^{0,0}_{\mathbf r ,\infty }(|z _1|^2,...,|z _n|^2  )$  and    $ B_{j \textbf{m}} =\mathbf  S^{0,0}_{\mathbf{r},\infty}(|z_j|^2  )$  for some $\mathbf{r}\in \N $ defined in $B _{\C^n} (0, \mathbf{d})$ for some $\mathbf{d}>0$:
\begin{equation} \label{eq:chi1}\begin{aligned}
\chi   &=     	\sum _{   \substack{|\textbf{m}  |   = \mathbf{l} +1\\   \mathbf{{m}}
 \not \in \mathcal{M} _0 (\mathbf{l}+1) } } \textbf{Z}^{\textbf{m}}   b_{  \textbf{m} }( |z _1|^2,...,|z _n|^2 )    +
 \sum _{j =1}^n	\sum _{   \substack{|\textbf{m}  |   = \mathbf{l} \\ \mathbf{{m}}
 \not \in \mathcal{M} _j (\mathbf{l}) } } (\overline{z}_j \textbf{Z}^{\textbf{m}}  \langle
  B_{j \textbf{m}}(|z_j|^2  ), \eta \rangle  +   \text{c.c.})   \ .
\end{aligned}\end{equation}

\noindent
The Hamiltonian vector field $X_{\chi}$ can be explicitly computed using \eqref{eq:hamf1}. We have
\begin{equation}\label{eq:approxY0}
\begin{aligned}
(X_{\chi})_j&=(Y_{\chi})_j + (\tilde Y _ \chi)_j\ , \quad
(X_{\chi})_\eta  = -\im \sum _{j =1}^n	\sum _{   \substack{|\textbf{m}  |   = \mathbf{l} \\ \mathbf{{m}} \not \in \mathcal{M} _j (\mathbf{l}) } } z_j \overline{\textbf{Z}}^{\textbf{m}}
\overline{  B}_{j \textbf{m}}(|z_j|^2  ),
\end{aligned}
\end{equation}
where
\begin{equation}\label{eq:aproxY}
\begin{aligned}
 (Y _{\chi })  _{j} (z,\eta) := &
    -\im   (1 +   \varpi  _j(|z_j|^2)) \big [ \sum _{    |\textbf{m}  |   = \mathbf{l}+1}  b_{  \textbf{m} }( |z_1|^2,\cdots,|z_n|^2)
 \partial _{\overline{j}}\textbf{Z}^{\textbf{m}}   \\&
 +\sum _{k=1} ^{n} \sum _{     |\textbf{m}  |   = \mathbf{l}  }  ( \langle
  B_{k \textbf{m}}(|z_k|^2 )   , \eta \rangle   \partial _{\overline{j}}( \overline{z}_k \textbf{Z}^{\textbf{m}} )   + \langle
  \overline{B}_{k \textbf{m}}(|z_k|^2 )   , \overline{\eta} \rangle  \partial _{\overline{j}} (  {z}_k \overline{\textbf{Z}}^{\textbf{m}}   )  \big ]  , \\
 (\tilde Y _{\chi })  _{j} (z,\eta   ) := &
    -\im   (1 +   \varpi  _j(|z_j|^2)) \big [ \sum _{    |\textbf{m}  |   = \mathbf{l}+1}  \partial _{|z_j|^2}b_{  \textbf{m} }( |z_1|^2,\cdots,|z_n|^2)
 z_j \textbf{Z}^{\textbf{m}}   \\&
 + \sum _{     |\textbf{m}  |   = \mathbf{l}  }  ( \langle
  B_{j \textbf{m}}'(|z_j|^2  )   , \eta \rangle    |z_j|^2 \textbf{Z}^{\textbf{m}}    + \langle
  \overline{B}_{j \textbf{m}}'(|z_j|^2 )   , \overline{\eta} \rangle  {z}_j^2 \overline{\textbf{Z}}^{\textbf{m}}     \big ] .
\end{aligned}
\end{equation}
 Notice that $(Y _ \chi)_j =\mathcal R ^{1,\mathbf l}_{\mathbf r ,\infty}$, $(\tilde Y _ \chi)_j =\mathcal R ^{1,\mathbf l + 1}_{\mathbf r,\infty}$ and $(X_{\chi})_\eta = \mathbf{S}^{1,\mathbf l}_{\mathbf r ,\infty}$.
We introduce now a new space.

\begin{definition}\label{def:xr}
 We denote by $X_{\mathbf{r}}(\mathbf{l})$  the space  formed by
\begin{equation*}        \begin{aligned} &
\{ (b, B):=(\{b _{\mathbf{m}}\}_{\mathbf m\in \mathcal A(\mathbf l)} , \{B  _{j\mathbf{n}}\}_{j\in 1,\cdots, n, \mathbf n\in \mathcal B_j(\mathbf l)} )  :  \text{ $b _{\mathbf{m}}\in \C$, $B  _{j\mathbf{n}} \in \Sigma _{\mathbf{r} } ^{c}$}\\&\quad \text{ and $\chi (b,B)$ is real valued for all $z\in B_{\C^n(0,\mathbf d)}$} \}, \text { where}\\&
\mathcal A(\mathbf l):=\{ \mathbf m:\ |\mathbf m|=\mathbf l +1,\quad \mathbf m\not\in \mathcal M_0(\mathbf l +1)\},\\&
\mathcal B_j(\mathbf l):=\{ \mathbf n : \ |\mathbf n|=\mathbf l ,\quad \mathbf n\not\in \mathcal M_j(\mathbf l +1)\},
\end{aligned}
\end{equation*}
where we have assigned some order in the coordinates  and where \begin{equation*} \begin{aligned} &
\chi (b,B)   =     	\sum _{  \mathbf m  \in \mathcal A(\mathbf l)
} \textbf{Z}^{\textbf{m}}   b_{  \textbf{m} }    +
 \sum _{j =1}^n	\sum _{   \mathbf m \in \mathcal B_j(\mathbf l)
} (\overline{z}_j \textbf{Z}^{\textbf{m}}  \langle
  B_{j \textbf{m}} , \eta \rangle  +   \text{c.c.})   \ .
\end{aligned}\end{equation*}
We provide  $X_\mathbf{r}(\mathbf{l})$  with the norm
\begin{equation*}        \begin{aligned} &
 \| (b,B)\| _{X_\mathbf{r}(\mathbf{l})} =\sum _{  \mathbf m\in  \mathcal A(\mathbf l)
 } | b _{\mathbf{m}} |+  \sum _{j =1}^n	\sum _{   \mathbf m \in \mathcal B_j(\mathbf l)
 } \| B  _{j\mathbf{m}}  \| _{\Sigma _{\mathbf{r} }} .
\end{aligned}
\end{equation*}

\end{definition}

    Set $\varrho (z)=   ( \varrho _1 (z),....,\varrho _n (z)) $ with $\varrho _j (z) =  |z_j|^2   $.

\begin{lemma}
  \label{lem:birkflow0}   Consider  the $\chi$ in \eqref{eq:chi1} for fixed $\mathbf{r}>0$  and $\mathbf{l}\ge 1$,  with   coefficients $(b (\varrho (z))  , B (\varrho (z))   )\in C^2( B_{\C ^n}(0 , \mathbf{d}), X _{\mathbf{r}}(\mathbf{l} ))$  and with $B  _{j\mathbf{m}}(\varrho (z))=B_{j\mathbf{m}} (\varrho _j(z))$. Consider the
    system
\begin{equation*}\label{eq:bsyst1}
\begin{aligned}
  &    \dot z_{j } = (X_\chi )_{j }  (z, \eta ) \text{  and }     \dot \eta = (X_\chi )  _{\eta}  (z, \eta ),
     \end{aligned}
\end{equation*}
which is defined   in
     $(t, z)\in \R\times B_{\C ^n} (0,\mathbf{d} ) $ and
    $\eta   \in \Sigma _{ k}^c   $  for all  $k\in \Z\cap [-\mathbf{r},\mathbf{r} ]$
    (or    $\eta   \in H^1 \cap \mathcal{H}_c[0]     $) . Let $\delta  \in (0, \min (\mathbf{d}, \delta _1 ))$  with $\delta _1$ the constant of Lemma \ref{lem:darflow0}.
Then  the following properties hold.

\begin{itemize}
\item[(1)]  If
the following inequality holds,
\begin{equation}\label{eq:delta0}
\begin{aligned}
  &  4 (\mathbf{l}+1) \delta    \| (b(\varrho (z)),B(\varrho (z)))\| _{W ^{1,\infty} ( B _{\C^n }   (0, \mathbf{d} ),  X_\mathbf{r}(\mathbf{l}))} <1,
     \end{aligned}
\end{equation}
 then  for all $k\in \Z\cap [-\mathbf{r},\mathbf{r}]$  for the flow $\phi ^{t}(z,\eta ) $ we have
 \begin{align}
  &   \phi ^t \in C ^{\infty} ((-2, 2)\times B_{\C ^n} (0,\delta /2 )
    \times   B_{\Sigma _{ k}^c} (0,\delta  /2 ) ,B_{\C ^n} (0,\delta    )
    \times   B_{\Sigma _{ k}^c} (0,\delta   )) \text{ and} \label{eq:bsyst1220}\\&
   \phi ^t \in C ^{\infty} ((-2, 2)\times B_{\C ^n} (0,\delta  /2 )
    \times   B_{H^1 \cap \mathcal{H}_{c}[0]} (0,\delta  /2 ) ,B_{\C ^n} (0,\delta    )
    \times   B_{H^1 \cap \mathcal{H}_{c}[0]} (0,\delta    ) ) . \nonumber
     \end{align}
 In particular  for $z^t _j := z_j  \circ \phi ^{t}(z,\eta ) $ and $\eta ^t   := \eta  \circ \phi ^{t}(z,\eta ) $ and in the sense of Remark \ref{rem:sym}
\begin{equation}\label{eq:bsyst10}
\begin{aligned}
  &      z_{j }  ^{t}= z_{j }+ S _{j }  (t,z, \eta ) \text{  and }       \eta  ^t= \eta+ S _{\eta}  (t,z, \eta )\\ & \text{ with
  $S _{j } (t,z, \eta )=\mathcal{R}^{1, \mathbf{l} }_{\mathbf{r}, \infty}(t,z,\mathbf{Z},  \eta )$   and
  $ S _{\eta}  (t,z, \eta )=\textbf{S}^{1, \mathbf{l} }_{\mathbf{r}, \infty}(t,z,\mathbf{Z},  \eta )$.}\end{aligned}
\end{equation}

 \item[(2)] We have  $S _{j }  (t,e^{\im \vartheta}z,e^{\im \vartheta} \eta )=e^{\im \vartheta}S _{j }  (t,z, \eta )$,
 $S _{\eta }  (t,e^{\im \vartheta}z,e^{\im \vartheta} \eta )=e^{\im \vartheta}S _{\eta }  (t,z, \eta )$.

 \item[(3)] The flow  $ \phi ^t $ is canonical, that is $\phi  ^{t*}\Omega _0=\Omega _0$
 in $B_{\C ^n} (0,\delta  /2 )
    \times   B_{H^1 \cap \mathcal{H}_c[0]} (0,\delta  /2 )$.

 \end{itemize}

\end{lemma}

 \proof
  Claim  (2)   is elementary. The same is true for (3) given that $\phi ^t$ is a standard sufficiently regular flow.
In claim (1), \eqref{eq:bsyst10} and the following sentence are   a consequence of Lemma \ref{lem:flow1}.  The first part of claim (1) follows from
elementary estimates such as
\begin{equation*}\begin{aligned}&
  |(X_\chi )_{j }  (z, \eta )| = | (1 +   \varpi  _j(|z_j|^2))
 \partial _{\overline{j}} \chi (z, \eta )| \\& \le  (1+\| \varpi  _j\| _{L^\infty (B_{\C  } (0,\delta  _0   ))})  (\mathbf{l}+1)  \| (b ,B )\| _{W ^{1,\infty} ( B _{\C^n }   (0, \delta _0),  X_\mathbf{r}(\mathbf{l}))}\delta _0 ^{\mathbf{l}+1}
  \end{aligned}
\end{equation*}
for $(z, \eta )\in  B_{\C ^n} (0,\delta     )
    \times   B_{\Sigma _{-r}^c} (0,\delta     )$.  Notice that taking $\delta_0$ sufficiently small in Lemma \ref{lem:vectorfield0}, we can arrange $\| \varpi  _j\| _{L^\infty (B_{\C  } (0,\delta  _0   ))}<1$.    We also have
    \begin{equation*}\begin{aligned}&
  \|(X_\chi )_{\eta }  (z, \eta )\| _{\Sigma _{\mathbf{ r}}}   \le     \| (0 ,B )\| _{L ^{ \infty} ( B _{\C^n }   (0, \delta _0),  X_\mathbf{r}(\mathbf{l}))}\delta _0 ^{\mathbf{l}+1}.
  \end{aligned}
\end{equation*}
     Then if
    \eqref{eq:delta0}  holds    we obtain \eqref{eq:bsyst1220}.
\qed

The main part of $\phi^t$ will be given by the following lemma.
\begin{lemma}\label{lem:approx}
Consider  a function $\chi $ as in \eqref{eq:chi1}.
     For a  parameter $\varrho \in  [0,\infty )^n$  consider
      the field  $W_\chi$  defined as follows (notice that $W_\chi (z,\eta ,\varrho (z)  ) = Y _\chi (z,\eta) $):
   \begin{equation} \label{eq:appr1}
\begin{aligned}
 &   (W _{\chi })  _{j} (z,\eta ,\varrho    ) :=
    -\im   (1 +   \varpi  _j(\varrho _j  )) \big [ \sum _{    |\textbf{m}  |   = \mathbf{l}+1}  b_{  \textbf{m} }( \varrho    )
 \partial _{\overline{j}}\textbf{Z}^{\textbf{m}}   +\\& \sum _{k=1} ^{n} \sum _{     |\textbf{m}  |   = \mathbf{l}  }  (
\langle
  B_{k \textbf{m}}(\varrho _k    )   , \eta \rangle   \partial _{\overline{j}}( \overline{z}_k \textbf{Z}^{\textbf{m}} )   + \langle
  \overline{B}_{k \textbf{m}}(\varrho _k    )   , \overline{\eta} \rangle  {z}_k \partial _{\overline{j}}   \overline{\textbf{Z}}^{\textbf{m}}     \big ]  , \\&
 (W _{\chi })  _{\eta} (z,\eta  ,\varrho   ) :=-\im   \sum _{k=1} ^{n} \sum _{     |\textbf{m}  |   = \mathbf{l}  }
    {z}_k \overline{\textbf{Z}}^{\textbf{m}} \overline{B}_{k \textbf{m}}(\varrho _k    )      .
\end{aligned}
\end{equation}
 Denote by  $(w^t, \sigma ^t )=\phi _0^t (z,\eta )  $ the flow associated to     the system
  \begin{equation} \label{eq:appr2}   \begin{aligned} &
  \dot w _j =   (W _{\chi })  _{j} (w ,\sigma ,\varrho (z)  ) \, , \quad  w_j(0)=z_j \, , \\&
   \dot \sigma =   (W _{\chi })  _{\sigma } (w ,\sigma  ,\varrho (z)    ) \, , \quad  \sigma(0)=\eta \, .
\end{aligned}
\end{equation}
Let   $\delta  \in (0, \min (\mathbf{d}, \delta _1 ))$  like in  Lemma \ref{lem:birkflow0}.
Then the following facts hold.

\begin{itemize}
\item[(1)]   If \eqref{eq:delta0} holds, then, for $B(\varrho (z))=(B_{j \mathbf{m}}(\varrho _j(z))_{j \mathbf{m}}$,
 \begin{equation} \label{eq:appr3}
\begin{aligned}
  &      w_{j }  ^{t}= z_{j }+ T _{j }  (t, b(\varrho (z)), B(\varrho (z)) ,z, \eta ) \text{  and }       \sigma   ^t= \eta+ T _{\eta}  (t, b(\varrho (z)), B(\varrho (z)),z, \eta )\end{aligned}
\end{equation}
\begin{equation} \label{eq:appr33}
\begin{aligned}
  &    \text{$T _{j } $     (resp.   $T _{\eta}$)    $C^{\infty }$ for $(t,b , B, z,\eta )\in (-2, 2)\times  B_{X_{\mathbf{r}}} (0,c) \times  B_{\C ^n} (0,\delta )
    \times   B_{\Sigma _{-\mathbf{r}}} (0,\delta )$ }\\& \text{ with values in $\C$
     (resp.   $\Sigma _{  \mathbf{r} })$.}\end{aligned}
\end{equation}
 Furthermore, we have
\begin{equation} \label{eq:appr4}
\begin{aligned}
  &      {T} _{j } (t,b,B, z, \eta        )=    \mathcal{R}^{1, \mathbf{l} }_{\mathbf{r}, \infty }( t,b,B ,z,\mathbf{Z},  \eta )   \\
  &     {T} _{ \eta } (t,b,B, z, \eta        )  =  \mathbf{S}^{1, \mathbf{l} }_{\mathbf{r}, \infty }(t,b,B  ,z,\mathbf{Z},\eta   ).   \end{aligned}
\end{equation}

\item[(2)] We have the gauge covariance for any fixed $\vartheta \in \R$  \begin{equation} \label{eq:appr404}
\begin{aligned}
  &      {T} _{j } (t,b,B,  e^{\im \vartheta}z, e^{\im \vartheta}\eta        )=  e^{\im \vartheta}   {T} _{j } (t,b,B,  z, \eta        )  \\
  &     {T} _{ \eta } (t,b,B, e^{\im \vartheta}z,e^{\im \vartheta} \eta        )  = e^{\im \vartheta}  {T} _{ \eta } (t,b,B, z, \eta       ).   \end{aligned}
\end{equation}

  \item[(3)]     Consider the Hamiltonian flow $(z^t,\eta^t )=\phi ^t(z ,\eta  )$ associated to  $\chi $, see Lemma \ref{lem:birkflow0}. Then
     \begin{equation}\label{eq:appr6}
\begin{aligned}
  &  z^t -w^t= \mathcal{R}^{1, \mathbf{l} +1}_{\mathbf{r}, \infty }(t,z,\mathbf{Z},  \eta )\quad , \quad
   \eta ^t -\sigma ^t= \mathbf{S}^{1, \mathbf{l} +1}_{\mathbf{r}, \infty}(t,z,\mathbf{Z},  \eta )
\end{aligned}
\end{equation}

\end{itemize}
\end{lemma}
\proof    We have \eqref{eq:appr3}--\eqref{eq:appr33} by standard   ODE theory.
For $\mathbf{W} =(w _i \overline{w}_j  )_{i\neq j}$  like the $\mathbf{Z}$ in \eqref{def:comb0}
\begin{equation} \label{eq:appr51}
\begin{aligned}
  &  w^t_j=  z_j -\im   (1 +   \varpi  _j(\varrho _j (z)))  \big [ \sum _{    |\textbf{m}  |   = \mathbf{l}+1}  b_{  \textbf{m} }( \varrho   (z))
  \int _0^t (\partial _{\overline{j}} \textbf{W}^{ \textbf{m}})^s
   ds +\\& \sum _{k=1} ^{n} \sum _{     |\textbf{m}  |   = \mathbf{l}  }   (
\langle
  B_{k \textbf{m}}(\varrho _k (z)  )   ,  \int _0^t \sigma  ^s  (\partial _{\overline{j}}( \overline{w}_k   \textbf{W} ^{ \textbf{m}} ))^s ds\rangle       + \langle
  \overline{B}_{k \textbf{m}}(\varrho _k (z)  )   ,   \int _0^t \overline{\sigma} ^s {w}_k^s (\partial _{\overline{j}}   \overline{ \textbf{W} ^{ \textbf{m}} })^s ds\rangle  \big ] .
\end{aligned}
\end{equation}
where  $(\partial _{\overline{j}}   \overline{ \textbf{W} ^{ \textbf{m}} })^s = \partial _{\overline{j}}   \overline{ \textbf{W} ^{ \textbf{m}} } | _{w=w^s} $.
Similarly we have
\begin{equation} \label{eq:appr52}
\begin{aligned}
  &  \sigma ^t =  \eta  - \im   \sum _{k=1} ^{n} \sum _{     |\textbf{m}  |   = \mathbf{l } }
  \overline{B}_{k \textbf{m}}(\varrho _k (z)  )  \int _0^t  {w}_k^s(   \overline{ \textbf{W} ^{ \textbf{m}} })^s ds .
\end{aligned}
\end{equation}
 Like in Lemma \ref{lem:flow1},  we have  also $  {\mathbf{W}} ^t =\mathbf{Z} +\int _0^t  \resto ^{1,\mathbf{l} }_{\mathbf{r},\infty}(s,b(\varrho (z)),B(\varrho (z)), z, \mathbf{Z},\eta ) ds$.
 We can apply  Gronwall inequality like in Lemma \ref{lem:flow1} on these formulas
 to obtain    \eqref{eq:appr4}.
    This yields claim (1).

\noindent $ (W _{\chi })  _{j} (e^{\im \vartheta}w,e^{\im \vartheta}\sigma  ,\varrho (z)  )= e^{\im \vartheta}  (W _{\chi })  _{j} (w,\sigma  ,\varrho (z)  ) $  and  $ (W _{\chi })  _{\eta} (e^{\im \vartheta}w,e^{\im \vartheta}\sigma  ,\varrho (z) ) = e^{\im \vartheta}  (W _{\chi })  _{\eta} (w,\sigma  ,\varrho (z) )$ yield  claim (2).

\noindent  Consider claim (3). Observe that  \eqref{eq:appr6}  holds
replacing   $\mathbf{l}+1$  by $\mathbf{l}$.
By \eqref{eq:approxY0},
we have  for a fixed $C$
\begin{equation*}  \begin{aligned}    |\dot z -\dot w |  &\le   | (W _{\chi })  _{j} (z, \eta   )-(W _{\chi })  _{j} (w,\sigma )| +|\mathcal{R}^{1, \mathbf{l} +1}_{\mathbf{r}, \infty}(t,z,\mathbf{Z},  \eta )| \\& \le C |  z -  w |  +  C \| \eta -\sigma \|  _{\Sigma _{-r }}
	+|\mathcal{R}^{1, \mathbf{l} +1}_{\mathbf{r}, \infty}(t, z,\mathbf{Z},  \eta )| . \end{aligned}   \end{equation*}
Similarly we have
\begin{equation*}  \begin{aligned}  &   \| \dot  \eta -\dot  \sigma \|  _{\Sigma _{ \mathbf{r} }}
\le  \|  (W _{\chi })  _{\eta } (z, \eta ,\varrho (z)  )-(W _{\chi })  _{\eta} (w,\sigma ,\varrho (z)  ) \|  _{\Sigma _{ r }}
\le C |  z -  w |  +  C \| \eta -\sigma \|  _{\Sigma _{-\mathbf{r} }}
	 . \end{aligned}   \end{equation*}
We then conclude  by Gronwall's inequality
\begin{equation*}  \begin{aligned}  &  |  z ^t -  w ^t | + \|   \eta ^t -   \sigma  ^t\|  _{\Sigma _{ \mathbf{r} }}  \le   |\mathcal{R}^{1, \mathbf{l} +1}_{\mathbf{r}, \infty}(t,z,\mathbf{Z},  \eta )|   \end{aligned}   \end{equation*}
which, along with  \eqref{eq:appr6}    with $\mathbf{l}+1$ replaced by $\mathbf{l}$, yields  \eqref{eq:appr6}
ending Lemma \ref{lem:approx}.

\qed

Using Lemma \ref{lem:approx}, we expand $\phi^1$ given in Lemma \ref{lem:birkflow0}.

\begin{lemma}\label{lem:expand}
Let $(z',\eta')=\phi^1(z,\eta)$, where $\phi^t$ is the canonical flow given in Lemma \ref{lem:birkflow0}.
We have:
\begin{itemize}
\item[(1)]  for $\mathcal T_j(b,B,z,\eta ) = \mathcal R ^{3, 2\mathbf l -1}_{\mathbf r, \infty}$, $\mathcal T_\eta (b,B,z,\eta )= \mathbf S ^{3, 2\mathbf l -1}_{\mathbf r, \infty}$ and $\mathcal T_j$, $\mathcal T_\eta$   smooth   in $(b,B,z,\eta)$,
\begin{equation}\label{eq:expand1}
\begin{aligned}
z_j'&=z_j + (Y_\chi)_j(z,\eta) + \mathcal T_j(b(\varrho (z)),B(\varrho (z)),z,\eta ) + \mathcal R^{1,\mathbf l +1}_{\mathbf r,\infty},\\
\eta'&= \eta + (X_\chi)_\eta (z,\eta) + \mathcal T_\eta (b(\varrho (z)),B(\varrho (z)),z,\eta )+  \mathcal S^{1,\mathbf l +1}_{\mathbf r,\infty};
\end{aligned}
\end{equation}

\item[(2)] for $\tilde {\mathcal T}_j (b ,B ,z,\eta )= \mathcal R ^{1, 2\mathbf l}_{\mathbf r, \infty}$    smooth in $(b,B,z,\eta)$,
\begin{equation}\label{eq:expand2}
|z_j'|^2= |z_j|^2 + \overline z _j (Y_\chi)_j(z,\eta) + z_j \overline{(Y_\chi)_{  j}(z,\eta)} + \tilde{\mathcal T}_j(b(\varrho (z)),B(\varrho (z)),z,\eta ) + \mathcal R^{1,2\mathbf l +1}_{\mathbf r, \infty}.
\end{equation}

\end{itemize}

\end{lemma}

\begin{remark}
For $\mathbf l \geq 2$,   $\mathcal T_j$ and $\mathcal T_\eta$ are absorbed in  $\mathcal R^{1,\mathbf l +1}_{\mathbf r,\infty}$ and $\mathcal S^{1,\mathbf l +1}_{\mathbf r,\infty}$
and do not appear in the  \textit{homological equations} in     Theorem \ref{th:main}.  But if   $\mathbf l =1$ they do, although as   small perturbations.
\end{remark}

\proof
First of all  by \eqref{eq:aproxY} and by Definition \ref{def:xr}  we have $  \overline{z}_j (\tilde Y_\chi)_{ j} + z_j\overline{ (\tilde Y_\chi)_{ j}}=  2\Re \left ( \overline{z}_j (\tilde Y_\chi)_{ j} \right ) =0.$
So, using the following formula to define $\mathcal Y_j  $, we have
\begin{equation}\label{eq:squareapprox-1}
\frac{d}{dt} |z_j|^2 = \overline z_j (X_\chi)_{ j} + z_j \overline{( X_\chi)_{  j}}=\overline z_j (Y_\chi)_{ j} + z_j \overline{(Y_\chi)_{  j}}=:\mathcal Y_j (z,\eta).
\end{equation}
Notice that $\mathcal Y_j $ is $\resto^{0,\mathbf l +1}_{\mathbf r, \infty}$.
Therefore, we have
\begin{equation}\label{eq:squareapprox1}
|z_j^s|^2-|z_j|^2=\resto^{0,\mathbf l +1}_{\mathbf r, \infty}.
\end{equation}
This implies
\begin{equation}\label{eq:squareapprox2}
b(\varrho (z^s))-b(\varrho  (z))=\resto^{0,\mathbf l +1}_{\mathbf r, \infty}\ \mathrm{and}\ B(\varrho (z^s))-B(\varrho (z))=\mathcal S^{0,\mathbf l +1}_{\mathbf r, \infty}.
\end{equation}
Similarly, see right before \eqref{eq:hamf1}, we have
\begin{equation}\label{eq:squareapprox3}
\varpi_j(|z_j^s|^2)-\varpi_j(|z_j|^2)=\resto^{2,\mathbf l +1}_{\mathbf r, \infty}
\end{equation}
Now  we show (1).
By \eqref{eq:approxY0} and \eqref{eq:appr1}, using \eqref{eq:squareapprox2} and \eqref{eq:squareapprox3}, we have
\begin{equation}\label{eq:squareapprox4}
 (Y_\chi)_j (z^s,\eta^s)- (W_\chi)_j(z^s,\eta^s,\varrho (z))=\resto^{1,2\mathbf l +1}_{\mathbf r, \infty}
\end{equation}
By \eqref{eq:approxY0}, \eqref{eq:bsyst10}, \eqref{eq:appr6} and \eqref{eq:squareapprox4}, we have
\begin{equation*}
\begin{aligned}
z_j'=&z_j + \int_0^1 (W_\chi)_j(z^s,\eta^s,\varrho  (z)) ds+\int_0^1\( (Y_\chi)_j (z^s,\eta^s)- (W_\chi)_j(z^s,\eta^s,\varrho (z))\) ds + \int_0^1 (\tilde Y _\chi)_j (z^s,\eta^s) ds\\
 = & z_j + \int_0^1 (W_\chi)_j (w^s+ \resto^{1,\mathbf l + 1}_{\mathbf r , \infty},\sigma^s+ S^{1,\mathbf l + 1}_{\mathbf r , \infty},\varrho (z))\,ds + \resto^{1,\mathbf l +1}_{\mathbf r, \infty}\\
= &  z_j + \int_0^1 (W_\chi)_j (w^s,\sigma^s,\varrho (z))\,ds + \resto^{1,\mathbf l +1}_{\mathbf r, \infty}= z_j  + (W_\chi)_j (z,\eta,\varrho (z)) + \mathcal T_j +\resto^{1,\mathbf l +1}_{\mathbf r, \infty},
\end{aligned}
\end{equation*}
where $\mathcal T_j = \int_0^1 (W_\chi)_j (w^s,\sigma^s,\varrho (z))\,ds - (W_\chi)_j (z,\eta,\varrho (z))$ and  the last     $\resto^{1,\mathbf l +1}_{\mathbf r, \infty}$ in the 2nd line   is different from the   $\resto^{1,\mathbf l +1}_{\mathbf r, \infty}$ in the 3rd line.
Finally, by (1) of Lemma \ref{lem:approx} and the fact $(W_\chi)_j=\resto^{1,\mathbf l}_{\mathbf r, \infty}$, we have $\mathcal T_j =\resto^{1,2\mathbf l-1}_{\mathbf r, \infty}$ with $\mathcal T_j$    smooth  in $(t,b, B, z,\eta)$.
The argument for  $\eta'$ is similar.

\noindent We next show (2).
Set $\tilde{\mathcal Y}_j(z,\eta,\varrho ):=\overline z_j (W_\chi)_{ j}(z,\eta,\varrho ) + z_j \overline{(W_\chi)_{  j}(z,\eta,\varrho )}$.
As in \eqref{eq:squareapprox1}--\eqref{eq:squareapprox2}  we have
\begin{equation*}
\tilde{\mathcal Y}_j(z^s,\eta^s,\varrho (z))-{\mathcal Y}_j(z^s,\eta^s)=\resto^{0,2\mathbf l + 2}_{\mathbf r,\infty}
\end{equation*}
where ${\mathcal Y}_j$ is defined in \eqref{eq:squareapprox-1}.
So  we have
\begin{equation*}
\begin{aligned}
|z_j'|^2&=|z_j|^2 + \int_0^1 \mathcal Y_j (z^s,\eta^s)\,ds
=|z_j|^2 + \int_0^1 \tilde{\mathcal Y}_j (z^s,\eta^s,\varrho (z))\,ds+\resto^{0,2\mathbf l + 2}_{\mathbf r,\infty} \\&
 =  |z_j|^2 + \int_0^1 \tilde{\mathcal Y}_j (w^s,\sigma^s,\varrho (z))\,ds + \resto^{1,2\mathbf l +1}_{\mathbf r, \infty} = |z_j|^2  + \tilde{\mathcal Y}_j (z,\eta) + \tilde {\mathcal T}_j +\resto^{1,2\mathbf l +1}_{\mathbf r, \infty},
\end{aligned}
\end{equation*}
where $ \tilde {\mathcal T}_j =\int_0^1 \tilde{\mathcal Y}_j (w^s,\sigma^s,\varrho (z))\,ds -\tilde{\mathcal Y}_j (z,\eta)$.
As in (1), we see $\tilde {\mathcal T}_j=\resto^{1,2\mathbf l}_{\mathbf r,\infty}$ and $\tilde T$ is $C^\infty$ for $(b,B,z,\eta)$.
\qed

After a coordinate change  $\phi=\phi^1$ as in Lemma \ref{lem:birkflow0}
 the Hamiltonian  expands like in  \eqref{eq:enexp10}.

\begin{lemma} [Structure Lemma]
  \label{lem:birkflow1}   Consider   a function $K$ which admits an expansion as in
  \eqref{eq:enexp10}  defined for $(z,\eta )\in B_{\C^n}(0 , \delta   ) \times ( B_{H^1}(0 , \delta   )\cap \mathcal{H}_c[0])$ for some small $\delta  >0$  and with   $r_1$ is replaced by a    $r'$.
  Suppose also that the  $l=0$  terms in the first two lines are zero.    Consider a function $\chi$  such as in
  \eqref{eq:chi1} with $1\le \mathbf{l} \le 2N+4$
  with $ \| (b ,B )\| _{W ^{1,\infty} ( B _{\C^n }   (0, \delta  ),  X_r(\mathbf{l}))}\le \underline{C}$  and with $\underline{C}$ a preassigned number.
  Suppose also that $2c_2 (2N+4)\delta   \underline{C}<1$ with $c_2$
  the constant of Lemma \ref{lem:birkflow0}.
  Denote by   $\phi =\phi ^1$
  the corresponding flow
      Then   claims (1)--(5) of  Lemma \ref{lem:birkflow0}
  hold and   for $(z,\eta )\in  B_{\C^n}(0 , \delta  /2) \times ( B_{H^1}(0 , \delta   /2)\cap \mathcal{H}_c[0])$ and for  $r= r'-2   $   for $\mathbf{Z} =(z_i \overline{z}_j) _{i,j =1,...,n} $
  where $i\neq j$ we have
  an expansion  \begin{equation}  \label{eq:enexp30}    \begin{aligned} &
    K\circ  \phi  (z,\eta )=H_2(z,\eta )  + \sum _{j=1}^{n}
    \lambda _j(|z_j| ^{2}  )    \\& +
	 	 \sum _{l=1}^{2N+3 } \sum _{
		  |\textbf{m}|=   l+1 }  \textbf{Z}^{\textbf{m}}   a_{  \textbf{m} }( |z_1 |^2,..., |z_n |^2 )  + \sum _{j =1}^n \sum _{l=1}^{2N+3 }	\sum _{   |\textbf{m}  |  =l} ( \overline{z}_j \textbf{Z}^{\textbf{m}}  \langle
  G_{j \textbf{m}}(|z_j|^2  ), \eta \rangle  +   c.c. )
		\\&          +
	 	  \resto ^{1,  2  }_{r, \infty } (z, \eta )+
	 	  \resto ^{0,  2N+5 }_{r, \infty} (z,\textbf{Z}  ,\eta )+   \Re \langle
  \textbf{S} ^{0,   2N+4 }_{r, \infty} (z,\textbf{Z},\eta ) , \overline{\eta} \rangle
		\\& +  \sum _{ i+j  = 2}   \sum _{
		  |\textbf{m}|\le 1 }  \textbf{Z}^{\textbf{m}}    \langle G_{2\textbf{m} ij } ( z,\eta  ),   \eta ^{   i} \overline{\eta} ^j\rangle
		+     \sum _{ d+c  = 3}   \sum _{ i+j  = d}   \langle G_{d ij } ( z,\eta ),   \eta ^{   i} \overline{\eta} ^j\rangle  \resto ^{0,  c }_{r,\infty } (z, \eta )     + E _P( \eta) \  , \end{aligned}
\end{equation}
where  $G_{j \textbf{m}}$,     $G_{2\textbf{m} ij }    $   and    $G_{d ij }  $ are   $\textbf{S} ^{0,  0} _{r,\infty }  $ and  the $a_{  \textbf{m}}$ are  $\resto ^{0,  0} _{\infty,\infty  }  $.  Furthermore,  for $|\textbf{m}|=0$ we have $G_{2\textbf{m} ij } ( z,\eta  ) = G_{2\textbf{m} ij } ( z   )$ are the functions in   \eqref{eq:b02}    and the  $ \lambda _j(|z_j| ^{2}  )$   are  the same of \eqref{eq:enexp10}. Furthermore the 1st function  in the 3rd line of \eqref{eq:enexp30}   satisfies
 $\resto ^{1,  2 }_{ r, \infty } (e^{\im \vartheta} z , e^{\im \vartheta} \eta )\equiv \resto ^{1,  2 }_{ r, \infty } (  z ,   \eta )  $.

\end{lemma}
\proof Like in Lemma \ref{lem:KExp1}  we consider   the expansion  \eqref{eq:enexp10}  for   $K( z',\eta ' )$, and substitute  the  formulas   $  z_{j }  '= z_{j }+ S _{j }  ( z, \eta ) $   and  $   \eta  '= \eta+ S _{\eta}  ( z, \eta )$.
Proceeding like in   Lemma \ref{lem:KExp1}      we have
\begin{equation} \label{eq:str1}   \begin{aligned} &  \resto ^{1,  2 }_{r', \infty} (z' ,\eta ' ) = \resto ^{1,  2 }_{r'  , \infty } (z  ,\eta   ) +   \resto ^{1,  2N+5 }_{ r', \infty} (z,\textbf{Z}  ,\eta   ) +
  \Re \langle \mathbf{S} ^{1,  2N+4  }_{ r', \infty} (z,\textbf{Z}  ,\eta   ),  \overline{\eta}   \rangle \\& +   \text{ terms like in the 2nd line of \eqref{eq:enexp30}}    ,
\end{aligned}
\end{equation}

 Similarly we have
    \begin{align} &
\langle   H   \eta',  \overline{ \eta}' \rangle = \langle   H   \eta ,  \overline{ \eta}  \rangle +\resto ^{1,  \mathbf{l}+1 }_{ r' -2  , \infty} (z,\textbf{Z} ,\eta  ) =    \langle   H   \eta ,  \overline{ \eta}  \rangle + \resto ^{1,  \mathbf{l}+1 }_{ r' -2  , \infty} (z, \eta  ) +\resto ^{1,  \mathbf{l}+1 }_{ r' -2  , \infty} (z,\textbf{Z}    ) \nonumber \\&+ \Re \langle \mathbf{S} ^{1,  \mathbf{l}  }_{ r' -2, \infty} (z,\textbf{Z}  ,\eta   ),  \overline{\eta}   \rangle    =  \langle   H   \eta ,  \overline{ \eta}  \rangle + \resto ^{1,  \mathbf{l}+1 }_{ r' -2  , \infty} (z, \eta  ) + \resto ^{1,  2N+5 }_{ r'-2, \infty} (z,\textbf{Z}  ,\eta   ) +
  \Re \langle \mathbf{S} ^{1,  2N+4  }_{ r'-2, \infty} (z,\textbf{Z}  ,\eta   ),  \overline{\eta}   \rangle \nonumber \\& +   \text{ terms like in the 2nd line of \eqref{eq:enexp30}} \label{eq:str2}
\end{align}
Consider  an   $ \lambda _j(|z_j| ^{2}  )$ in \eqref{eq:enexp10}. Then by \eqref{eq:expand2} we have
\begin{equation} \label{eq:str3}  \begin{aligned} &
\lambda (|z_j'| ^{2}  ) = \lambda \left (|z_j|^2 + \mathcal{R}^{0, \mathbf{ {l}}+1  }_{r, \infty }( z, \mathbf{{Z}},  \eta )  \right ) = \mu  (|z_j| ^{2}  ) +\mathcal{R}^{1,  \mathbf{{l}}+1  }_{r, \infty }( z,\mathbf{Z},  \eta ).
\end{aligned}
\end{equation}
The latter admits an expansion like    in  and below formula \eqref{eq:tay1}.

The term    $  \resto ^{1,  2 }_{ r, \infty } (  z ,   \eta )  $  in  the 3rd line of \eqref{eq:enexp30}    is either the first in the r.h.s in  \eqref{eq:str1} for $l>1$ in Lemma \ref{lem:darflow0}, or the sum of the latter with the $\resto ^{1,  l+1 }_{ r' -2  , \infty} (z, \eta  )$ originating from
\eqref{eq:str2}--\eqref{eq:str3} for $l=1$ in Lemma \ref{lem:darflow0}.  In either case  it satisfies
 $\resto ^{1,  2 }_{ r, \infty } (e^{\im \vartheta} z , e^{\im \vartheta} \eta )\equiv \resto ^{1,  2 }_{ r, \infty } (  z ,   \eta )  $.
Other  terms in \eqref{eq:enexp10}  computed at  $(z', \eta ')$  and by similar
elementary  expansions  are similarly absorbed in  \eqref{eq:enexp30}.
\qed

All of the above lemmas are preparatory for the following result, which will give us
an effective Hamiltonian by picking $\iota  =  2 {N}+4$.

\begin{theorem}[Birkhoff normal form]
\label{th:main} For any $\iota \in \N \cap [2,  2 {N}+4]$
there are a $\delta _\iota >0$,   a  polynomial
$\chi _{\iota}$  as in \eqref{eq:chi1}    with   $\mathbf{l}=\iota$,  $\mathbf{d}=\delta _\iota $ and   $\mathbf{r}= {r_{\iota}} =r_0- 2(\iota +1)$ s.t.
for all $k\in \Z\cap [-r(\iota),r(\iota)]$ we have for each $\chi _{\iota}$
a flow (for $\delta _1>0$ the constant in Lemma \ref{lem:KExp1})
 \begin{align}
  &   \phi ^t_{\iota} \in C ^{\infty} ((-2, 2)\times B_{\C ^n} (0,  \delta _\iota  )
    \times   B_{\Sigma _{ k}^c} (0,\delta _\iota  ) ,B_{\C ^n} (0,\delta _{\iota -1 } )
    \times   B_{\Sigma _{ k}^c} (0,\delta _{\iota -1 }    )) \text{ and} \label{eq:bsyst122}\\&   \phi ^t_{\iota} \in C ^{\infty} ((-2, 2)\times B_{\C ^n} (0,\delta _\iota  )
    \times   B_{H^1 \cap \mathcal{H}_{c}[0]} (0,\delta _\iota  ) ,B_{\C ^n} (0,\delta _{\iota -1 }   )
    \times   B_{H^1 \cap \mathcal{H}_{c}[0]} (0,\delta _{\iota -1 }   ) )   \nonumber
     \end{align}
 and  s.t.,for  $\mathfrak{F} ^{( \iota )} := \mathfrak{F}  \circ \phi _2\circ ...\circ \phi _ \iota   $,  $\mathfrak{F} $  the transformation in
Lemma  \ref{lem:darflow0} and
  $\phi _j=  \phi ^1_{\iota}   $, then
  for $(z,\eta )\in  B_{\C^n}(0 , \delta _\iota) \times ( B_{H^1}(0 ,\delta _\iota  )\cap \mathcal{H}_c[0])$  and for $\mathbf{Z} =(z_i \overline{z}_j) _{i,j =1,...,n} $,
  where $i\neq j$,  we have    \begin{equation}  \label{eq:enexp40}    \begin{aligned} &
   H^{(\iota )} (z,\eta ):= E\circ \mathfrak{F} ^{( \iota )} (z,\eta )=H_2(z,\eta )  +\sum _{j=1}^{n} \lambda _j (|z_j| ^{2}  )+ Z  ^{(\iota )} (z,\mathbf{Z},\eta ) \\& +
	 	 \sum _{l=\iota  }^{2N+3 } \sum _{
		  |\textbf{m}|=   l+1 }  \textbf{Z}^{\textbf{m}}   a_{  \textbf{m} }^{(\iota )}( |z_1 |^2,...,|z_n |^2 )     + \sum _{j =1}^n \sum _{l=\iota  }^{ 2N+3 }	 \sum _{   |\textbf{m}  |  =l} ( \overline{z}_j \textbf{Z}^{\textbf{m}}  \langle
  G_{j \textbf{m}}^{(\iota )}(|z_j|^2  ), \eta \rangle  +   c.c. )
		\\&          +
	 	  \resto ^{1,  2  }_{{r_{\iota}},\infty} (z, \eta )+
	 	  \resto ^{0,  2N+5 }_{{r_{\iota}},\infty} (z,\textbf{Z}  ,\eta )+   \Re \langle
  \textbf{S} ^{0,   2N+4 }_{{r_{\iota}},\infty} (z,\textbf{Z},\eta ) , \overline{\eta} \rangle +
		\\&    \sum _{ i+j  = 2}   \sum _{
		  |\textbf{m}|\le 1 }  \textbf{Z}^{\textbf{m}}    \langle G_{2\textbf{m} ij } ^{(\iota )}( z,\eta  ),   \eta ^{   i} \overline{\eta} ^j\rangle
		+     \sum _{ d+c  = 3}  \ \sum _{ i+j  = d}   \langle G_{d ij }^{(\iota )} ( z,\eta ),   \eta ^{   i} \overline{\eta} ^j\rangle  \resto ^{0,  c }_{{r_{\iota}},\infty} (z, \eta )     + E _P( \eta)  \end{aligned}
\end{equation}
where, for coefficients like in  Def. \ref{def:normal form}  for $(r ,m )=({r_{\iota}},\infty)$,
\begin{equation}  \label{eq:enexp41}
\begin{aligned}
&   Z ^{(\iota )} = \sum _{   \mathbf{m}  \in \mathcal{M}_{0}(\iota )  } \textbf{Z}^{\textbf{m}}   a_{  \textbf{m} }( |z _1|^2,...,|z _n|^2 )    + \sum _{j =1}^n	 ( \sum _{   \mathbf{m}  \in \mathcal{M}_{j}(\iota- 1)  } \overline{z}_j \textbf{Z}^{\textbf{m}}  \langle
  G_{j \textbf{m}}(|z_j|^2  ), \eta \rangle  +   \text{c.c.} ).
     \end{aligned}
\end{equation}
We have
$\resto ^{1,  2 }_{ r _{\iota}, \infty }  = \resto ^{1,  2 }_{ r _{2}, \infty }
 $ and
 $\resto ^{1,  2 }_{ r _{2}, \infty }
 (e^{\im \vartheta} z , e^{\im \vartheta} \eta )\equiv \resto ^{1,  2 }_{ r _{2}, \infty } (  z ,   \eta )  $.

\noindent  In particular we have for $\delta  _f:=  \delta  _{2 {N}+4} $ and for the $\delta _0$
 in Lemma \ref{lem:vectorfield0},
   \begin{equation}\label{eq:diff1}
\begin{aligned} &
  \mathfrak{F}^{(2 {N}+4)}(B_{\C^n}(0 , \delta  _f) \times ( B_{H^1}(0 , \delta  _f)\cap \mathcal{H}_c[0]) )\subset  B_{\C^n}(0 , \delta  _0) \times ( B_{H^1}(0 , \delta  _0)\cap \mathcal{H}_c[0])
\end{aligned}
\end{equation}
   with $ \mathfrak{F}|_{B_{\C^n}(0 , \delta  _f) \times ( B_{H^1}(0 , \delta  _f)\cap \mathcal{H}_c[0])} $   a diffeomorphism between its domain and   an open neighborhood of the origin in
   $\C^n \times (H^1 \cap \mathcal{H}_c[0])$.

 \noindent    Furthermore, for $r=r_0- 4N-10$
   there is a pair $\mathcal{R}^{1, 1} _{r, \infty}$ and  $\mathbf{S}^{1, 1} _{r, \infty}$ s.t.
for $(z',\eta ')=\mathfrak{F} ^{(2 {N}+4)}(z ,\eta  )$
\begin{equation}\label{eq:diff2}
\begin{aligned} &  z'=z+ \mathcal{R}^{1, 1} _{r, \infty}(z,\mathbf{Z},\eta ) \, \quad  \eta'= \eta+ \mathbf{S} ^{1, 1} _{r, \infty}(z,\mathbf{Z},\eta ) .
\end{aligned}
\end{equation}
Furthermore, by taking  all the  $\delta  _\iota >0$ sufficiently small, we can assume that all the symbols in the proof, i.e. the symbols in \eqref{eq:diff2}
and the symbols in the expansions \eqref{eq:enexp40}, satisfy the estimates
of Definitions \ref{def:scalSymb} and \ref{def:opSymb} for $|z|<\delta  _\iota $ and $\|  \eta \| _{\Sigma_{r(\iota)}}<\delta  _\iota$ for their respective
$\iota$'s.

\end{theorem}

\proof   Notice that the functional $K$  in Lemma \ref{lem:KExp1} satisfies
case $\iota =1$.
The proof will be    by induction  on $\iota$.
We   assume that
 $H^{(\iota )}$ satisfies the statement for $\iota \ge 1$  and prove that there is a
 $\phi _{\iota  +1}$ such that  $H^{(\iota +1)} := H^{(\iota )}\circ \phi _{\iota  +1}$
  satisfies the statement for $\iota +1$. We consider the representation \eqref{eq:enexp30} for $H^{(\iota )}$, which is guaranteed by the Structure Lemma \ref{lem:birkflow1}.
  Using  \eqref{eq:enexp30}    we set $\mathbf{h}=H^{(\iota )}(z,  \mathbf{{Z}}, \eta )$   interpreting    $(z , \mathbf{{Z} } , \eta)$ as independent variables.
  Then we have    for $\mathbf l=\iota$
\begin{align}
   \label{eq:derivmain1}  &  a_{  \textbf{m} }^{(\mathbf l )}( |z_1 |^2,...,|z_n |^2  )
=\frac{1}{ \textbf{m} ! }  \partial  _{\mathbf{Z}}  ^{\textbf{m}}
 \mathbf{h} {|_{(z,\eta ,\mathbf{{Z} } )= (z;0,0 )}}   \  ,  \quad   |\textbf{m} |\le 2 {N}+4 ,\\&  \label{eq:derivmain2} \overline{z} _j G_{j \textbf{m}}^{(\mathbf l )}(|z_j|^2  )
=\frac{1}{\textbf{m} ! }     \partial  _{\mathbf{Z}}  ^{\textbf{m}} \nabla _\eta  \mathbf{h} {|_{(z,\eta ,\mathbf{{Z} } )= (0,..., z_j,0,...0; 0,0 )}}
\  ,  \quad   |\textbf{m} |\le   2 {N}+3    .
\end{align}
 The inductive hypothesis on   $H^{(\iota )}$ is a statement on the
 Taylor coefficients in \eqref{eq:derivmain1}--\eqref{eq:derivmain2}, that is that,   for $\mathbf l=\iota $  (see Def. \ref{def:setM} and Remark \ref{rem:H3})
\begin{align}
   \label{eq:derivmain3}  &     \partial  _{\mathbf{Z}}  ^{\textbf{m}}
 \mathbf{h} {|_{(z,\eta ,\mathbf{{Z} } )= (z;0,0 )}} =0   \text{ for  all   $\mathbf{m} \not \in \mathcal{M}_0  (\mathbf l )$, }     \\&  \label{eq:derivmain4}       \partial  _{\mathbf{Z}}  ^{\textbf{m}} \nabla _\eta  \mathbf{h} {|_{(z,\eta ,\mathbf{{Z} } )= (0,..., z_j,0,...0; 0,0 )}}=0
  \text{ for  all $(j,\mathbf{m})$ with $\mathbf{m} \not \in \mathcal{M}_j (\mathbf l -1)$. }
\end{align}
 We consider now  a yet unknown $\chi$ as in  \eqref{eq:chi1}   with $\mathbf{l}=\iota  $, $\mathbf{r}=r_\iota $ and a yet to be determined $\mathbf{d}=\delta >0$.
Set $\phi :=\phi ^1$, where  $\phi ^t$ is the flow
of Lemma \ref{lem:birkflow0}. We are seeking $\chi$  such that $H^{(\iota )}
\circ\phi $ satisfies the conclusions of Theorem \ref{th:main} for $\iota +1$,
i.e. that  using again Lemma \ref{lem:birkflow1} and   setting this time   $\mathbf{h}=(H^{(\iota )}
\circ\phi )(z,\eta , \mathbf{{Z}} )$, we will have \eqref{eq:derivmain3}--\eqref{eq:derivmain4}
for $\mathbf l = \iota +1$. Notice that for any   $\chi$,    \eqref{eq:derivmain3}--\eqref{eq:derivmain4} are automatically true for  $\mathbf l = \iota  $.
This because    $  H^{(\iota )}
 (z,\eta , \mathbf{{Z}} )$ and  $  (H^{(\iota )}
\circ\phi )(z,\eta , \mathbf{{Z}} )$ have same derivatives in   \eqref{eq:derivmain1}  for $|\textbf{m} |\le \iota   $ and  in \eqref{eq:derivmain2}  for $|\textbf{m} |\le \iota -1  $.
So it is enough to consider   \eqref{eq:derivmain3}  for  $|\textbf{m} |= \iota +1 $ and
 \eqref{eq:derivmain4} for  $|\textbf{m} |= \iota   $.  This will be true for a specific choice of $\chi$ whose coefficients solve the \textit{Homological Equations}, which we set up in the sequel.

By \eqref{eq:expand1}   and  by $G_{2\textbf{0} ij } ^{(\iota )}( z, \eta   ) =G_{2\textbf{0} ij }  ( z    ) $ we have
\begin{align} &\nonumber
   H^{(\iota )} (z',\eta ' )=H_2(z',\eta ' )  + \sum _{j=1}^{n} \lambda_{j} (|z_j'| ^{2}  )+ Z ^{(\iota )} (z',\mathbf{Z}',\eta' ) +
	 	  \resto ^{1,  2  }_{r, \infty  } (z', \eta ') + \sum _{ i+j  = 2}      \langle G_{2\textbf{0} ij } ( z ' ),   \eta ^{\prime   i} \overline{\eta} ^{\prime   j} \rangle \\& +(*)+
	 	   \sum _{
		  |\textbf{m}|=   \iota +1 }  \textbf{Z}^{\textbf{m}}   a_{  \textbf{m} }^{(\iota )}( |z |^2 )     + \sum _{j =1}^n  	\sum _{   |\textbf{m}  |  =\iota} ( \overline{z}_j \textbf{Z}^{\textbf{m}}  \langle
  G_{j \textbf{m}}^{(\iota )}(|z_j|^2  ), \eta \rangle  +  \text{ c.c.} )  , \label{eq:enexp50}
		  \end{align}
where   $\mathbf{h}:=(*)(z,\eta , \mathbf{Z})$  satisfies  \eqref{eq:derivmain3}--\eqref{eq:derivmain4} for $\mathbf{l} = \iota +1 $.
In the sequel we will   use $(*)$ with this meaning.
Let $(z ',\eta ') =\phi  (z,\eta )$.
We have
\begin{equation}\label{eq:h2expan1}
\begin{aligned}
&\sum_{j=1}^n e_j \(\overline z_j (Y_\chi)_j(z,\eta) + z_j (Y_\chi)_{\overline j}(z,\eta)\)=
\sum_{|\mathbf m|=\iota +1}\im\widetilde{\mathbf{e}}\cdot (\mu(\mathbf m)-\nu(\mathbf m))b_{  \textbf{m} }( |z _1|^2,...,|z _n|^2 )\mathbf Z^{\mathbf m}\\&
 + \sum _j \sum _{|\mathbf{ m }| =\iota  }\(
\im\widetilde{\mathbf{e}}\cdot(\tilde \mu_j(\mathbf m) -\tilde \nu_j(\mathbf m))\<B_{j \mathbf m}(|z_j|^2),\eta\>\overline z_j \mathbf Z^{\mathbf m} +\text{c.c.}\) \text{  for }
\end{aligned}
\end{equation}\begin{equation}\label{eq:etilde}
\begin{aligned}& \text{$ \mathbf{{Z}}^{  \mathbf{{m}}}= z^{\mu (\mathbf{{m}})}
{\overline{z}}^{\nu (\mathbf{{m}})}$,   $ \overline{z}_j\mathbf{{Z}}^{  \mathbf{{m}}}= z^{\tilde\mu _j(\mathbf{{m}})}
{\overline{z}}^{\tilde\nu _j (\mathbf{{m}})}$, }\\
&\widetilde{\mathbf{e}} (z):=(e_1
    (1 +   \varpi  _1(|z_1|^2)) ,....,  e_n   (1 +   \varpi  _n(|z_n|^2))),
\end{aligned}
\end{equation}
and, summing on repeated indexes,
\begin{equation}\label{eq:h2expan2}
\begin{aligned}
&\<H\eta, (X_\chi)_{\overline \eta}(z,\eta)\>+\<H (X_\chi)_{ \eta}(z,\eta), \overline \eta\>=
\im\overline z_j  \mathbf{ Z}^{\mathbf m}\<HB_{j,\mathbf m}(|z_j|^2),\eta\>+\text{c.c.}\quad .
\end{aligned}
\end{equation}
So, by Lemma \ref{lem:expand}, \eqref{eq:h2expan1}--\eqref{eq:h2expan2} and using the notation in\eqref{eq:etilde},  we have

\begin{align}
 H_2 (z ',\eta ')=& \sum_{j=1}^ne_j|z_j'|^2+\<H\eta',\overline\eta'\>=H_2  (z,\eta ) + \sum_{ \substack{|\textbf{m}  |   = \mathbf{l} +1\\   \mathbf{{m}}
 \not \in \mathcal{M} _0 (\mathbf{l}+1) }}\im\widetilde{\mathbf{e}}\cdot (\mu(\mathbf m)-\nu(\mathbf m))b_{  \textbf{m} }( |z _1|^2,...,|z _n|^2 )\mathbf{ Z}^{\mathbf m} \nonumber\\&+
 \sum _j \sum _{ \substack{|\textbf{m}  |   = \mathbf{l} \\ \mathbf{{m}}
 \not \in \mathcal{M} _j (\mathbf{l}) }  }
\(\im\<\(\widetilde{\mathbf{e}}\cdot(\tilde \mu_j(\mathbf m) -\tilde \nu_j(\mathbf m))+H\)B_{j\mathbf m}(|z_j|^2),\eta\>\overline z_j \mathbf{ Z}^{\mathbf m} +\text{c.c.}\) \label{eq:backH}\\&+
\resto^{2,2\iota}_{\mathbf r,\infty}(b,B,z,\mathbf{ Z},\eta) + (*),\nonumber
\end{align}
 where    c.c. refers
 only to the second line and   in the last line
$$
\resto^{2,2\iota}_{\mathbf r,\infty}(b,B,z,\mathbf{ Z},\eta)=\sum_{j=1}^ne_j \tilde{\mathcal T}_j+\<H\eta,\overline{\mathcal T}_\eta\>+\<H\mathcal T_\eta,\overline \eta\>+\<H\mathcal T_\eta,\overline{\mathcal T}_\eta\>,
$$
where here and in the sequel of this proof we abuse notation
denoting by $(b,B)$ the element  in $X_r(\iota )$, see Def. \ref{def:xr},
with  entries $b _{\mathbf{m}}(|z_1|^2,...,|z_n|^2)$ and $B _{j\mathbf{m}}(|z_j|^2)$.
$\resto^{2,2\iota}_{\mathbf r,\infty}(b,B,z,\mathbf{ {Z}},\eta) $ can   be absorbed in $(*)$   if $\iota\geq 2$   but  if $\iota=1$ needs to be considered explicitly.
    By $\lambda_j(|z_j|^2)=\resto^{2,0}_{\infty,\infty}$ and \eqref{eq:expand2}  we have
\begin{equation}\label{eq:lamexp}
\lambda_j(|z_j'|^2)=\lambda_j(|z_j|^2)+ \resto^{2,\iota +1}_{\mathbf r, \infty}(b,B,z,\mathbf Z, \eta)+(*).
\end{equation}
Next, we claim
\begin{equation}\label{eq:normexp}
Z ^{(\iota )} (z',\mathbf{Z}',\eta' )=Z ^{(\iota )} (z,\mathbf{Z},\eta )+ \resto^{2,\iota +1}_{\mathbf r, \infty}(b,B,z,\mathbf Z, \eta)+(*).
\end{equation}
 Let us take a term  $\mathbf{Z}^{\mathbf{m}} a_{\mathbf{m}}(\varrho (z))$ in
 the sum \eqref{eq:enexp41}. Notice that by Lemma \ref{lem:M0} we have necessarily $|\mathbf{m}| \ge 2$. Furthermore, by   \eqref{eq:expand2} it is easy to see that we
 can omit the factor $a_{\mathbf{m}}(\varrho (z))$. For definiteness let $\mathbf{Z}^{\mathbf{m}} = |z_1|^2 |z_2|^2$ (so $ |\mathbf{m}| = 2$; the case $ |\mathbf{m}| > 2$ is simpler).  By   \eqref{eq:expand2} we have
 \begin{equation*}  \begin{aligned} &
  |z_1'|^2 |z_2'|^2 = (|z_1|^2 +   \mathcal R^{0, \iota +1   }_{\mathbf r, \infty} )(|z_2|^2  +   \mathcal R^{0, \iota +1   }_{\mathbf r, \infty})  = |z_1 |^2 |z_2 |^2+ R^{2, \iota +1   }_{\mathbf r, \infty}(b,B, z,\mathbf{Z},\eta),
\end{aligned}\end{equation*}
 where we used information such as $\tilde {\mathcal T}_j = \mathcal R ^{1, 2\iota }_{\mathbf r, \infty}$ contained in Lemma \ref{lem:expand} and the fact, easy to check, that
 $\overline z _j (Y_\chi)_j(z,\eta) + z_j (Y_\chi)_{\overline j}(z,\eta) =R^{0, \iota +1   }_{\mathbf r, \infty}(b,B,z,\mathbf{Z},\eta) $.

 \noindent To complete the proof of \eqref{eq:normexp} let us take now a term of the
 form $\overline{z}_2  \mathbf{Z}^{\mathbf{m}} \langle G (|z_2|^2) , \eta \rangle $.
  Here we can write $G=G (|z_2|^2) $  ignoring the  dependence on $|z_2|^2$
  and we can focus on $ |\mathbf{m}| = 1$. For definiteness let $\mathbf{Z}^{\mathbf{m}}= z_1\overline{z}_2  $. By Lemma \ref{lem:expand}
\begin{equation*}  \begin{aligned} &
    z_1'(\overline{z}_2') ^2  \langle G   , \eta ' \rangle  =
  ( {z}_1 +  \mathcal R^{1, \iota     }_{\mathbf r, \infty})   (\overline{z}_2 +  \mathcal R^{1, \iota     }_{\mathbf r, \infty}) ^2
    \langle G   , \eta + S ^{1, \iota     }_{\mathbf r, \infty} \rangle  .
\end{aligned}\end{equation*}
which for  $\iota >1$ is of the form  $ z_1  \overline{z}_2  ^2  \langle G   , \eta   \rangle +(*)$ and for    $\iota =1$ using   formula   \eqref{eq:expand1}
yields
\eqref{eq:normexp}.

By claim (1) in Lemma \ref{lem:birkflow0} and $d_\eta     \resto ^{1,  2  }_{r, \infty} (z , \eta   ) \cdot   \textbf{S}^{1, \iota }_{\mathbf r,\infty}(b,B,z,\eta)=\resto^{2,\iota+1}_{\mathbf r,\infty}(b,B,z,\mathbf Z,\eta)$ we get
 \begin{equation}  \label{eq:enexp51}    \begin{aligned}
	 	  \resto ^{1,  2  }_{r, \infty} (z', \eta ')&= \resto ^{1,  2  }_{r, \infty} (z, \eta ')+(*)=\resto ^{1,  2  }_{r, \infty} (z , \eta    ) + (*)\\&  +\int _0^1 d_\eta     \resto ^{1,  2  }_{r, \infty} (z , \eta  + \tau \textbf{S}^{1, \iota }_{\mathbf r,\infty}(b,B,z,\eta) ) \cdot   \textbf{S}^{1, \iota }_{\mathbf r,\infty}(b,B,z,\eta) d\tau\\&
=\resto ^{1,  2  }_{r, \infty} (z , \eta    )  + d_\eta     \resto ^{1,  2  }_{r, \infty} (z , \eta   ) \cdot   \textbf{S}^{1, \iota }_{\mathbf r,\infty}(b,B,z,\eta) + (*) . \\
		  \end{aligned}
\end{equation}

\noindent Like in   \eqref{eq:enexp51} and using \eqref{eq:expand1} and $ G_{2\textbf{0} ij } ( z ) =\mathcal{R}^{2,  0  }_{\infty, \infty}  (z)$, see \eqref{eq:b02}, we have
\begin{equation}  \label{eq:enexp54}    \begin{aligned}
	 	\sum _{ i+j  = 2}      \langle G_{2\textbf{0} ij } ( z ' ),   \eta ^{\prime   i} \overline{\eta} ^{\prime   j} \rangle &=  	\sum _{ i+j  = 2}      \langle G_{2\textbf{0} ij } ( z   ),   \eta ^{\prime   i} \overline{\eta} ^{\prime   j} \rangle + (*)\\&   =
 \sum _{ i+j  = 2}      \langle G_{2\textbf{0} ij } ( z ),   \eta ^{    i} \overline{\eta} ^{   j} \rangle + \resto^{3,\iota+1}_{\mathbf r,\infty}(b,B,z,\mathbf Z, \eta)+(*).
		  \end{aligned}
\end{equation}
 Therefore, we seek $\chi _{\iota}$ s.t. the following holds, with $\varrho (z)=( |z _1|^2,...,|z _n|^2 )$ and the notation in \eqref{eq:etilde}:
 \begin{align}
(*)&= \sum_{ \substack{|\textbf{m}  |   = \iota +1\\   \mathbf{{m}}
 \not \in \mathcal{M} _0 (\iota+1) }}\im\widetilde{\mathbf{e}}\cdot (\mu(\mathbf m)-\nu(\mathbf m))b_{  \textbf{m} }( \varrho (z)    )\mathbf{Z}^{\mathbf m}\nonumber\\&+
 \sum _j \sum _{ \substack{|\textbf{m}  |   = \iota \\ \mathbf{{m}}
 \not \in \mathcal{M} _j (\iota ) }  }
\(\im\<\(  \widetilde{\mathbf{e}} \cdot(  \mu_j(\mathbf m) -  \nu_j(\mathbf m))+H\)B_{j\mathbf m}(|z_j|^2),\eta\>\overline z_j \mathbf{Z}^{\mathbf m} +\text{c.c.}\)     \label{2ndline}\\&+
\resto^{2,\iota+1}_{\mathbf r,\infty}(b,B,z,\mathbf{Z},\eta)+
 \sum _{
		 \substack{|\textbf{m}  |   = \iota +1\\   \mathbf{{m}}
 \not \in \mathcal{M} _0 (\iota +1) } }  \textbf{Z}^{\textbf{m}}   a_{  \textbf{m} }^{(\iota )}(  \varrho (z))     + \sum _{j =1}^n  	\sum _{   \substack{|\textbf{m}  |   = \iota \\ \mathbf{{m}}
 \not \in \mathcal{M} _j (\iota ) } } ( \overline{z}_j \textbf{Z}^{\textbf{m}}  \langle
  G_{j \textbf{m}}^{(\iota )}(|z_j|^2  ), \eta \rangle  +  \text{ c.c.} ) . \nonumber
\end{align}
By a Taylor expansion we  can write
\begin{equation*}
\begin{aligned} & \resto^{2,\iota+1}_{\mathbf r,\infty}(b,B,z,\mathbf{Z},\eta) = (*)+\sum _{
		 \substack{|\textbf{m}  |   = \iota +1\\   \mathbf{{m}}
 \not \in \mathcal{M} _0 (\iota +1) } }  \textbf{Z}^{\textbf{m}}   \alpha _{  \textbf{m} } (b,B,  \varrho (z))    \\&  + \sum _{j =1}^n  	\sum _{   \substack{|\textbf{m}  |   = \iota \\ \mathbf{{m}}
 \not \in \mathcal{M} _j (\iota ) } } ( \overline{z}_j \textbf{Z}^{\textbf{m}}  \langle
  \Gamma _{j \textbf{m}} (b (0,...,|z_j|^2,0,...,0), B (0,...,|z_j|^2,0,...,0), |z_j|^2  ), \eta \rangle  +  \text{ c.c.} )
\end{aligned}
\end{equation*}
where $\alpha _{  \textbf{m} } (b,B,  \varrho (z))  = \resto^{1,0}_{\mathbf r,\infty}(b,B, \varrho (z) )$ and
\begin{equation*}
     \begin{aligned}
 \text{ where } &  \Gamma _{j \textbf{m}} ( b (0,...,|z_j|^2,0,...,0), B (0,...,|z_j|^2,0,...,0), |z_j|^2) \\  =&S^{1,0}_{\mathbf r,\infty}(b (0,...,|z_j|^2,0,...,0), B (0,...,|z_j|^2,0,...,0),|z_j|^2).
 \end{aligned}
\end{equation*}
Furthermore, by \eqref{eq:etilde} and  $\varpi  _j(|z_j|^2) =\resto ^{2,0}_{{r_0}, \infty }(|z_j|^2)$  the 2nd line of \eqref{2ndline} has an expansion
\begin{equation*}
     \begin{aligned}
  \sum _j \sum _{ \substack{|\textbf{m}  |   = \iota \\ \mathbf{{m}}
 \not \in \mathcal{M} _j (\iota ) }  }
\(\im\<\( {\mathbf{e}}\cdot(  \mu_j(\mathbf m) - \nu_j(\mathbf m)) + \resto ^{1,0}_{r_0,\infty }(|z_j|^2)+H\)B_{j\mathbf m}(|z_j|^2),\eta\>\overline z_j \mathbf{Z}^{\mathbf m} +\text{c.c.}\) +(*).
 \end{aligned}
\end{equation*}
Then we reduce to the following system:
 \begin{equation}
     \begin{aligned}
b_{  \textbf{m} }( \varrho (z)  ) &= \frac{\im  }{\widetilde{\textbf{e}} (z)    \cdot (  \mu (\mathbf{{m}})- \nu  (\mathbf{{m}}))}   [ a_{  \textbf{m} }^{(\iota )}(  \varrho (z) ) +\alpha _{  \textbf{m} } (( b _{\mathbf{n}}( \varrho (z)  ))_{\mathbf{n}}  ,(B_{j\mathbf{n}}( \varrho _j (z)  ) ) _{j\mathbf{n}},  \varrho (z)) ], \label{eq:enexp56}\\
B_{j \textbf{m}}(|z_j|^2  ) &=\im  R_H ( {\mathbf{e}}\cdot(  \mu_j(\mathbf m) -  \nu_j(\mathbf m)) + \resto ^{1,0}_{r_0,\infty }(|z_j|^2))   [ G_{j \textbf{m}}^{(\iota )}(|z_j|^2  ) \\& +
\Gamma _{j \textbf{m}} ( b (0,...,|z_j|^2,0,...,0), B (0,...,|z_j|^2,0,...,0), |z_j|^2)
 \end{aligned}
\end{equation}

\noindent The $b_{\textbf m}( \varrho (z)  )$, $ B_{j \textbf{m}}(|z_j|^2  )$ can be found by implicit function theorem  for $|z|<\delta _{\iota  }'$   for $\delta _{\iota  }'$   sufficiently small.
This gives us the desired
polynomial $\chi  $  yielding $H^{(\iota +1)}$. Formulas \eqref{eq:bsyst122}
for the flow $ \phi   ^{t} $ of   $\chi $ are obtained choosing
$  \delta _{\iota    }>0  $ sufficiently small by claim (1) in Lemma \ref{lem:birkflow0}. For the composition $ \mathcal{F}^{(2 {N}+4)}$ we obtain \eqref{eq:diff1}
as a consequence of \eqref{eq:bsyst122} and of \eqref{eq:KExp11}.

\qed

\section{Dispersion} \label{sec:disp}

We apply Theorem \ref{th:main},  set $  \mathcal{H} = H^{(2N+4 )}$ so that
\begin{equation}  \label{eq:enexp401}    \begin{aligned} &
  \mathcal{H} (z,\eta )=H_2(z,\eta )  +  \sum _{j=1}^{n} \lambda _j (|z_j| ^{2}  ) + {Z}^{(2N+4 )}  (z,\mathbf{Z},\eta  )   +\resto  \end{aligned}
\end{equation}
\begin{equation}  \label{eq:rest1}  \begin{aligned} & \resto :=
	 	  \resto ^{1,  2  }_{r , \infty} (z, \eta )+
	 	  \resto ^{0,  2N+5 }_{r , \infty} (z,\textbf{Z}  ,\eta )+   \Re \langle
  \textbf{S} ^{0,   2N+4 }_{r , \infty} (z,\textbf{Z},\eta ) , \overline{\eta} \rangle
		\\& +  \sum _{ i+j  = 2}   \sum _{
		  |\textbf{m}|\le 1 }  \textbf{Z}^{\textbf{m}}    \langle G_{2\textbf{m} ij } ( z,\eta  ),   \eta ^{   i} \overline{\eta} ^j\rangle
		+     \sum _{ d+c  = 3}  \ \sum _{ i+j  = d}   \langle G_{d ij }  ( z,\eta ),   \eta ^{   i} \overline{\eta} ^j\rangle  \resto ^{0,  c }_{r , \infty} (z, \eta )     + E _P( \eta)  . \end{aligned}
\end{equation}
Using  formula \eqref{eq:enexp41} for $\iota = 2N+4$   we have
\begin{equation}  \label{eq:rest-1}  \begin{aligned} &   \sum _{j=1}^{n} \lambda _j (|z_j| ^{2}  ) + {Z}^{(2N+4 )}  (z,\mathbf{Z},\eta  )  =  {Z}_0(z) +     \sum _{j =1}^n	 ( \sum _{   \mathbf{m}  \in \mathcal{M}_{j}(2N+3)  } \overline{z}_j \textbf{Z}^{\textbf{m}}  \langle
  G_{j \textbf{m}}(|z_j|^2  ), \eta \rangle  +   \text{c.c.} ), \\&   {Z}_0(z):=    \sum _{j=1}^{n} \lambda _j (|z_j| ^{2}  ) +  \sum _{   \mathbf{m}  \in \mathcal{M}_{0}(2N+4 )  } \textbf{Z}^{\textbf{m}}   a_{  \textbf{m} }( |z _1|^2,...,|z _n|^2 ) =\mathcal{Z}_0(|z_1|^2, ...,|z_n|^2)  ,   \end{aligned}
\end{equation}
 where the last equality holds for some $\mathcal{Z}_0(|z_1|^2, ...,|z_n|^2)  $  by       Lemma \ref{lem:M0}.

      \begin{theorem}[Main Estimates]\label{thm:mainbounds}
There exist $\epsilon_0>0$ and $C_0>0$ s.t.\ if the constant  $0<\epsilon   $  of Theorem \ref{thm:small en} satisfies $\epsilon<\epsilon_0$,    for $I= [0,\infty )$ and $C=C_0$    we have:
\begin{align}
&   \|  \eta \| _{L^p_t(I,W^{ 1 ,q}_x)}\le
  C   \epsilon \text{ for all admissible pairs $(p,q)$,}
  \label{Strichartzradiation}
\\& \| z_j \mathbf{Z}  ^{\mathbf{m}} \| _{L^2_t(I)}\le
  C   \epsilon \text{ for all   $(j,\mathbf{m})$
  with  $\mathbf{m} \in \mathcal{M}_j (2N+4)  $,} \label{L^2discrete}\\& \| z _j  \|
  _{W ^{1,\infty} _t  (I )}\le
  C   \epsilon \text{ for all   $j\in \{ 1, \dots ,  {n}\}$ } \label{L^inftydiscrete}
   .
\end{align}
Furthermore,  there exists $\rho  _+\in [0,\infty )^n$ s.t.   there exist  a   $j_0$  with $\rho_{+j}=0$ for $j\neq j_0$,
and   there exists $\eta _+\in H^1$   s.t.
  $| \rho  _+ - |z(0)| | \le C   \epsilon $ and $\eta _+\in H^1$
with $\|  \eta _+\| _{H^1}\le C    \epsilon $, such that
\begin{equation}\label{eq:small en31}
\begin{aligned}&     \lim_{t\to +\infty}\| \eta (t,x)-
e^{\im t\Delta }\eta  _+ (x)   \|_{H^1_x}=0  \quad  , \quad
  \lim_{t\to +\infty} |z_j(t)|  =\rho_{+j}  .
\end{aligned}
\end{equation}

\end{theorem}

\noindent {\it Proof that Theor.\ref{thm:mainbounds} implies Theor.\ref{thm:small en}.}
 Denote by $(z',\eta ')$  the initial coordinate system.   By  \eqref{eq:diff2}
\begin{equation*}  \begin{aligned} &
 z' =z+\resto ^{1, 1}_{r,\infty}(z,  \mathbf{{Z}}, \eta )  \, , \quad   \eta ' =\eta +\mathbf{S} ^{1, 1}_{r,\infty}(z,  \mathbf{{Z}}, \eta ).
\end{aligned}\end{equation*}
Notice that    \eqref{eq:small en31} and $\lim _{t\to + \infty}\mathbf{Z}(t)=0$ and  that by standard arguments  for $s>3/2$ we have \begin{equation} \label{eq:w-conv}
\begin{aligned} &  \lim _{t\to  + \infty} \|  e^{ t \Delta }\eta _+\| _{L^{2,-s}(\R ^3)}=0   \text{   for any $\eta _+\in L^2$}.
\end{aligned} \end{equation}
     These two limits,   Definitions \ref{def:scalSymb}--\ref{def:opSymb}  and \eqref{eq:small en31} imply
\begin{equation*}  \begin{aligned} &
 \lim _{t\to  + \infty }\resto ^{1, 1}_{r,\infty}(z, \mathbf{Z}, \eta ) =  0 \text{ in $\C ^n$ and }  \quad   \lim _{t\to  + \infty } \mathbf{S} ^{1, 1}_{r,\infty}(z, \mathbf{Z}, \eta ) =0  \text{ in $\Sigma _{r}$}.
\end{aligned}\end{equation*}
This means     that
\begin{equation}\label{eq:small en32}
\begin{aligned}&     \lim_{t\to +\infty}\| \eta ' (t,x)-
e^{\im t\Delta }\eta  _+ (x)   \|_{H^1_x}=0  \quad  , \quad
 \lim_{t\to +\infty} |z_j'(t)|  =\rho  _{+j} .
\end{aligned}
\end{equation}
so that   \eqref{eq:small en3}  is true.
Notice also that if we set $\widetilde{\eta}=\eta$ and $A(t,x)= \mathbf{S} ^{1, 1}_{r,\infty}(z,  \mathbf{{Z}}, \eta )$ we obtain the desired decomposition of $\eta'$ satisfying \eqref{eq:small en2} and \eqref{eq:small en4}.  Finally we have
\begin{equation*}
\begin{aligned}&          \dot z _j' +\im e_j z_j'  =  \dot z _j  +\im e_j z_j + \frac{d}{dt}\resto ^{1, 1}_{r,\infty}(z, \mathbf{Z}, \eta ) +  \resto ^{1, 1}_{r,\infty}(z, \mathbf{Z}, \eta )  = O(\epsilon ^2),
\end{aligned}
\end{equation*}
where $\dot z _j  +\im e_j z_j= O(\epsilon ^2) $  by \eqref{eq:FGR01} below, $ \resto ^{1, 1}_{r,\infty}(z, \mathbf{Z}, \eta ) =  O(\epsilon ^2)$ by \eqref{eq:scalSymb}
and   $\frac{d}{dt}\resto ^{1, 1}_{r,\infty}(z, \mathbf{Z}, \eta )= O(\epsilon ^2) $. To check the latter, we write  (it is easy that $d_w \resto ^{1, 1}_{r,\infty}(z, \mathbf{Z}, \eta )=  \resto ^{1, 0}_{r,\infty}(z, \mathbf{Z}, \eta ) $ for $w=z,\mathbf{Z}$)
\begin{equation*}
\begin{aligned}&        \frac{d}{dt}  \resto ^{1, 1}_{r,\infty}(z, \mathbf{Z}, \eta )  =  \resto ^{1, 0}_{r,\infty}(z, \mathbf{Z}, \eta )\dot z +   \resto ^{1, 0}_{r,\infty}(z, \mathbf{Z}, \eta )\dot {\mathbf{Z}} + d _\eta  \resto ^{1, 1}_{r,\infty}(z, \mathbf{Z}, \eta ) \cdot\dot \eta ,
\end{aligned}
\end{equation*}
with $ d _\eta  \resto ^{1, 1}_{r,\infty}$ the partial derivative in $\eta$.
 By a simple use of Taylor expansions and  Def. \ref{def:scalSymb}  \begin{equation*}
\begin{aligned}&        \| d_\eta  \resto ^{1, 1}_{r,\infty}(z, \mathbf{Z}, \eta ) \| _{\Sigma _{-r}^c \to \Sigma _{ r}^c} \le C (|z| + \| \eta \|  _{\Sigma _{-r}}).
\end{aligned}
\end{equation*}
Then by equations \eqref{eq:eq f}  and \eqref{eq:FGR01}   below, we have $\frac{d}{dt}\resto ^{1, 1}_{r,\infty}(z, \mathbf{Z}, \eta )= O(\epsilon ^2) $.
 This yields
the inequality claimed in the  second line in  \eqref{eq:small en2}.

 \qed

\noindent By a standard argument
\eqref{Strichartzradiation}--\eqref{L^inftydiscrete} for  $I= [0,\infty )$ are a consequence of the following Proposition.

\begin{proposition}\label{prop:mainbounds} There exists  a  constant $c_0>0$  such that
for any  $C_0>c_0$ there is a value    $\epsilon _0= \epsilon _0(C_0)   $ such that   if   the inequalities  \eqref{Strichartzradiation}--\eqref{L^inftydiscrete}
hold  for $I=[0,T]$ for some $T>0$, for $C=C_0$  and for $0< \epsilon < \epsilon _0$,
then in fact for $I=[0,T]$  the inequalities  \eqref{Strichartzradiation}--\eqref{L^inftydiscrete} hold  for   $C=C_0/2$.
\end{proposition}

 \subsection{Proof of Proposition \ref{prop:mainbounds}} \label{subsec:prop}

\begin{lemma}\label{lem:conditional4.2} Assume
the hypotheses of Prop. \ref{prop:mainbounds} and  take the ${M}$ of  Def. \ref{def:setM}.    Then  $\exists$ a fixed $c$
s.t.
\begin{equation}
  \|  \eta  \| _{L^p_t([0,T],W^{ 1 ,q} )}\le
  c  \epsilon   + c    \sum _{(\mu , \nu )\in {M}  }|   z ^{\mu}\overline{z}  ^{\nu}  | _{L^2_t( 0,T  )}  \text{ for all admissible pairs $(p,q)$} .
  \label{4.5}
\end{equation}

\end{lemma}
 \proof   First of all,  for $|z|< \delta _f$ and  $\|\eta \| _{H^1\cap \mathcal{H}_{c}[0]}< \delta _f$ defining the domain of the Hamiltionian   $\mathcal{H} (z,\eta )$ in  \eqref{eq:enexp401}, we will pick $ \epsilon _0\in (0,  \delta _f )$ sufficiently small.  Let  $ \epsilon  \in (0,  \epsilon _0)$, where  $\epsilon =\| u (0)\| _{H^1} $. By \eqref{eq:coo11}  we have
 $|z '(0)|+\| \eta ' (0) \| _{X} \le  c_1   \epsilon $, where  $(z '(0), \eta ' (0))$ are the  coordinates in the initial system of coordinates introduced in Lemma \ref{lem:systcoo}.
 Let   $(z  (0), \eta   (0))$ be the corresponding  coordinates in the final system of coordinates. Then by the relation \eqref{eq:diff2}, if $\epsilon _0$ is sufficiently small we conclude that
 \begin{equation}
  |z  (0)|+\| \eta   (0) \| _{H^1} \le  c_1'   \epsilon
  \label{4.5id}
\end{equation}
 for some other fixed constant $ c_1'$. We now turn to the equation  of
 $\eta$.
We have  for $\overline{G}_{j \textbf{m}}   = \overline{G}_{j \textbf{m}} (0  )$
\begin{equation} \label{eq:eq f} \begin{aligned} &
 \im \dot \eta =\im \{ \eta , \mathcal{H}  \} =    H \eta  + \sum _{j =1}^n \sum _{l=1 }^{ 2N+3 }	\sum _{   |\textbf{m}  |  =l}  {z}_j \overline{\textbf{Z}}^{\textbf{m}}
  \overline{G}_{j \textbf{m}}  + \mathbb{A} \text{  where}\\&
    \mathbb{A}:=
    \sum _{j =1}^n \sum _{l=1 }^{ 2N+3 }	\sum _{   |\textbf{m}  |  =l}  {z}_j \overline{\textbf{Z}}^{\textbf{m}}
  [\overline{G}_{j \textbf{m}} (|z_j|^2  )  - \overline{G}_{j \textbf{m}}  ] + \nabla _{\overline \eta} \resto   .
\end{aligned}\end{equation}
We rewrite
\begin{equation} \label{4.9} \begin{aligned} &
 \sum _{j =1}^n \sum _{l=1 }^{ 2N+3 }	\sum _{   |\textbf{m}  |  =l}  {z}_j \overline{\textbf{Z}}^{\textbf{m}}
  \overline{G}_{j \textbf{m}}= \sum _{(\mu , \nu )\in {M}  }    \overline{z} ^{\mu} {z}  ^{\nu} \overline{G} _{\mu \nu}.
\end{aligned}\end{equation}
Notice that \eqref{L^2discrete} is the same as
 \begin{align}
&    \|   {z} ^{\mu} \overline{{z}}  ^{\nu} \| _{L^2_t(I)}\le
  C   \epsilon \text{ for all     $(\mu , \nu )\in {M}  $ } .\label{L^2disbis}
\end{align}
Suppose we can show that for $I_T:= [0,T]$
\begin{equation} \label{eq:eq A} \begin{aligned} &
 \| \mathbb{A} \|  _{ L^2 (I_T, H ^{1, S}) + L^1 (I_T , H^1) } \le C (S,C_0) \epsilon ^2.
\end{aligned}\end{equation}
Then,   if $\epsilon _0 $  is small enough    and $\epsilon \in (0,\epsilon _0 )$, we obtain \eqref{4.5}
by $H ^{1, S}(\R ^3)\hookrightarrow W^{1,\frac{6}{5}}(\R ^3)$,
    by  \eqref{4.5id}, \eqref{L^2disbis} and   \eqref{eq:eq A}
 and by the   Strichartz estimates, which, for $P_c $ the orthogonal projection of $L^2$ onto $\mathcal{H}[0]$,  are valid for  $P_cH$
 by \cite{Y1} (here notice that all the terms in \eqref{eq:eq f} belong to $\mathcal{H}[0]$).

So now we prove  \eqref{eq:eq A}.
We have    for   $r-1\ge S>9/2$
\begin{equation} \label{4.88} \begin{aligned} &
\| {z}_j \overline{\textbf{Z}}^{\textbf{m}}
  [\overline{G}_{j \textbf{m}} (|z_j|^2  )  - \overline{G}_{j \textbf{m}}  ]   \|  _{  L^2 (I_T, H ^{1, S})} \le \| {z}_j \overline{\textbf{Z}}^{\textbf{m}}
    \|  _{  L^2 (I_T, \C)}    \|  \overline{G}_{j \textbf{m}} (|z_j|^2  )  - \overline{G}_{j \textbf{m}}    \|  _{  L^\infty (I_T, H ^{1, S})} \\& \le C_0  \epsilon
       \sup \{ \|  {G}_{j \textbf{m}} '(|z_j|^2  )\| _{\Sigma _{r}}:  |z_j|\le   \delta _{0} \}
        \|   z_j ^2     \|  _{  L^\infty (I_T, \C)} \le   C C_0^3 \epsilon ^{3} < c\epsilon .
\end{aligned}\end{equation}
We have for a fixed $c_1>0$
\begin{equation} \label{4.8} \begin{aligned} &
 \| \nabla _\eta E_P (\eta )   \|  _{L^1 (I_T , H^1) } = 2 \|  |\eta  | ^2 \eta    \|  _{L^1 (I_T , H^1) }   \le   c_1  \|    \eta    \|  _{L^\infty  (I_T , H^1) } \|    \eta    \| ^2  _{L^2 (I_T , L^6) }\le c_1 C_0^3   \epsilon ^3  .
\end{aligned}\end{equation}
 We finally show that for an arbitrarily preassigned $S>2$
\begin{align} &
  \| R_1  \|  _{ L^2 (I_T, H ^{1, S})} \le C (S,C_0) \epsilon ^2  \text{ for $R_1 =\nabla _\eta ( \resto  -E_P (\eta ))$.} \label{4.6}
\end{align}
$R_1$ is a sum of various term obtained from the expansion \eqref{eq:rest1}.
Let us start by showing
\begin{align} &
  \|  \nabla _{\overline{\eta}  }\resto ^{1,  2  }_{r,\infty} (z, \eta )   \|  _{ L^2 (I_T, H ^{1, S})} \le C (S,C_0) \epsilon ^2 . \label{4.7}
\end{align}
Recalling \eqref{eq:scalSymb2},  it is elementary to show that $\nabla _{\overline{\eta}  }\resto ^{1,  2  }_{r,\infty} (z, \eta ) =\mathbf{S}^{1,  1  }_{r,\infty} (z, \eta )$ and
\begin{equation*}  \begin{aligned} &
 \| \mathbf{S}^{1,  1  }_{r,\infty} (z, \eta )   \|  _{ L^2 (I_T, H ^{1, S})} \le C _1
 \|  (\|  \eta \| _{\Sigma   _{-r }}+  |z |)   \|  _{ L^\infty (I_T )}
 \|  \eta     \|  _{ L^2 (I_T, \Sigma   _{-r })} \\& \le C_2  \|  (\|  \eta \| _{H^1}+  |z |)   \|  _{ L^\infty (I_T )}
 \|  \eta     \|  _{ L^2 (I_T, L^6)}  \le C
 (S,C_0) \epsilon ^2 .
\end{aligned}\end{equation*}
We next show
\begin{align} &
  \|  \nabla _{\overline{\eta}  }\resto ^{0,  2N+5 }_{r,\infty} (z,\textbf{Z}  ,\eta )   \|  _{ L^2 (I_T, H ^{1, S})} \le C (S,C_0) \epsilon ^2 . \label{4.17}
\end{align}
We have, for a reminder   $\| O (  \| \eta  \|  _{ \Sigma _{-r}}^2   )  \|  _{ \Sigma _{ r}}\le C \| \eta  \|  _{ \Sigma _{-r}}^2 $   easily shown to satisfy an inequality like \eqref{4.17},
\begin{equation*}  \begin{aligned} &
   \nabla _{\overline{\eta}  }\resto ^{0,  2N+5 }_{r,\infty} (z,\textbf{Z}  ,\eta ) = \mathbf{S} ^{0,  2N+4 }_{r,\infty} (z,\textbf{Z}  ,\eta )  \\& =  \mathbf{S} ^{0,  2N+4 }_{r,\infty} (z,\textbf{Z}   ) + d _\eta  \mathbf{S} ^{0,  2N+4 }_{r,\infty} (z,\textbf{Z}  ,0 )
   \cdot \eta + O (  \| \eta  \|  _{ \Sigma _{-r}}^2   ) .
\end{aligned}\end{equation*}
We have  by Lemma \ref{lem:comb1}
\begin{equation*}  \begin{aligned} &
 \| \mathbf{S} ^{0,  2N+4 }_{r,\infty} (z,\textbf{Z}   )    \|  _{ L^2 (I_T, H ^{1, S})} \le C _1 \sup_{|z|\le C_0\epsilon} \| \mathbf{S} ^{0,  0 }_{r,\infty} (z,\textbf{Z}   )    \|  _{ \Sigma_{M'}}   \|   |\mathbf{Z} |  ^{2N+4 } \|  _{ L^2 (I_T)} \\& \le C_2     \|  z \|  _{ L^\infty (I)}  \sum _{j} \sum _{(\mu , \nu )\in M_j(N+1)}  \|  z ^{\mu}\overline{z}  ^{\nu}    \|  _{ L^\infty (I_T)}    \|  z ^{\mu}\overline{z}  ^{\nu}   \|  _{ L^2 (I_T)}
  \le C
 (S,C_0) \epsilon ^3.
\end{aligned}\end{equation*}
We have
\begin{equation*}  \begin{aligned} &
 \| d _\eta  \mathbf{S} ^{0,  2N+4 }_{r,\infty} (z,\textbf{Z}  ,0 )
   \cdot \eta  \|  _{ L^2 (I_T, H ^{1, S})}  \le C _1(S)    \|  \eta  \|  _{ L^2 (I_T, \Sigma_{-r})}\sup_{|z|\le C_0 \epsilon } \| d _\eta  \mathbf{S} ^{0,  2N+4 }_{r,\infty} (z,\textbf{Z}  ,0 )   \|  _{ \Sigma_{-r}  \to \Sigma_{r}}   \\& \le C _2(S)  \|  \eta  \|  _{ L^2 (I_T , L^6)} \sup_{|z|\le C_0 \epsilon } |  \textbf{Z} |   ^{2N+3}
  \le C
 (S,C_0) \epsilon  ^2
\end{aligned}\end{equation*}
Hence \eqref{4.17}  is proved.
 Other terms in $R_1$ can be bounded with similarly elementary  arguments,
 yielding  \eqref{4.6}.  Then  \eqref{4.88},   \eqref{4.8} and  \eqref{4.6}
 imply  \eqref{eq:eq A}.

   \qed

 Setting $M= M (2N+4)$, see  Def. \ref{def:setM},
  we now introduce a new variable $g$ setting
\begin{equation} \label{eq:def g} \begin{aligned} &
 g   =  \eta +Y \text{   with } Y:=
  \sum _{(\alpha , \beta )\in M   }   \overline{{z}}^\alpha  {{z}}^\beta   R_H^{+}( {\textbf{e} }     \cdot (\beta  -
\alpha  ))
  \overline{G}_{\alpha \beta  }     .
\end{aligned}\end{equation}

\begin{lemma}\label{lem:bound g}  Assume the hypotheses of Prop. \eqref{prop:mainbounds} and fix $S>9/2$. Then  there is a $c_1(S)>0$
 s.t.
 for any $C_0$ there is a   $\epsilon _0=\epsilon _0(C_0,S) >0$   such that for $\epsilon \in (0, \epsilon _0)$ in Theor.\ref{thm:small en} we have
\begin{equation} \label{bound:auxiliary}\| g
\| _{L^2 ([0,T], L^{2,-S}  )}\le c_1(S) \epsilon  .\end{equation}
\end{lemma}
\proof
We have \begin{equation} \label{eq:eq g} \begin{aligned} &
 \im \dot g =     H g  +   \mathbb{A} + \mathbf{T} \text{ where }      \textbf{T} :=\sum _j \left [\partial _{z_j}Y (\im \dot z_j- {e}_jz_j)+\partial _{\overline{z}_j}Y (\im \dot {\overline{z}}_j+ {e}_j\overline{z}_j)
 \right ] .
\end{aligned}\end{equation}
We then have
\begin{equation} \label{eq:exp g} \begin{aligned} &
 g(t)= e^{-\im Ht} \eta (0) + e^{-\im Ht} Y (0)-\im \int _0^t  e^{-\im H(t-s)}  (\mathbb{A} (s) +\textbf{T} (s) ) ds .
\end{aligned}\end{equation}
We have  for fixed constants  by \eqref{4.5id}  and \eqref{eq:eq A}   the following inequalities:
\begin{equation*}   \begin{aligned} &
\|  e^{-\im Ht} \eta (0)
\| _{L^2 ([0,T], L^{2,-S} )} \le c_2  \|  e^{-\im Ht} \eta (0)
\| _{L^2 ([0,T], L^ 6 )} \le  c_2 '  \|   \eta (0)
\| _{L^2} \le   c_3  \epsilon ; \\&\|  \int _0^t  e^{-\im H(t-s)}   \mathbb{A} (s)   ds
\| _{L^2 ([0,T], L^{2,-S} )} \le c_2   \|   \mathbb{A}
\| _{L^2 ([0,T],  H ^{1,S} )+ L^1 ([0,T],  H ^{1 } ) } \le   C(C_0,S)  \epsilon ^2 .
\end{aligned}\end{equation*}
 For a proof of the following  standard lemma see for instance to the proof of Lemma 5.4   \cite{CuRMP}.

\begin{lemma}\label{lem:lemg9}
Let $ \Lambda$ be a
compact subset of $( 0,\infty) $
and let $S >9/2$. Then  there exists a fixed $ c(S,\Lambda)$ s.t.
  for every $t \ge 0$ and  $\lambda \in  \Lambda$
\begin{equation*} \label{eq:lemg91}
\|  e^{-\im H  t}R_{ H }^{+}( \lambda )
P_cv_0 \|_{L^{ 2, - S}(\R^3)} \le c(S,\Lambda)\langle  t\rangle ^{-\frac 32} \| P_cv_0  \|_{L^{ 2,   S}(\R^3)}  \text{  for all $v_0\in L^{ 2,   S}(\R^3)$} .
\end{equation*}
\end{lemma}\qed

\noindent  By   {Lemma} \ref{lem:lemg9}, by      \eqref{4.5id}  and by $ {G}_{\alpha \beta  }=P_c {G}_{\alpha \beta  }$  we have
\begin{equation*}   \begin{aligned} &
\|  e^{-\im Ht} Y (0)
\| _{L^2 ([0,T], L^{2,-S} )} \le  \sum _{(\alpha , \beta )\in M   } | {{z}}^\alpha  (0)  {{z}}^\beta  (0)|\  \|  e^{-\im Ht}
   R_H^{+}( {\textbf{e} }     \cdot (\beta  -
\alpha  ))
  \overline{G}_{\alpha \beta  }
\| _{L^2 ([0,T], L^{2,-S} )}
\\&  \le ( \sharp   M  ) c_2 \epsilon ^2  \|  \langle  t\rangle ^{-\frac 32}
\| _{L^2 (0,T) }  c(S,\Lambda)    \|  \overline{G}_{\alpha \beta  }   \|_{L^{ 2,   S} }
 \le   C(N,C_0,S)  \epsilon ^2
\end{aligned}\end{equation*}
with $\sharp   M$ the cardinality of $M$ and a fixed $c_2$ and where   the following set $\Lambda$ is as in {Lemma} \ref{lem:lemg9},
\begin{equation}  \label{eq:FGR22}    \begin{aligned}   &   \Lambda := \{ (\nu - \mu ) \cdot \mathbf{e}    :   (\mu  , \nu )\in M \}  .
\end{aligned}    \end{equation}
We finally consider, for definiteness  (the term $\partial _{\overline{z}_j}Y (\im \dot {\overline{z}}_j+ {e}_j\overline{z}_j)$ can be treated similarly)
\begin{equation} \label{eq:bdY1}  \begin{aligned} &
\|  \int _0^t  e^{-\im H(t-s)}  R_H^{+}( {\textbf{e} }     \cdot (\beta  -
\alpha  ))
  \overline{G}_{\alpha \beta  }  \partial _{z_j}Y(s) (\im \dot z_j- {e}_jz_j)(s)   ds
\| _{L^2 ([0,T], L^{2,-S} )}\\&  \le c(S,\Lambda)
 \sum _{(\alpha , \beta )\in M   } \|  {G}_{\alpha \beta  } \| _{L^{2, S}} \beta _j
\|  \int _0^t   \langle t-s\rangle ^{-\frac{3}{2}} |\frac{ \overline{{z}} ^\alpha  (s)  {{z}}^\beta  (s)}{z_j(s)} (\im \dot z_j- {e}_jz_j)(s) |  ds
\| _{L^2 (0,T) }  \\&  \le c( S,\Lambda)   c_2  \sum _{(\alpha , \beta )\in M   }  \beta _j
\|     \frac{ \overline{{z}} ^\alpha  (s)  {{z}}^\beta   }{z_j } (\im \dot z_j- {e}_jz_j)
\| _{L^2 (0,T) }  ,
\end{aligned}\end{equation}
for  fixed $c_2$.
We have
\begin{equation}\label{eq:FGR01} \begin{aligned} &
\im \dot z _j
=(1 +   \varpi  _j(|z_j|^2)) (e_jz_j+
\partial _{\overline{z}_j}\mathcal{Z}_0(|z _1|^2, ..., |z _n|^2) +
\partial _{  \overline{z} _j}   \mathcal{R}) \\  &
+  (1 +   \varpi  _j(|z_j|^2))[   \sum _{ (\mu , \nu )\in M }    \nu _j  \frac{z  ^{\mu }
 \overline{ {z }}^ { {\nu}  } }{\overline{z}_j}
\langle \eta  ,
  {G}_{\mu \nu }  \rangle      +  \sum _{(\mu ', \nu' )\in M}    \mu _j ' \frac{z  ^{\nu '}
 \overline{ {z }}^ { {\mu}'  } }{\overline{z}_j}
\langle \overline{\eta}  ,
  \overline{{G}}_{\mu' \nu' }  \rangle ]\\&
+(1 +   \varpi  _j(|z_j|^2))[\sum_{\mathbf m \in \mathcal M_j (2N+3)}|z_j|^2\mathbf Z^{\mathbf m}\<G_{j\mathbf m}',\eta\>+ z_j^2 \overline{\mathbf{Z}}^{\mathbf m}\<\overline G_{j\mathbf m}',\overline \eta\>]  .
\end{aligned}  \end{equation}
To bound \eqref{eq:bdY1} we  substitute  $(\im \dot z _j
-e_jz_j)$ by the other terms in \eqref{eq:FGR01}  in the last line of \eqref{eq:bdY1} . So for example we have $\partial _{\overline{z}_j}\mathcal{Z}_0(|z _1|^2, ..., |z _n|^2)  \sim z_j O(\epsilon)$
which by \eqref{L^2disbis} yields
\begin{equation*}   \begin{aligned} &
  \beta _j
\|     \frac{ \overline{{z}} ^\alpha     {{z}}^\beta   }{z_j }
 \partial _{\overline{z}_j}\mathcal{Z}_0(|z _1|^2, ..., |z _n|^2) \| _{L^2 (0,T) }  \le C(C_0) \epsilon  \|     \overline{{z}} ^\alpha     {{z}}^\beta     \| _{L^2 (0,T) } \le C(C_0)C_0 \epsilon ^2 .
\end{aligned}\end{equation*}
 For  $(\mu , \nu )\in M$ we have  in $(0,T)$
\begin{equation*}   \begin{aligned} &
  \beta _j \nu _j
\|     \frac{ \overline{{z}} ^\alpha     {{z}}^\beta   }{z_j }  \frac{z  ^{\mu }
 \overline{ {z }}^ { {\nu}  } }{\overline{z}_j}
 \langle \eta  ,
  {G}_{\mu \nu }  \rangle \| _{L^2 _t }  \le  \beta _j \nu _j
  \|     \frac{ \overline{{z}} ^\alpha     {{z}}^\beta   }{z_j }  \frac{z  ^{\mu }
 \overline{ {z }}^ { {\nu}  } }{\overline{z}_j} \| _{L^\infty _t }
 \|  {G}_{\mu \nu }   \| _{L ^{\frac{6}{5}}   } \| \eta  \| _{L^\infty _t  L^6  }
 \le C(C_0) \epsilon ^2.
\end{aligned}\end{equation*}
A similar argument works for the terms in the 2nd summation in the 2nd line of
\eqref{eq:FGR01}. Finally
\begin{equation*}   \begin{aligned} &
  \beta _j
\|     \frac{ \overline{{z}} ^\alpha     {{z}}^\beta   }{z_j }
 \partial _{\overline{z}_j}\mathcal{R}  \| _{L^2 (0,T) }
\le  \beta _j
\|     \frac{ \overline{{z}} ^\alpha     {{z}}^\beta   }{z_j } \| _{L^\infty (0,T) }
 \|  \partial _{\overline{z}_j}\mathcal{R}  \| _{L^2 (0,T) }
\le C(C_0)   \epsilon ^3
\end{aligned}\end{equation*}
  is a consequence of   the bound
\begin{equation} \label{eq:bd R}  \begin{aligned} &
   \|
 \partial _{\overline{z}_j}\mathcal{R}  \| _{L^p (0,T)   }  \le  C(C_0)   \epsilon ^2
 \text{ for any $p\in [1, \infty ]$.}
\end{aligned}\end{equation}
Here we need to check  \eqref{eq:bd R} term by term   for the sum in the r.h.s. of \eqref{eq:rest1}. This is straightforward  using     \eqref{eq:scalSymb},  \eqref{eq:scalSymb2}  and \eqref{eq:opSymb} and the fact,
 stated in Lemma  \ref{lem:birkflow1}, that
 $G_{2\textbf{m} ij }    $   and    $G_{d ij }  $ are   $\textbf{S} ^{0,  0} _{r,\infty }  $.

 \qed

We turn now to the   Fermi Golden Rule (FGR).  We   substitute \eqref{eq:def g} in \eqref{eq:FGR01}
getting
\begin{equation}\label{eq:FGR02}
\begin{aligned}
\im \dot z _j
&= (1 +   \varpi  _j(|z_j|^2))(  e_j z_j+
\partial _{\overline{z}_j}\mathcal{Z}_0(|z _1|^2, ..., |z _n|^2) ) \\&   -\sum _{ \substack{  (\mu , \nu )\in M \\
(\alpha , \beta )\in M}}     \nu _j  \frac{z  ^{\mu +\beta}
 \overline{ {z }}^ { {\nu} +\alpha } }{\overline{z}_j}
\langle  R_H^{+}( {\textbf{e} }     \cdot (\beta -
\alpha ))  \overline{G}_{\alpha \beta}  ,
  {G}_{\mu \nu }  \rangle  \\&
-  \sum _{ \substack{  (\mu ', \nu' )\in M \\
(\alpha ', \beta ')\in M}}         \mu _j ' \frac{z  ^{\nu '+ \alpha '}
 \overline{ {z }}^ { {\mu}'+\beta '  } }{\overline{z}_j}
\langle  R_H^{-}( {\textbf{e} }     \cdot (\beta '-
\alpha' ))  {G}_{\alpha ' \beta '}  ,
  \overline{{G}}_{\mu' \nu' }  \rangle      +
    \mathcal{F} _j  \text{, where}
\end{aligned}
\end{equation}
\begin{align} &
 \mathcal{F} _j  :=(1 +   \varpi  _j(|z_j|^2))
\partial _{  \overline{z} _j}   \mathcal{R} + \varpi  _j(|z_j|^2) [   \sum _{ (\mu , \nu )\in M }    \nu _j  \frac{z  ^{\mu }
 \overline{ {z }}^ { {\nu}  } }{\overline{z}_j}
\langle \eta  ,
  {G}_{\mu \nu }  \rangle      +  \sum _{(\mu ', \nu' )\in M}    \mu _j ' \frac{z  ^{\nu '}
 \overline{ {z }}^ { {\mu}'  } }{\overline{z}_j}
\langle \overline{\eta}  ,
  \overline{{G}}_{\mu' \nu' }  \rangle  ] \nonumber
\\&+     \sum _{ (\mu , \nu )\in M }    \nu _j  \frac{z  ^{\mu }
 \overline{ {z }}^ { {\nu}  } }{\overline{z}_j}
\langle g  ,
  {G}_{\mu \nu }  \rangle      +  \sum _{(\mu ', \nu' )\in M}    \mu _j ' \frac{z  ^{\nu '}
 \overline{ {z }}^ { {\mu}'  } }{\overline{z}_j}
\langle \overline{g}  ,
  \overline{{G}}_{\mu' \nu' }  \rangle \label{eq:FGRrem1}\\&
+(1 +   \varpi  _j(|z_j|^2))[\sum_{\mathbf m \in \mathcal M_j (2N+3)}|z_j|^2\mathbf Z^{\mathbf m}\<G_{j\mathbf m}',\eta\>+ z_j^2 \overline{\mathbf{Z}}^{\mathbf m}\<\overline G_{j\mathbf m}',\overline \eta\>]  .\nonumber
\end{align}
 We now introduce the new variable   $\zeta
 $  defined by
 \begin{equation}\label{eq:FGR21}  \begin{aligned}   &
z _j -\zeta _j =-\sum _{ \substack{  (\mu , \nu )\in M \\
(\alpha , \beta )\in M}}       \frac{\nu _jz  ^{\mu +\beta}
 \overline{ {z }}^ { {\nu} +\alpha } }{((\mu - \nu)\cdot \mathbf{e}-(\alpha - \beta )\cdot \mathbf{e}  )\overline{z}_j}
\langle  R_H^{+}( {\textbf{e} }     \cdot (\beta -
\alpha ))  \overline{G}_{\alpha \beta}  ,
  {G}_{\mu \nu }  \rangle  \\&   - \sum _{ \substack{  (\mu ', \nu' )\in M \\
(\alpha ', \beta ')\in M}}         \frac{\mu _j ' z  ^{\nu '+ \alpha '}
 \overline{ {z }}^ { {\mu}'+\beta '  } }{((\alpha ' - \beta ')\cdot \mathbf{e}-(\mu ' - \nu ' )\cdot \mathbf{e}  )\overline{z}_j}
\langle  R_H^{-}( {\textbf{e} }     \cdot (\beta '-
\alpha' ))  {G}_{\alpha ' \beta '}  ,
  \overline{{G}}_{\mu' \nu' }  \rangle
  \end{aligned}
\end{equation}
where we are summing only on pairs where the formula makes sense (i.e. only
on pairs not in a same set $M_L$ for an $L\in   \Lambda$, see \eqref{eq:FGR23} below).
It is easy to see that
\begin{equation}  \label{equation:FGR3} \begin{aligned}   & \| \zeta  -
 z  \| _{L^2(0,T)} \le  c(N,C_0) \epsilon ^2  \text{  and }  \| \zeta  -
 z \| _{L^\infty (0,T)} \le  c(N,C_0) \epsilon ^2 .
\end{aligned}
\end{equation}

 Recall now the set $\Lambda = \{ (\nu - \mu ) \cdot \mathbf{e}    :   (\mu  , \nu )\in M \}      $ defined in \eqref{eq:FGR22}.
For any $L\in \Lambda $ set
\begin{equation}  \label{eq:FGR23}    \begin{aligned}   &   M_L  := \{  (\mu  , \nu )\in M : (\nu - \mu ) \cdot \mathbf{e}   =L  \}  .
\end{aligned}    \end{equation}
We then get
\begin{equation}\label{eq:FGR251} \begin{aligned} &
\im \dot \zeta _j
=(1+\varpi(|z_j|^2))(e_j\zeta _j+
\partial _{\overline{j} }\mathcal{Z}_0(|\zeta _1|^2, ..., |\zeta _n|^2))\\& - \sum _{L\in \Lambda}  \sum _{ \substack{  (\mu , \nu )\in M _L\\
(\alpha , \beta )\in M _L }}     \nu _j  \frac{\zeta  ^{\mu +\beta}
 \overline{ {\zeta }}^ { {\nu} +\alpha } }{\overline{\zeta}_j}
\langle  R_H^{+}( {\textbf{e} }     \cdot (\beta -
\alpha ))  \overline{G}_{\alpha \beta}  ,
  {G}_{\mu \nu }  \rangle  \\&   - \sum _{L\in \Lambda}  \sum _{ \substack{  (\mu ', \nu' )\in M _L\\
(\alpha ', \beta ')\in M_L}}         \mu _j ' \frac{\zeta  ^{\nu '+ \alpha '}
 \overline{ {\zeta }}^ { {\mu}'+\beta '  } }{\overline{\zeta}_j}
\langle  R_H^{-}( {\textbf{e} }     \cdot (\beta '-
\alpha' ))  {G}_{\alpha ' \beta '}  ,
  \overline{{G}}_{\mu' \nu' }  \rangle      +
    \mathcal{G} _j  ,
\end{aligned}  \end{equation}
where for some $A_{k\alpha\beta\mu\nu}, B_{k\alpha\beta\mu\nu}$ we have
\begin{equation}  \label{eq:restG}
\begin{aligned}
&  \mathcal{G} _j =   \mathcal{F} _j +(1+\varpi(|z_j|^2))[  {\partial} _{ \overline{j}}\mathcal{Z}_0(|z_1|^2, ..., |z_n|^2) -{\partial} _{ \overline{j}}\mathcal{Z}_0(|\zeta _1|^2, ..., |\zeta _n|^2)]\\&
-e_j \varpi(|z_j|^2)[\sum _{ \substack{  (\mu , \nu )\in M \\
(\alpha , \beta )\in M}}       \frac{\nu _jz  ^{\mu +\beta}
 \overline{ {z }}^ { {\nu} +\alpha } }{((\mu - \nu)\cdot \mathbf{e}-(\alpha - \beta )\cdot \mathbf{e}  )\overline{z}_j}
\langle  R_H^{+}( {\textbf{e} }     \cdot (\beta -
\alpha ))  \overline{G}_{\alpha \beta}  ,
  {G}_{\mu \nu }  \rangle  \\&   + \sum _{ \substack{  (\mu ', \nu' )\in M \\
(\alpha ', \beta ')\in M}}         \frac{\mu _j ' z  ^{\nu '+ \alpha '}
 \overline{ {z }}^ { {\mu}'+\beta '  } }{((\alpha ' - \beta ')\cdot \mathbf{e}-(\mu ' - \nu ' )\cdot \mathbf{e}  )\overline{z}_j}
\langle  R_H^{-}( {\textbf{e} }     \cdot (\beta '-
\alpha' ))  {G}_{\alpha ' \beta '}  ,
  \overline{{G}}_{\mu' \nu' }  \rangle]\\&
+\sum_k\sum _{ \substack{  (\mu , \nu )\in M \\(\alpha , \beta )\in M}} (\im \dot z_k-e_kz_k)\frac{z^{\mu+\beta}\overline z^{\nu+\alpha}}{\overline z_j}A_{k\alpha\beta\mu\nu}+\overline{(\im \dot z_k-e_kz_k)}\frac{z^{\mu+\beta}\overline z^{\nu+\alpha}}{\overline z_j}B_{k\alpha\beta\mu\nu}.
\end{aligned}
\end{equation}
 \begin{lemma}\label{lemma:E} There are fixed $c_4$  and $ \epsilon _0>0$ such that
 for $\epsilon \in (0, \epsilon _0)$ we have
 \begin{equation}\label{4.33}
  \begin{aligned}  & \| \mathcal{G}_j \overline{\zeta} _j\| _{L^1 [0,T]}\le (1+C_0) c_4 \epsilon ^2   .
\end{aligned}
\end{equation}
\end{lemma}
\proof  We consider separately the terms in the  r.h.s. of \eqref{eq:restG} and \eqref{eq:FGRrem1}. By \eqref{L^inftydiscrete}, \eqref{eq:bd R}   \eqref{equation:FGR3}
\begin{equation*}
  \begin{aligned}  & \| \partial _{  \overline{z} _j}   \mathcal{R}  \overline{\zeta} _j\| _{L^1_t[0,T]}\le C(C_0) \epsilon ^ 3 .
\end{aligned}
\end{equation*}
 For fixed constants $c_2$ and $c_3$, by \eqref{Strichartzradiation} and \eqref{bound:auxiliary}, we have
\begin{equation} \label{5.33}
  \begin{aligned}  &
\| \frac{z  ^{\mu }
 \overline{ {z }}^ { {\nu}  }  \overline{\zeta} _j}{\overline{z}_j}
\langle g  ,
  {G}_{\mu \nu } \rangle  \| _{L^1 [0,T]} \le c_2  \| \frac{z  ^{\mu }
 \overline{ {z }}^ { {\nu}  }  \overline{\zeta} _j}{\overline{z}_j} \| _{L^2 [0,T]}
\|  g      \| _{L^2([0,T], L ^{2,-S})} \le  c_3 C_0 \epsilon ^2.
 \end{aligned}
\end{equation}
To get  \eqref{5.33} we exploit Lemma \ref{lem:bound g} and the following bound:
\begin{equation} \label{5.331}
  \begin{aligned}  &    \nu _j \| \frac{z  ^{\mu }
 \overline{ {z }}^ { {\nu}  }  \overline{\zeta} _j}{\overline{z}_j} \| _{L^2 [0,T]}
 \le  \nu _j\| z  ^{\mu }
 \overline{ {z }}^ { {\nu}  }   \| _{L^2 [0,T]} +  \nu _j \| \frac{z  ^{\mu }
 \overline{ {z }}^ { {\nu}  }  }{\overline{z}_j} \| _{L^\infty [0,T]}  \|  {\zeta} _j-\overline{z}_j  \| _{L^2 [0,T]}\\& \le  c_ 2 C_0 \epsilon  +  C(C_0) \epsilon ^3
\end{aligned}
\end{equation}
for fixed $c_2$, where we used \eqref{L^2disbis} and \eqref{equation:FGR3}.
Terms such as \eqref{5.33},
that is the terms  from the 2nd term in the r.h.s. of \eqref{eq:FGRrem1},
are the ones  responsible for the $C_0c_4\epsilon ^2$ in
\eqref{4.33}, where $C_0$ could be large. The other terms   are $O(\epsilon ^2)$ with fixed constants, if $\epsilon  _0$
is small enough.

\noindent By \eqref{Strichartzradiation} and \eqref{L^2discrete}, for $\mathbf m \in \mathcal M_j (2N+4)$ we have
\begin{equation}\label{eq:F_jlast}
\||z_j|^2\mathbf{Z}^{\mathbf m}\<G_{j\mathbf m}',\eta\>\overline{\zeta}_j\|_{L^1[0,T]}\leq c_4\|z_j\zeta_j\|_{L^\infty}\|z_j \mathbf{Z}^{\mathbf m}\|_{L^2[0,T]}\|\eta\| _{L^2([0,T], L ^{2,-S})}\leq C(C_0)\epsilon^4.
\end{equation}
It is easy to see by  \eqref{equation:FGR3} that
\begin{equation} \label{eq:skip3}
  \begin{aligned}  & \|     \overline{{\zeta}} _j (\text{2nd--6th\ line\ of\ r.h.s.\eqref{eq:restG}})     \| _{L^2 [0,T]}\le C(C_0) \epsilon ^ 3,
\end{aligned}
\end{equation}
see Lemma 4.11 \cite{Cu4},
\begin{equation} \label{eq:skip4}
  \begin{aligned}  & \| [{\partial} _{ \overline{j}}\mathcal{Z}_0(|z_1|^2, ..., |z_n|^2) -{\partial} _{ \overline{j}}\mathcal{Z}_0(|\zeta _1|^2, ..., |\zeta _n|^2) ]  \overline{\zeta} _j\| _{L^2 [0,T]}\le C(C_0) \epsilon ^ 3,
\end{aligned}
\end{equation}
see Lemma 4.10 \cite{Cu4}.
   Finally we have for $(\mu , \nu )\in M$

\begin{equation*}
  \begin{aligned}  &
 \|  \varpi  _j(|z_j|^2)      \nu _j  \frac{z  ^{\mu }
 \overline{ {z }}^ { {\nu}  } }{\overline{z}_j}
\langle \eta  ,
  {G}_{\mu \nu }  \rangle      \zeta _j\| _{L^1_t} \le \|  \varpi  _j(|z_j|^2)      \nu _j   z  ^{\mu }
 \overline{ {z }}^ { {\nu}  }
\langle \eta  ,
  {G}_{\mu \nu }  \rangle       \| _{L^1_t} \\& + \|  \varpi  _j(|z_j|^2)      \nu _j  \frac{z  ^{\mu }
 \overline{ {z }}^ { {\nu}  } }{\overline{z}_j}
\langle \eta  ,
  {G}_{\mu \nu }  \rangle      (\zeta _j -z_j)\| _{L^1_t}  \le   C(C_0) \epsilon ^ 3
\end{aligned}
\end{equation*}
by $\varpi  _j(|z_j|^2) =O(|z_j|^2)$, \eqref{Strichartzradiation}--\eqref{L^inftydiscrete}  and \eqref{equation:FGR3}.
This completes the proof of Lemma \ref{lemma:E}.
\qed

We now consider
\begin{align} &
   2 ^{-1} \frac{d}{dt}\sum _j   |e_j|\  |  \zeta _j| ^2
=-  \sum _j e_j \overbrace{\Im [ (1+\varpi(|z_j|^2)) e_j |\zeta _j|^2+
\partial _{\overline{\zeta}_j}\mathcal{Z}_0(|\zeta _1|^2, ..., |\zeta _n|^2) \overline{\zeta} _j ]}^{0}
 -  \sum _j e_j  \Im [
  \mathcal{G} _j  \overline{\zeta} _j ]
\nonumber \\&
 + \sum _{L\in \Lambda} \Im [  \sum _{ \substack{  (\mu , \nu )\in M _L\\
(\alpha , \beta )\in M _L }}     \nu  \cdot \mathbf{e}  \   \zeta  ^{\mu +\beta}
 \overline{ {\zeta }}^ { {\nu} +\alpha }
\langle  R_H^{+}(L)  \overline{G}_{\alpha \beta}  ,
  {G}_{\mu \nu }  \rangle  \label{eq:FGR25}\\&   +    \sum _{ \substack{  (\mu ', \nu' )\in M _L\\
(\alpha ', \beta ')\in M_L}}         \mu   ' \cdot \mathbf{e}\  \zeta  ^{\nu '+ \alpha '}
 \overline{ {\zeta }}^ { {\mu}'+\beta '  }
\langle  R_H^{-}( L)  {G}_{\alpha ' \beta '}  ,
  \overline{{G}}_{\mu' \nu' }  \rangle    ]  . \nonumber
\end{align}
We can now substitute
$
  R_H^{\pm }( L) = P.V. \frac{1}{H-L}\pm \im \pi \delta ( H-L) .$
 \begin{lemma}
  \label{lem:can1}   The contributions to \eqref{eq:FGR25} from the $P.V. \frac{1}{H-L} $  cancel out:
  \begin{equation} \label{eq:can2} \begin{aligned} &
 \Im [  \sum _{ \substack{  (\mu , \nu )\in M _L\\
(\alpha , \beta )\in M _L }}     \nu  \cdot \mathbf{e}  \   \zeta  ^{\mu +\beta}
 \overline{ {\zeta }}^ { {\nu} +\alpha }
\langle  P.V. \frac{1}{H-L}  \overline{G}_{\alpha \beta}  ,
  {G}_{\mu \nu }  \rangle  \\&   +    \sum _{ \substack{  (\mu ', \nu' )\in M _L\\
(\alpha ', \beta ')\in M_L}}         \mu   ' \cdot \mathbf{e}\  \zeta  ^{\nu '+ \alpha '}
 \overline{ {\zeta }}^ { {\mu}'+\beta '  }
\langle P.V. \frac{1}{H-L}  {G}_{\alpha ' \beta '}  ,
  \overline{{G}}_{\mu' \nu' }  \rangle     ] =0.
\end{aligned}\end{equation}

\end{lemma}
 \proof We set $(\alpha ', \beta ')=(\mu  , \nu  )$ and
 $(\mu ', \nu' )=(\alpha  , \beta  )$ in the 2nd line of \eqref{eq:can2}.
 With these choices
 \begin{equation*} \begin{aligned} &     \mu   ' \cdot \mathbf{e}\  \zeta  ^{\nu '+ \alpha '}
 \overline{ {\zeta }}^ { {\mu}'+\beta '  }
\langle P.V. \frac{1}{H-L}  {G}_{\alpha ' \beta '}  ,
  \overline{{G}}_{\mu' \nu' }  \rangle =  \alpha \cdot \mathbf{e}\  \zeta  ^{\mu +\beta}
 \overline{ {\zeta }}^ { {\nu} +\alpha }
\langle  P.V. \frac{1}{H-L}  \overline{G}_{\alpha \beta}  ,
  {G}_{\mu \nu }  \rangle
\end{aligned}\end{equation*}
 Then 2 times the  l.h.s. of  \eqref{eq:can2} becomes
\begin{equation*} \label{eq:can3} \begin{aligned} &2 \Im   [  \sum _{ \substack{  (\mu , \nu )\in M _L\\
(\alpha , \beta )\in M _L }}   (\alpha +  \nu )   \cdot \mathbf{e}  \  \zeta  ^{\mu +\beta}
 \overline{ {\zeta }}^ { {\nu} +\alpha }
\langle  P.V. \frac{1}{H-L}  \overline{G}_{\alpha \beta}  ,
  {G}_{\mu \nu }  \rangle     ]  =   \sum _{ \substack{  (\mu , \nu )\in M _L\\
(\alpha , \beta )\in M _L }} \Im   [    (\alpha +  \nu )   \cdot \mathbf{e}  \  \zeta  ^{\mu +\beta}
 \overline{ {\zeta }}^ { {\nu} +\alpha } \times \\& \times
\langle  P.V. \frac{1}{H-L}  \overline{G}_{\alpha \beta}  ,
  {G}_{\mu \nu }  \rangle    +  (\mu  +  \beta )   \cdot \mathbf{e}  \  \overline{\zeta}  ^{\mu +\beta}
  { {\zeta }}^ { {\nu} +\alpha }
\langle  P.V. \frac{1}{H-L}  \overline{G}_{\mu \nu }  ,
  {G}_{\alpha \beta }  \rangle     ] \\& =
 \Im   [  \sum _{ \substack{  (\mu , \nu )\in M _L\\
(\alpha , \beta )\in M _L }}   (\alpha +  \nu )   \cdot \mathbf{e}  \left (   \zeta  ^{\mu +\beta}
 \overline{ {\zeta }}^ { {\nu} +\alpha }
\langle  P.V. \frac{1}{H-L}  \overline{G}_{\alpha \beta}  ,
  {G}_{\mu \nu }  \rangle  +\text{c.c.} \right )    ]  =0
\end{aligned}\end{equation*}
  where we exploited the fact that if  $ (\mu , \nu )$ and
   $ (\alpha , \beta )$ both belong to $M_L$ then $(\alpha +  \nu )   \cdot \mathbf{e}= (\mu  +  \beta )   \cdot \mathbf{e}$.
\qed

 \begin{lemma}
  \label{lem:pos1}   Set for any $L\in \Lambda$
  \begin{equation} \label{eq:pos2} \begin{aligned} &
  G _L(\zeta ) :=  \sqrt{\pi} \sum _{   (\mu , \nu )\in M _L } \zeta  ^{\mu }
 \overline{ {\zeta }}^ { {\nu}   } {G}_{\mu \nu }   .
\end{aligned}\end{equation}
  Then we have
  \begin{equation} \label{eq:pos3} \begin{aligned} &
 \Im [ \im \pi  \sum _{ \substack{  (\mu , \nu )\in M _L\\
(\alpha , \beta )\in M _L }}     \nu  \cdot \mathbf{e}  \   \zeta  ^{\mu +\beta}
 \overline{ {\zeta }}^ { {\nu} +\alpha }
\langle \delta ({H-L})  \overline{G}_{\alpha \beta}  ,
  {G}_{\mu \nu }  \rangle  \\&  +\im \pi     \sum _{ \substack{  (\mu ', \nu' )\in M _L\\
(\alpha ', \beta ')\in M_L}}         \mu   ' \cdot \mathbf{e}\  \zeta  ^{\nu '+ \alpha '}
 \overline{ {\zeta }}^ { {\mu}'+\beta '  }
\langle \delta ({H-L})  {G}_{\alpha ' \beta '}  ,
  \overline{{G}}_{\mu' \nu' }  \rangle     ]   =  L \langle \delta ({H-L})  \overline{G} _L(\zeta )  ,
    {G} _L(\zeta )   \rangle      \ge 0.
\end{aligned}\end{equation}

\end{lemma}
 \proof  First of all the last inequality
 is a consequence of the formula
 \begin{equation*}   \begin{aligned} &   \langle    F
  ,
\delta ( H -L)   \overline{G}\rangle =   \frac{1}{16\pi ^2\sqrt{L}} \int _{|\xi |= \sqrt{{L}}} \widehat{F } (\xi
)\overline{\widehat{G  } } (\xi ) d\sigma (\xi )
\end{aligned}\end{equation*}
with $\widehat{F}$  and $\widehat{G}$   the   Fourier transforms  of
$F$  and $G$ associated to $ H$, see Prop. 2.2 Ch. 9 \cite{taylor}.

 To prove the first equality in \eqref{eq:pos3}
 set $(\alpha ', \beta ')=(\alpha  , \beta  )$ and
 $(\mu ', \nu' )=(\mu  , \nu  )$ in the 2nd line of \eqref{eq:pos3}.
 Then the l.h.s. of  \eqref{eq:pos3} equals
 \begin{equation*} \label{eq:pos4} \begin{aligned} &
 \pi \Re [    \sum _{ \substack{  (\mu , \nu )\in M _L\\
(\alpha , \beta )\in M _L }}    \overbrace{( \nu  -\mu )  \cdot \mathbf{e}} ^{L} \   \zeta  ^{\mu +\beta}
 \overline{ {\zeta }}^ { {\nu} +\alpha }
\langle \delta ({H-L})  \overline{G}_{\alpha \beta}  ,
  {G}_{\mu \nu }  \rangle       ]  = L \langle \delta ({H-L})  \overline{G} _L(\zeta )  ,
    {G} _L(\zeta )   \rangle     .
\end{aligned}\end{equation*}

     \qed

From \eqref{eq:FGR25} and Lemmas \ref{lem:can1}--\ref{lem:pos1} we obtain

\begin{equation}\label{eq:FGR31} \begin{aligned} &
    2\sum _{L\in \Lambda}L \langle \delta ({H-L})  \overline{G} _L(\zeta )  ,
    {G} _L(\zeta )   \rangle
= \frac{d}{dt}\sum _j   |e_j|\  |  \zeta _j| ^2
 +2  \sum _j e_j  \Im [
  \mathcal{G} _j  \overline{\zeta} _j ]
     .
\end{aligned}  \end{equation}
We are able to restate,   precisely this time, hypothesis (H4).

\begin{itemize}
\item[(H4)]  We assume
that for some fixed constants   we have:
\begin{equation}\label{eq:FGR}      \sum _{
L \in \Lambda  }
    \langle \delta ( {H}-L)
   \overline{G} _L(\zeta ),   {G} _L(\zeta )\rangle
 \sim \sum _{ ( \mu , \nu )
   \in M}  | \zeta ^{\mu +\nu }  | ^2\quad  \text{ for all $\zeta \in
\mathbb{C}^n$  with $|\zeta| \le 1$}.
\end{equation}
\end{itemize}
We now  complete the proof of Prop. \ref{prop:mainbounds}.
 We \textit{claim}  we have for a fixed $c$
 \begin{equation}\label{eq:conch}
   |\sum _j   |e_j| ( |  \zeta _j(t)| ^2 - |  \zeta _j(0)| ^2 )|\le c \epsilon ^2.
 \end{equation}
 Indeed, first of all   we have $|  \zeta _j(0)|\le c' \epsilon$
 by $\epsilon :=\| u _0\| _{H^1} $. Observe that for $(z', \eta ')$ the initial
coordinates in Lemma \ref{lem:systcoo},  by Proposition \ref{prop:bddst}
 and Lemma \ref{lem:contcoo} it is easy to see that we have
 \begin{equation*}  \begin{aligned} &  \epsilon ^2>\| u_0 \| _{L^2}^2=\| u(t) \| _{L^2}^2 = \| (\sum_{j=1}^n  z'_j(t)  \phi _j   +  \eta'(t) ) +(\sum_{j=1}^n q_{j z'_j(t) } +(R[z'(t) ] -1)\eta'(t) ) \| _{L^2}^2\\& =\sum_{j=1}^n |  z'_j(t)  | ^2 +\|\eta '(t)   \| _{L^2}^2  + O( | z'(t) | ^6+| z'(t) | ^4 \|\eta '  (t) \| _{L^2}^2 ).  \end{aligned}
\end{equation*}
 This gives   the following version of \eqref{eq:coo11}:
 \begin{equation}\label{eq:conch1}
  \sum_{j=1}^n |  z'_j(t)  | ^2 +\|\eta '(t)   \| _{L^2}^2 \le 2 \epsilon ^2.
 \end{equation}
 This yields an analogous formula for  the last system of coordinates
   $(z , \eta  )$ in \eqref{eq:diff2}. Finally, this yields  the following
   inequality for the variables $\zeta$ introduced in \eqref{eq:FGR21}:
  \begin{equation}\label{eq:conch2}
  \sum_{j=1}^n |  \zeta _j(t)  | ^2   \le 3 \epsilon ^2.
 \end{equation}
 Hence the \textit{claim} \eqref{eq:conch} is proved. By Lemma \ref{lemma:E}, by
 the hypothesis \eqref{eq:FGR},   by
 \eqref{equation:FGR3} and by  \eqref{eq:conch},
 for $\epsilon \in (0, \epsilon _0)$ with $\epsilon _0$
small enough we obtain for a fixed $c$
 \begin{equation} \label{eq:crunch}\begin{aligned}&  \sum _{ ( \mu , \nu )
   \in M}  \| z^{\mu +\nu } \| _{L^2(0,t)}^2\le c
\epsilon ^2+ cC_0\epsilon ^2.\end{aligned}
\end{equation}
\eqref{eq:crunch} tells us that $\| z ^{\mu +\nu } \| _{L^2(0,t)}^2\lesssim
  C_0^2\epsilon ^2$ implies  $\| z ^{\mu +\nu } \|
_{L^2(0,t)}^2\lesssim \epsilon ^2+ C_0\epsilon ^2$ for all $( \mu , \nu )
   \in M$.   This means that we can take
$C_0\sim 1$. This completes the proof of  Prop. \ref{prop:mainbounds}.

\qed

\subsection{Proof   of the asymptotics \eqref{eq:small en32}} \label{subsec:prop1}

We write \eqref{eq:eq f} in the form $\im \dot \eta =   -\Delta  \eta    + V\eta + \mathbb{B}$.
Then $\partial _t(e^{- \im \Delta  t}\eta )=-\im e^{-\im \Delta t}(\eta + \mathbb{B})$ and so
\begin{equation*}  \begin{aligned} &   e^{-\im \Delta t_2}\eta (t_2) -e^{-\im \Delta t_1}\eta (t_1) =-\im  \int _{t_1} ^{t_2}  e^{-\im \Delta t}(V \eta  (t) +\mathbb{B} (t)) dt  \text{  for $t_1<t_2$}.\end{aligned}
\end{equation*}
Then   for a fixed $c_2$ by the Strichartz estimates
\begin{equation*}  \begin{aligned} &  \|  e^{-\im \Delta  t_2}\eta (t_2) -e^{-\im \Delta  t_1}\eta (t_1) \| _{H^1}\le  c_2 ( \| \eta \| _{L^2 (\R _+, W ^{1,6})}    +  \|   \mathbb{B} (t)   \|  _{L^1([{t_1} ,{t_2}], H^1) + L^2([{t_1} ,{t_2}], W ^{\frac{6}{5}}))}).\end{aligned}
\end{equation*}
Since we have
\begin{equation*}  \begin{aligned} &  \mathbb{B} =  \sum _{(\mu , \nu )\in {M}  }    \overline{z} ^{\mu} {z}  ^{\nu} \overline{G} _{\mu \nu} +   \mathbb{A}  , \end{aligned}
\end{equation*}
 and by \eqref{L^2disbis} and  \eqref{eq:eq A},
  valid now in $[0,\infty)$,   for a fixed $C$ we have
\begin{equation*}  \begin{aligned} &    \|   \sum _{(\mu , \nu )\in {M}  }    \overline{z} ^{\mu} {z}  ^{\nu} \overline{G} _{\mu \nu}  \|  _{  L^2(\R _+, W ^{1,\frac{6}{5}}) } \le C \epsilon  \, , \,
 \| \mathbb{A} \|  _{ L^2 (\R _+,  W ^{1,\frac{6}{5}}) + L^1 (\R _+ , H^1) } \le C  \epsilon ^2,
  \end{aligned}
\end{equation*}     we conclude that there exists an $\eta _+\in   H ^{1} $ with
\begin{equation*}  \begin{aligned} & \lim _{t\to + \infty}    e^{ -\im \Delta t }\eta (t )
  = \eta _+   \text{ in $ H ^{1}$   with }  \|    \eta (t ) -  e^{  \im \Delta t }\eta _+ \| _{H^1} \le C \epsilon   \text{  for all $t\ge 0$ }  .\end{aligned}
\end{equation*}
 So we have the first limit in \eqref{eq:small en31} and  the inequality $\|  \eta _+\| _{H^1}\le C    \| u (0)\| _{H^1}$ in Theorem \ref{thm:mainbounds}.

\noindent  We prove now  the existence of $z_+$  and the facts about it in  Theorem \ref{thm:mainbounds}. First of all, from \eqref{eq:FGR01}  \begin{equation*}  \begin{aligned} &
 \frac{1}{2}  \sum _j \frac{d}{dt} | z _j|^2 =
 \sum _j \Im \left [
\partial _{   \overline{j}}   \mathcal{R} \overline{z} _j   +     \sum _{ (\mu , \nu )\in M }    \nu _j   z  ^{\mu }
 \overline{  z }^ {  \nu   }
\langle \eta  ,
  {G}_{\mu \nu }  \rangle      +  \sum _{(\mu ', \nu' )\in M}    \mu _j '   z  ^{\nu '}
 \overline{  z  }^ {  \mu  '    }
\langle \overline{\eta}  ,
  \overline{{G}}_{\mu' \nu' }  \rangle  \right ]      .
\end{aligned}  \end{equation*}
Since the r.h.s. has $L^1(0,\infty )$ norm  bounded by $C \epsilon ^2$ for a fixed $C$,
we conclude that   the limit
\begin{equation*} \begin{aligned} &
\lim _{t\to + \infty} ( | z _1(t)|  , ...   , | z _n(t)|)=  ( \rho _{ +1}  ,  ...   , \rho _{ +n}) \text{ exists with  $| \rho  _+   | \le C  \| u (0)\| _{H^1}   $.}
\end{aligned}  \end{equation*}
By $\lim _{t\to  + \infty} \mathbf{Z}(t)=0$ we conclude that all but at most one
of the $\rho _{+j}$  are equal to 0.

\section{Proof   of Theorem  \ref{thm:orbstab}} \label{sec:stab}

 The stability of   $e^{-\im t E _{1z}}Q_{1z}$ is known.
By   Theorem 1 \cite{GSS} the stability of  $e^{-\im t E _{1z}}Q_{1z}$, or  equivalently of   $e^{-\im t E _{1\rho _1}}Q_{1\rho }$ for $\rho >0$,
   is a consequence of the following two
points.

\begin{itemize}
\item[(1)]  The  self--adjoint operator $L_{- \rho }:= H - E _{1\rho }+   |Q _{1 \rho } |^{2} $
has kernel $\ker L_{- \rho } =\{ Q _{1 \rho } \}$ and  $L_{- \rho } >0$  in $\{ Q _{1 \rho } \} ^{\perp}$.

\item[(2)] The   self--adjoint operator $L_{+ \rho }= H - E _{1\rho }+ 3 |Q _{1 \rho } | ^{2} $ is strictly positive: $L_{+ \rho }>0$.

\end{itemize}

\noindent If $|Q _{1 \rho }(x) |>0$ $\forall$ $x$, then (2) is an immediate consequence of (1).  The fact that
$\ker L_{- \rho } =\{ Q _{1 \rho } \}$  follows
by the fact that $Q _{1 \rho} \in \ker L_{- \rho }$ and by the fact  that   for  $|\rho |<\epsilon _0$   with $ \epsilon _0>0$ small, the number $E _{1\rho }\sim e_1$
is the smallest eigenvalue of $H  +   |Q _{1 \rho } |^{2}$ since $e_1$ is the smallest eigenvalue of $H   $.

\bigskip

 We recall that \cite{TY2,TY3,TY4,SW4,GW1,GW2,GP,NPT} give partial  proofs
of the instability of the 2nd excited state, and only for $2e_2>e_{1}$. We now prove the instability of the excited states.

\noindent Fix a $j>1$ and assume   that $Q_{j r}$ is orbitally stable.
Then     $Q_{j r}$ is asymptotically stable by Theorem \ref{thm:small en}.
So if $\|u(0)-Q_{j r}\|_{H^1}\ll 1$  then $\|u(t)-Q_{j z_j}-e^{\im \Delta t}\eta_+\|_{H^1} \to  0$    for $\ t\to \infty $ and $|z_j(t)|\to \rho$ with $\rho \neq 0$ and close to $r$.
  In this case  we have
\begin{align*}
E(u(0))&=\lim_{t\to \infty}E(u(t))=\lim_{t\to\infty}E(Q_{j z_j(t)}+e^{\im \Delta t}\eta_+),\\
\|u(0)\|_{L^2}^2&=\lim_{t\to \infty}\|Q_{j z_j(t)}+e^{\im \Delta t}\eta_+\|_{L^2}^2.
\end{align*}
Since $\|e^{\im \Delta t}\eta_+\|_{L^2_tL^6_x}\lesssim \|\eta_+\|_{L^2}$  there exists $t_n\to\infty$ s.t. $\|e^{\im \Delta t_n}\eta\|_{L^6_x}\to 0$.
So,  since $\|e^{\im t_n \Delta}\eta_+\|_{L^4}\to 0$, $\int V|e^{\im t_n \Delta}\eta_+|^2\,dx\to 0$ and the cross terms in \eqref{eq:enexp1} disappear, we have
\begin{align*}
E(u(0))&=\lim_{n\to\infty}E(Q_{j z_j(t_n)}+e^{\im \Delta t_n}\eta_+)=E(Q_{j \rho})+\|\nabla \eta_+\|_{L^2}^2,\\
\|u(0)\|_{L^2}^2&=\lim_{n\to \infty}\|Q_{j z_j(t_n)}+e^{\im \Delta t_n}\eta_+\|_{L^2}^2=\|Q_{j \rho}\|_{L^2}^2+\|\eta_+\|_{L^2}^2.
\end{align*}
We claim that for  $j\geq 2$  we can construct a curve on $H^1$ with the following property.
\begin{lemma}\label{lem:curve}
For sufficiently small $\delta$, there exists a map $[0,\delta)\ni\varepsilon\mapsto \Psi(\varepsilon)\in H^1$ s.t.
\begin{itemize}
\item
$\Psi(0)=Q_{j r}$,
\item
$\|\Psi(\varepsilon)\|_{L^2}^2=\|Q_{j r}\|_{L^2}^2$,
\item
$E(\Psi(\varepsilon))<E(Q_{j r})$ if $\varepsilon>0$.
\end{itemize}
\end{lemma}

\noindent Before proving the lemma  we show that the assumption that $Q_{j r}$ is asymptotically stable and the existence of $\Psi$ lead to a  contradiction.

\begin{proof}[Proof of instability]
Since $\|Q_{j  r}\|_{L^2}^2 =r^2+O(r^6)$ by Proposition \ref{prop:bddst},    $\|Q_{j  r}\|_{L^2}^2  $
is strictly  increasing in  $r$ for $r$ small.  By  Proposition \ref{prop:bddst}
we have $E'(Q_{j r})= (e_j+O(r^2))Q'(Q_{j r})$. This implies that  $E (Q_{j r})$  is a   strictly  decreasing function of $r$.
Setting $u(0)=\Psi(\varepsilon)$, we have
\begin{align*}
\|Q_{j r}\|_{L^2}^2=\|\Psi(\varepsilon)\|_{L^2}^2=\|Q_{j \rho}\|_{L^2}^2+\|\eta _{+}\|_{L^2}^2.
\end{align*}
Therefore  we have $\|Q_{j r}\|_{L^2}^2\geq \|Q_{j \rho}\|_{L^2}^2$.
This implies $r\geq \rho$ and  so  $E(Q_{j \rho})\geq E(Q_{j r}) $.
But  looking at the energy  we get the following  contradiction which ends the proof
of Theorem \ref{thm:orbstab}:
\begin{align*}
E(Q_{j r})>E(\Psi(\varepsilon))=E(Q_{j \rho})+\|\nabla \eta _{+}\|_{L^2}^2\geq E(Q_{j \rho})\geq E(Q_{j r}).
\end{align*}
\end{proof}

We now construct the curve $\Psi$.

\begin{proof}[Proof of Lemma \ref{lem:curve}]
We set $
\Psi(\varepsilon)=\beta(\varepsilon)Q_{j,r}+\varepsilon \phi_1
$
and choose $\beta(\varepsilon)$ to make $\|\Psi(\varepsilon)\|_{L^2}^2=\|Q_{j r}\|_{L^2}^2$:
\begin{align*}
\|Q_{j r}\|_{L^2}^2 \beta^2 + 2\varepsilon \<Q_{j r},\phi_1\>\beta+\varepsilon^2-\|Q_{j r}\|_{L^2}^2=0.
\end{align*}
So, we have
\begin{align*}
\beta(\varepsilon)&=\frac{-\<Q_{j r},\phi_1\>\varepsilon+\sqrt{\<Q_{j r},\phi_1\>^2 \varepsilon^2 -\|Q_{j r}\|_{L^2}^2(\varepsilon^2-\|Q_{j r}\|_{L^2}^2)}}{\|Q_{j r}\|_{L^2}^2}=
\sqrt{1- g_1(r)\varepsilon^2}+g_2(r)\varepsilon,
\\
g_1(r)&:=\frac{1}{\|Q_{j r}\|_{L^2}^4}\(\|Q_{j r}\|_{L^2}^2-\<Q_{j r},\phi_1\>^2\)  = \frac{1}{\|Q_{j r}\|_{L^2}^4}\(\|Q_{j r}\|_{L^2}^2-\<q_{j r},\phi_1\>^2\)  ,\\
g_2(r)&:=-\frac{\<Q_{j r},\phi_1\>}{\|Q_{j r}\|_{L^2}^2}  =-\frac{\<q_{j r},\phi_1\>}{\|Q_{j r}\|_{L^2}^2}  ,\\
\end{align*}
We now show $E(\Psi(\varepsilon))<E(Q_{j,r})$ for $\varepsilon>0$.
It suffices to show $S_{E_{j r}}(\Psi(\varepsilon))<S_{E_{j r}}(Q_{j r})$, where
\begin{align*}
S_{E_{j r}}(u)=E(u)-E_{j r}\|u\|_{L^2}^2.
\end{align*}
Notice that we have $S_{E_{j r}}'(Q_{jr})=0$.
Therefore, setting $\gamma(\varepsilon)=\beta(\varepsilon)-1$, we have
\begin{align*}
&S_{E_{j r}}(\Psi(\varepsilon))=S_{E_{j r}}(Q_{jr}+\gamma(\varepsilon)Q_{j r}+\varepsilon\phi_1)\\&=S_{E_{j r}}(Q_{j r})+\frac 1 2 \<S_{E_{j r}}''(Q_{j r})\(\gamma(\varepsilon)Q_{j r}+\varepsilon\phi_1\)
,\gamma(\varepsilon)Q_{j r}+\varepsilon\phi_1\>
+o\(\|\gamma(\varepsilon)Q_{j r}+\varepsilon\phi_1\|_{H^1}^2\)
\end{align*}
If $g_2(r)=0$  we have $\gamma(\varepsilon) = O(\varepsilon^2 r^{-2})$ and  we conclude
\begin{align*}
&S_{E_{j r}}(\Psi(\varepsilon))=S_{E_{j r}}(Q_{j r}) + \varepsilon^2\<S_{E_{j r}}(Q_{j r})\phi_1,\phi_1\>+o(\varepsilon^2)\\&=S_{E_{j r}}(Q_{j r})+\varepsilon^2(e_1-e_j) + O(\varepsilon^2 r)+o(\varepsilon^2) <S_{E_{j r}}(Q_{j r}) .
\end{align*}
If $g_2(r)\neq0$  we have $\gamma(\varepsilon)=O(r \varepsilon)$ and
\begin{align*}
S_{E_{j r}}(\Psi(\varepsilon))= S_{E_{j r}}(Q_{j r})+\varepsilon^2(e_1-e_j) + O(r \varepsilon^2) < S_{E_{j r}}(Q_{j r}).
\end{align*}
Therefore Lemma \ref{lem:curve} is proved. This also completes the proof of Theorem \ref{thm:orbstab}.

\end{proof}

\appendix
\section{Appendix:  a generalization of  Proposition \ref{prop:bddst} }
\label{app:bdstates}

For the reference purposes  we generalize \eqref{eq:NLS} as
\begin{equation}\label{NLSb}
 \im u_{t }=-\Delta  u +  V(x)u+\beta  (|u|^2)u  \ , \quad
 (t,x)\in\mathbb{ R}\times
 \mathbb{ R}^3.
\end{equation}
and assume that $\beta  (0)=0$, $\beta\in C^\infty(\R,\R)$ and further, there exists a
$p\in(1,5)$ such that for every $k\ge 0$ there is a fixed $C_k$ with $$\left|
\frac{d^k}{dv^k}\beta(v^2)\right|\le C_k |v|^{p-k-1} \quad\text{if $|v|\ge
1$}.$$

\begin{proposition} \label{prop:bddst1}
Fix $j\in \{ 1,\cdots,n\}$. Then
$\exists a _0>0$ s.t.\ $\forall z_j \in     B_\C (0,   a _0 )$     there is a unique  $  Q_{jz_j } \in \mathcal{S}(\R^3, \C ) := \cap _{t\ge 0}\Sigma _t (\R^3, \C )$ s.t.
\begin{equation}\label{eq:sp1}
\begin{aligned}
&(-\Delta + V) Q_{jz_j } + \beta (|Q_{jz_j }|^2) Q_{jz_j }= E_{jz_j }Q_{jz_j },\\&
Q_{jz_j }=  z_j\phi_j + q_{j z_j}, \ \langle q_{j z_j},\overline{\phi}_j\rangle =0,
\end{aligned}
\end{equation}
and s.t. we have for any $r\in \N$:
 \begin{itemize}
\item[(1)]    $(q_{jz_j },E_{jz_j }) \in C^\infty ( B_\C (   a _0 ), \Sigma _r\times \R )$;   we have $q_{jz_j } =  z_j \widehat{q}_{j  } (|z_j|^2)$ , with
$ \widehat{q}_{j  } (t^2 ) =t ^2\widetilde{q}_{j }(t^2)$,   $\widetilde{q}_{j } (t ) \in C^\infty (   ( - {a _0 }^{2}, {a _0 }^{2}), \Sigma _r (\R ^3, \R ) )$  and  $E_{jz_j }  =E_{j } (|z_j|^2)$ with $E_{j } (t ) \in C^\infty (   ( - {a _0 }^{2}, {a _0 }^{2}),   \R  )$;
\item[(2)]    $\exists$ $C >0$ s.t.
$\|q_{jz_j }\|_{\Sigma _r} \leq C |z_j|^3$, $|E_{jz_j }-e_j|<C | z_j|^2$.

\end{itemize}

\end{proposition}
The rest of this section is devoted to the proof of Prop. \ref{prop:bddst1}.

The first step is the following lemma, which follows by
a direct computation.

\begin{lemma}\label{lem:commutator}
Let $m\in \N_0$ and $k\in \{1,2,3\}$.
Then, we have
\begin{equation}\label{eq:commutator}
\begin{aligned}
&[-\Delta, |x|^{2m}] =  -2m(2m+1)|x|^{2m-2}-4m|x|^{2m-2}x\cdot \nabla \\
&[-\Delta, |x|^{2m}x_k]  = -2m(2m+3)|x|^{2m-2}x_k - 4m x_k |x|^{2m-2} x\cdot \nabla - 2 |x|^{2m}\partial_{x_k}
\end{aligned}
\end{equation}

\end{lemma}  \qed

Our second step is the following lemma.
\begin{lemma}\label{lemA:1} The eigenfunctions $\phi_j$ of $-\Delta + V$ satisfy
$\phi_j\in \mathcal S (\R^3)$.
\end{lemma}

\begin{proof}
First, $\phi_j\in L^2(\R^3)$, so  we have $\phi_j\in H^2(\R^3)$ by
$$
(-\Delta-e_j)\phi_j=-V\phi_j.
$$
Furthermore, if we have $\phi_j\in H^{2m}(\R^3)$, then we have $\phi_j\in H^{2m+2}(\R^3)$.
This implies $\phi_j\in \cap_{m=1}^\infty H^m$.

\noindent
Next, by Lemma \ref{lem:commutator}, we have
$$
(-\Delta - e_j) x_k\phi_j = -2\partial_{x_k}\phi_j - V x_k \phi_j,
$$
for $k=1,2,3$.
Therefore, we have $x_k\phi_j\in H^2(\R^3)$.
Again, by Lemma \ref{lem:commutator}, we have
$$
(-\Delta - e_j) |x|^2 \phi_j = -6\phi_j-4x\cdot\nabla \phi_j - V x_k \phi_j.
$$
So, by $x\cdot \nabla \phi_j=\nabla (x\phi_j)-3\phi_j\in L^2(\R^3)$, we have $|x|^2\phi_j\in H^2$.

\noindent
Now, suppose $|x|^{2m} \phi_j\in H^2(\R^3)$.
By Lemma \ref{lem:commutator}, we have
\begin{equation*}
\begin{aligned}
(-\Delta - e_j) |x|^{2m}x_k \phi_j =& -2m(2m+3)|x|^{2m-2}x_k\phi_j - 4mx_k|x|^{2m-2}x\cdot \nabla \phi_j \\&
-2|x|^{2m}\partial_{x_k}\phi_j-V|x|^{2m}x_k\phi_j.
\end{aligned}
\end{equation*}
Since
$$
|x|^{2m}\partial_{x_k}\phi_j=\partial_{x_k}\(|x|^{2m}\phi_j\)-4m|x|^{2m-2}x_k\phi_j\in L^2(\R^3),
$$
we have $ |x|^{2m}x_k \phi_j\in H^2(\R^3)$.
Finally, by
\begin{equation*}
\begin{aligned}
(-\Delta - e_j) |x|^{2m+2}\phi_j =& -2(m+1)(2m+3)|x|^{2m}\phi_j - 4(m+1)|x|^{2m}x\cdot \nabla \phi_j
-V|x|^{2m+2}\phi_j,
\end{aligned}
\end{equation*}
and $|x|^{2m}x\cdot \nabla \phi_j = \nabla \cdot (|x|^{2m}x\phi_j)-(4m+3)|x|^{2m}\phi_j\in L^2(\R^3)$, we have $|x|^{2m+2}\phi_j\in H^2(\R^2)$.
By induction we have $\phi_j\in \Sigma _{2m}$ for any $m\geq 1$.
\end{proof}

The next step is the following lemma.

\begin{lemma}\label{lem:bddst}
Fix $j\in \{ 1,\cdots,n\}$  and $r\in \N  $  with $r\ge 2$.    Then
$\exists \delta _{r}>0$ s.t.\ $\forall z_j \in     B_\C (  0, \delta _{r})$  there is a unique  $  Q_{jz_j } \in\Sigma _{r}(\R^3, \C )$  satisfying  \eqref{eq:sp}
and  claims (1) and (2) in Prop. \ref{prop:bddst}.

\end{lemma}

\begin{proof}In this proof we write  $g(u):=\beta(|u|^2)u$.
Notice that it suffices to show the claim of Lemma \ref{lem:bddst} for $z_j\in\R$ with $\R$ valued $Q_{j,z_j}$.
Indeed, if we define
\begin{equation}\label{eq:gauge}
Q_{j z_j}=e^{\im \theta}Q_{j \rho},\quad E_{j z_j}=E_{j \rho}
\end{equation}
for $z=e^{\im \theta}\rho$, $Q_{j z}$ and $E_{j z}$ satisfies \eqref{eq:sp} if $Q_{j \rho}$ and $E_{j \rho}$ satisfy \eqref{eq:sp}.
Further, if $B_{\R}(0,\delta)\ni z\mapsto (Q_{j z}, E_{j z})\in \Sigma _r\times \R$ is $C^\infty$, then by \eqref{eq:gauge}, we have
$B_{\C}(0,\delta)\ni z\mapsto (Q_{j,z}, E_{j,z})\in \Sigma _r\times \R$ is $C^\infty$.

\noindent
Fix $j\in \{0,1,\cdots,n\}$.
For simplicity   we set $z_j=z$, $e_j=e$ and $\phi_j=\phi$.
Set
$$
Q_{j,z}=z\(\phi + |z|^2\psi(z)\),\quad E_{j,z}=e + |z|^2f(z).
$$
We   solve \eqref{eq:sp} under the above ansatz.
Substituting the ansatz into \eqref{eq:sp}, we have
\begin{equation}
H \psi + z^{-3}g(z(\phi+z^2\psi)) = e\psi + f\phi + z^2 f \psi.
\end{equation}
Set $Pu=u- \langle u,\phi \rangle \phi$.
Then, we have
\begin{equation*}
\begin{aligned}
H\psi + z^{-3}P g(z(\phi+z^2\psi))&=e\psi + z^2 f\psi \, , \quad
\langle z^{-3}g(z(\phi+z^2\psi)),\phi \rangle =f.
\end{aligned}
\end{equation*}
Therefore, it suffices to solve
\begin{equation}\label{eq:sp2}
(H-e)\psi = -z^{-3}P g(z(\phi+z^2\psi))+z^{-1}\langle g(z(\phi+z^2\psi)),\phi \rangle \psi.
\end{equation}
Now, set $\tilde \phi (z):= \phi + z^2 \psi (z)$.
Then,
\begin{equation*}
\begin{aligned}
g(z\tilde \phi)&= \beta(z^2\tilde \phi) z\tilde \phi
= z^3 \int_0^1 \beta'(s z^2 \tilde \phi ^2) \,ds \tilde \phi ^3.
\end{aligned}
\end{equation*}

\noindent
So, \eqref{eq:sp2} can be rewritten as
\begin{equation}\label{eq:sp3}
(H-e)\psi = - P \(\int_0^1 \beta'(sz^2\tilde \phi ^2)\,ds \tilde \phi ^3\) + \langle \beta(z^2\tilde \phi ^2)\tilde \phi, \phi \rangle \psi.
\end{equation}

\noindent
To  show that $z\mapsto \psi(z)\in \Sigma _ r$ exists and  is $  C^\infty$, we use the inverse function theorem.
Set
\begin{equation*}
\Phi(z, \psi):=-(H-e)^{-1}P \(\int_0^1 \beta'(sz^2\tilde \phi ^2)\,ds \tilde \phi ^3\) + \langle \beta(z^2\tilde \phi ^2)\tilde \phi, \phi \rangle (H-e)^{-1}\psi,
\end{equation*}
and
\begin{equation*}
F(z,\psi):=\psi-\Phi(z,\psi).
\end{equation*}
Then, $F:\R\times P \Sigma _r\to P\Sigma _r$ is $C^\infty$.
Next, since
\begin{equation*}
F(0,\psi)=\psi+\beta'(0)(H-e)^{-1}P\phi^3,
\end{equation*}
we have
\begin{equation*}
F(0,-\beta'(0)(H-e)^{-1}P\phi^3)=0.
\end{equation*}
We now compute $F_\psi(z,\psi)$.
\begin{equation*}
\begin{aligned}
\Phi_{\psi}(z,\psi)h=&-2z^4(H-e)^{-1}P \(\int_0^1 \beta''(sz^2\tilde \phi ^2)s\,ds \tilde \phi ^4 h\)-3z^2(H-e)^{-1}P \(\int_0^1 \beta'(sz^2\tilde \phi ^2)\,ds \tilde \phi ^2 h\)\\&
+ 2z^4 \langle  \beta'(z^2\tilde \phi ^2)\tilde \phi^2 h, \phi \rangle (H-e)\psi
+ z^2\langle \beta(z^2\tilde \phi ^2)h, \phi \rangle (H-e)\psi
+ \langle \beta(z^2\tilde \phi ^2)\tilde \phi, \phi \rangle(H-e)h.
\end{aligned}
\end{equation*}
So, we have
\begin{equation*}
F_{\psi}(0,\psi)h=h.
\end{equation*}
Therefore, by the inverse function theorem we have the conclusion of the Lemma.
\end{proof}

The final step is that the $\delta _{r}>0$ can be chosen independent of $r$.
\begin{lemma}\label{lemA:2} Consider the   $Q_{j z_j}$ in  {Lemma} \ref{lem:bddst}.
Then $\exists$
a $\delta >0$  s.t.    $Q_{j z_j}\in \mathcal S (\R^3)$ for $|z_j|<\delta $.
\end{lemma}

\begin{proof}
We use a bootstrap argument similar to the proof of Lemma \ref{lemA:1}.
We can consider the $Q_{j z}$   given in Lemma \ref{lem:bddst} with $r=4$.
It is enough to consider $z= \rho \in (0,\delta)$  with $\delta <\delta _4$. For $\delta>0$ sufficiently small
we also have $E_{j\rho}<\frac 1 2 e_j <0 $.
By \eqref{eq:sp1} we have
\begin{equation}\label{eq:boot1}
(-\Delta - E_{j\rho})Q_{j\rho}=- VQ_{j\rho} -    \int_0^1 \beta'(s Q_{j\rho} ^2)\,ds  Q_{j\rho}^3.
\end{equation}

\noindent
We proceed as in  Lemma \ref{lemA:1}.
Since the commutator term and $-VQ_{j\rho}$ are the same as in \ref{lemA:1},c
we conclude that
 Lemma \ref{lemA:2}  is a consequence of the following two simple  facts for $m\ge 2$.
 \begin{itemize}
\item [(i)]
If $Q_{j\rho}\in H^m$, then $\beta(Q_{j\rho}^2)Q_{j\rho}=\int_0^1 \beta'(s Q_{j\rho}^2)\,ds Q_{j\rho}^3\in H^m$.
\item [(ii)]
If $|x|^{2m}Q_{j\rho} \in L^2(\R^3)$, then $|x|^{2m+2}\int_0^1 \beta'(s Q_{j\rho}^2)\,ds Q_{j\rho}^3\in L^2 $.
\end{itemize}
 (i)  follows from the fact that
  $H^m (\R ^3)$ is a ring for $m\geq 2$.
We now look at (ii). Since $Q_{j\rho}$ is a continuous function with $ Q_{j\rho}(x) \to 0$ as $|x|\to \infty$, the range of $Q_{j\rho} $ (i.e. $\{ Q_{j\rho} (x)\in\R\ |\ x\in\R^3\}$) is relatively compact.
So, since $t\to \int_0^1\beta'(s t^2)\, ds$ is a  continuous function from $\R\to \R$, the range of $\int_0^1\beta'(s Q_{j\rho}^2)\, ds$ is  relatively compact.
Therefore, we have $\int_0^1\beta'(s Q_{j\rho}^2)\, ds\in L^\infty$.
On the other hand, by $Q_{j\rho}\in \Sigma _4$ we have $|x|Q_{j\rho}\in \Sigma _3\hookrightarrow L^\infty$.
Therefore, we have
$$|x|^{2m+2}\int_0^1 \beta'(s Q_{j\rho}^2)\,ds Q_{j\rho}^3=\int_0^1 \beta'(s Q_{j\rho}^2)\,ds\(|x|Q_{j\rho}\)^2|x|^{2m}Q_{j\rho}\in L^2(\R^3).
$$
  This proves (ii) and completes the proof of Lemma \ref{lemA:2}.
\end{proof}
Finally, {Proposition} \ref{prop:bddst1}  is a consequence of Lemmas
  \ref{lem:commutator}--\ref{lemA:2}.

\section{Appendix: expansions of gauge invariant functions}\label{app:B}

We prove here  \eqref{eq:fix2} and \eqref{eq:Tayl1}, which are  direct consequences of Lemmas \ref{lem:expandzzj} and \ref{lem:expandzzS}.

\begin{lemma}\label{lem:gaugesmooth}
Let    $a(z)\in C^\infty(B_{\C}(0,\delta),\R)$  and    $a(e^{\im\theta}z)=a(z)$ for any $\theta\in\R$.
Then  there exists $\alpha\in \C^\infty([0,\delta^2);\R)$ s.t. $\alpha(|z|^2)=a(z)$.
\end{lemma}

\proof
For $z=re^{\im \theta} $ we have $a(z)=a(r +\im 0) $.
Since $x\to a(x +\im 0 )$ is even and smooth,  we have $a(x +\im 0 ) = \alpha  (x^2)$ with $\alpha (x )$ smooth, see \cite{Whitney}.   So $a(z)=\alpha  (|z|^2)$.
\qed

\begin{lemma}\label{lem:expandzz}
Let $\delta>0$.
Suppose $a\in C^\infty(B_{\C^n}(0,\delta);\R)$ satisfies $a(e^{\im\theta}z_1,\cdots,e^{\im\theta}z_n)=a(z_1, \cdots, z_n)$ for all $\theta\in\R$ and $a(0,\cdots,0)=0$.
Then, for any $M>0$, there exists $b_{\mathbf{m}}$ s.t.
\begin{equation}\label{eq:lemexpandzz}
a(z_1,\cdots,z_n)=\sum_{j=1}^n\alpha_j(|z_j|^2)+\sum_{|m|=1}\mathbf Z^{\mathbf m}b_{\mathbf m}(z_1,\cdots,z_n)+\mathcal R^{0,M}(z,\mathbf Z),
\end{equation}
where $\alpha_j(|z_j|^2)=a(0,\cdots,0,z_j,0,\cdots,0)$.
Furthermore, $b_{\mathbf m}\in C^\infty(B_{\C^{ n}}(0,\delta);\R)$ and satisfies\\ $b_{\mathbf m}(e^{\im\theta}z_1,\cdots,e^{\im\theta}z_n)=b_{\mathbf m}(z_1,\cdots,z_n)$ for all $\theta\in\R$.
\end{lemma}

\proof
First, we expand $a$ as
\begin{equation*}
a(z_1,\cdots,z_n)=a(z_1,0,\cdots,0)+\int_0^1\big (\sum_{j=2}^n\partial_j a(z_1,tz_2\cdots,t z_n)z_j+\partial_{\overline j} a(z_1,tz_2\cdots,t z_n)\overline z_j\big )\,dt.
\end{equation*}
Then, by
\begin{equation*}
\begin{aligned}
&a(0,z_2,\cdots, z_n)=\int_0^1\big (\sum_{j=2}^n\partial_j a(0,tz_2\cdots,t z_n)z_j+\partial_{\overline j} a(0,tz_2\cdots,t z_n)\overline z_j\big )\,dt,
\end{aligned}
\end{equation*}
we have
\begin{equation*}
\begin{aligned}
&a(z_1,\cdots,z_n)=a(z_1,0,\cdots,0)+a(0,z_2,\cdots, z_n) \\& +
\int_0^1 \sum_{j=2}^n \Big [ \(\partial_j a(z_1,tz_2\cdots,t z_n)-\partial_j a(0,tz_2\cdots,t z_n)\)z_j
 \\& \ \ \
+\(\partial_{\overline j} a(z_1,tz_2\cdots,t z_n)-\partial_{\overline j} a(0,tz_2\cdots,t z_n)\)\overline z_j\Big ] \,dt
=a(z_1,0,\cdots,0)+a(0,z_2,\cdots, z_n)\\&+
\sum_{j\geq2}\int_0^1\int_0^1 \Big [ \(\partial_1\partial_{ j} a(sz_1,tz_2\cdots,t z_n)\)z_1z_j +
 \(\partial_{\overline 1}\partial_{ j} a(sz_1,tz_2\cdots,t z_n)\)\overline z_1z_j \\&  \qquad +
 \(\partial_{ 1}\partial_{\overline j} a(sz_1,tz_2\cdots,t z_n)\)\overline z_1z_j +
 \(\partial_{\overline 1}\partial_{\overline j} a(sz_1,tz_2\cdots,t z_n)\) \overline z_1\overline z_j\Big ]dsdt.
\end{aligned}
\end{equation*}
Iterating this argument first for  $a(0,z_2,\cdots, z_n)$ and then for $a(0,...,0, z_{k},\cdots, z_n)$, we have
\begin{equation}\label{eq:lemexpandzz1}
\begin{aligned}
 & a(z_1,\cdots,z_n) = a(z_1,0,\cdots,0)+a(0,z_2,0,\cdots,0)+\cdots +a(0,\cdots,0,z_n)\\&+
\sum_{k=1}^{n-1}\sum_{j\geq k+1} \int_0^1\int_0^1\left [ \(  \partial_k\partial_{ j} a(0,\cdots,0,sz_k,tz_{k+1}\cdots,t z_n)\)z_kz_j \right.    \\&  +
 \(\partial_{\overline k}\partial_{ j} a(0,\cdots,0,sz_k,tz_{k+1}\cdots,t z_n)\)\overline z_kz_j  +
 \(\partial_{ k}\partial_{\overline j} a(0,\cdots,0,sz_k,tz_{k+1}\cdots,t z_n)\) z_k\overline{z_j} \\&\quad\left.+
 \(\partial_{\overline k}\partial_{\overline j} a(0,\cdots,0,sz_k,tz_{k+1}\cdots,t z_n)\) \overline z_k\overline z_j \right ]  \, dsdt .
\end{aligned}
\end{equation}
By Lemma \ref{lem:gaugesmooth}, there exist smooth $\alpha_j$ s.t. $\alpha_j(|z_j|^2)=a(0,\cdots,0,z_j,0,\cdots,0)$.
Furthermore,  the 3rd   line  of \eqref{eq:lemexpandzz1} has the same form as the 2nd term in the r.h.s.\ of \eqref{eq:lemexpandzz}.
So, it remains to handle the terms in the 2nd and 4th lines of \eqref{eq:lemexpandzz1}.
Since they can be treated similarly, we focus only  the  2nd line of \eqref{eq:lemexpandzz1}.
Set
$$
\beta_{jk}(z_k,\cdots, z_n)=\int_0^1\int_0^1\(\partial_k\partial_{ j} a(0,\cdots,0,sz_k,tz_{k+1}\cdots,t z_n)\)\,dsdt,
$$
with $j\geq k+1$.
Notice that $\partial ^\alpha  \overline{\partial} ^\beta  a(0,...,0)\neq 0 $  by the gauge invariance of $a$ is easily shown to imply $|\alpha | =|\beta|$. This in particular
implies $\beta_{jk}(0,\cdots,0)=0$.
So  as in \eqref{eq:lemexpandzz1}  we have
 \begin{align}
&\beta_{jk}(z_k,\cdots,z_n)=\beta_{jk}(z_k,0,\cdots,0)+\beta_{jk}(0,z_{k+1},0,\cdots,0)+\cdots +\beta_{jk}(0,\cdots,0,z_n)\nonumber \\&+
\sum_{m=k}^{n-1}\sum_{l\geq m+1}\int_0^1\int_0^1 \left [ \(\partial_m\partial_{ l} \beta_{jk}(0,\cdots,0,sz_m,tz_{m+1}\cdots,t z_n)\) z_mz_l \right. \label{eq:appB1}\\&\quad+
 \(\partial_{\overline{m}}\partial_{ l} \beta_{jk}(0,\cdots,0,sz_m,tz_{m+1}\cdots,t z_n)\)\overline z_mz_l  +
 \(\partial_m\partial_{\overline l} \beta_{jk}(0,\cdots,0,sz_m,tz_{m+1}\cdots,t z_n)\)z_m\overline z_l \nonumber\\&\quad\left.+
 \(\partial_{\overline m}\partial_{\overline l} \beta_{jk}(0,\cdots,0,sz_m,tz_{m+1}\cdots,t z_n)\)\overline z_m\overline z_l\right ]   \,dsdt. \nonumber
\end{align}
Since $z_l^2\beta_{jk}(0,\cdots,0,z_l,0,\cdots,0)$ is gauge invariant  by Lemma \ref{lem:gaugesmooth}  we have
$$
z_l^2\beta_{jk}(0,\cdots,0,z_l,0,\cdots,0)=\tilde \beta_{jkl}(|z_l|^2)=\tilde \beta_{jkl}(0)+\tilde \beta_{jkl}'(0)|z_l|^2+\gamma_{jkl}(|z_l|^2)|z_l|^4,
$$
for some smooth $\tilde \beta_{jkl}$ and $\gamma_{jkl}$.
By the smoothness of $\beta_{jk}$, we have $\tilde \beta_{jkl}(0)=\tilde \beta_{jkl}'(0)=0$.
Therefore,
$$
\beta_{jk}(0,\cdots,0,z_l,0,\cdots,0)z_kz_j=\gamma_{jkl}(|z_l|^2)  z_kz_j\overline{z}_l^2 \text{ with $k< \min \{ j,l \}$}.
$$
This can be absorbed in the 2nd term of the r.h.s.\ of \eqref{eq:lemexpandzz}.
 The same is true of the contribution of the last 2 lines of \eqref{eq:appB1}. The term
\begin{equation}\label{eq:appB3}
 \int_0^1\int_0^1\(\partial_m\partial_{ l} \beta_{jk}(0,\cdots,0,sz_m ,tz_{m+1}\cdots,t  z_n)\)z_mz_lz_jz_k\,dsdt
\end{equation}
does not have as factors components of $ {\mathbf{Z}}=(z_i  \overline{z}_j) _{i\neq j }$ but it is $O(|\mathbf Z|^2)$.
Treating  \eqref{eq:appB3} the way we treated  the 2nd line of \eqref{eq:lemexpandzz1}.
and repeating the procedure a sufficient number of times, we can express  \eqref{eq:appB3}
as a sum of a summation like the 2nd in the r.h.s. of \eqref{eq:lemexpandzz}   and of a
 $O(|\mathbf Z|^M)$ for arbitrary $M$. Furthermore, notice that since we can think of the dependence on $ {\mathbf{Z}}=(z_i  \overline{z}_j) _{i\neq j }$ to be polynomial, and so  the remainder term $R^{0,M}(z,\mathbf Z)$ in  \eqref{eq:lemexpandzz} can be thought to depend polynomially  on $ {\mathbf{Z}}=(z_i  \overline{z}_j) _{i\neq j }$, it can be thought as the
 restriction of a function in $ {\mathbf{Z}}\in  L$.
\qed

\begin{lemma}\label{lem:expandzzj} Take  $a(z_1,\cdots,z_n)$ like in Lemma \ref{lem:expandzz}.
Then, for any $M>0$, there exist smooth $a_j$ and $b_{j\mathbf{m}}$ s.t. for  $\alpha_j(|z_j|^2)=a(0,\cdots,0,z_j,0,\cdots,0)$ we have
\begin{equation}\label{eq:lemexpandzzj}
a(z_1,\cdots,z_n)=\sum_{j=1}^n\alpha_j(|z_j|^2)+\sum_{1\le |\mathbf{m}|\leq M-1}\mathbf Z^{\mathbf m}b_{j\mathbf m}(|z_j|^2)+\mathcal R^{0,M}(z,\mathbf Z).
\end{equation}
\end{lemma}

\proof
To prove \eqref{eq:lemexpandzzj}  one only has to repeatedly use Lemma \ref{lem:expandzz}.
\qed

\begin{lemma}\label{lem:expandzzS}
Suppose that $a:\C^n\to \mathcal S$ is smooth from $B_{\R^{2n}}(0,\delta_r)$ to $\Sigma_r$ for arbitrary $r\in \R$ and satisfies $a(e^{\im\theta}z_1,\cdots,e^{\im\theta}z_n)=a(z_1,\cdots,z_n)$, $a(0,\cdots,0)=0$.
Then, for any $M>0$, there exist smooth $a_j$ and $b_{j\mathbf{m}}$ s.t. for  $\alpha_j(|z_j|^2)=a(0,\cdots,0,z_j,0,\cdots,0)$ we have
\begin{equation}\label{eq:lemexpandzzS}
a(z_1,\cdots,z_n)=\sum_{j=1}^n\alpha_j(|z_j|^2)+\sum_{1\le |\mathbf{m}|\leq M-1}\mathbf Z^{\mathbf m}G_{j\mathbf m}(|z_j|^2)+\mathcal S^{0,M}(z,\mathbf Z).
\end{equation}
\end{lemma}

\proof
The proof is same as the proof of Lemmas \ref{lem:gaugesmooth}--\ref{lem:expandzzj}
\qed

\section*{Acknowledgments}   S.C. was partially funded      grants FIRB 2012 (Dinamiche Dispersive) from the Italian Government  and   FRA 2013 from the University of Trieste.
M.M. was supported by the Japan Society for the Promotion of Science (JSPS) with the Grant-in-Aid for Young Scientists (B) 24740081.

\bibliographystyle{amsplain}

Department of Mathematics and Geosciences,  University
of Trieste, via Valerio  12/1  Trieste, 34127  Italy. {\it E-mail Address}: {\tt scuccagna@units.it}
\\

Department of Mathematics and Informatics,
Faculty of Science,
Chiba University,
Chiba 263-8522, Japan.
{\it E-mail Address}: {\tt maeda@math.s.chiba-u.ac.jp}

\end{document}